\documentclass{article}
\usepackage{amssymb}
\usepackage{algorithm}
\usepackage{algpseudocode}
\usepackage{float}
\usepackage{PRIMEarxiv}
\usepackage{graphicx}
\usepackage{booktabs}
\usepackage{amsmath,amssymb}
\usepackage{subcaption}
\usepackage[table]{xcolor}
\usepackage{multirow}
\usepackage{amsmath}
\usepackage{amsthm}
\newtheorem{assumption}{Assumption}
\newtheorem{remark}{Remark}
\newtheorem{lemma}{Lemma}
\newtheorem{theorem}{Theorem}

\usepackage[T1]{fontenc}    
\usepackage{url}            
\usepackage{booktabs}       
\usepackage{amsfonts}       
\usepackage{nicefrac}       
\usepackage{microtype}      
\usepackage{lipsum}
\usepackage{fancyhdr}       
\usepackage{graphicx}       
\graphicspath{{media/}}     
\usepackage{etoc}
\usepackage{color}

\newcommand{\R}{\mathbb R}

\newcommand{\xkh}[1]{\left(#1\right)}

\newcommand{\dkh}[1]{\left\{#1\right\}}

\newcommand{\norm}[1]{\left\|{#1}\right\|_2}

\newcommand{\Abs}[1]{\left\lvert#1\right\rvert}

\title{Higher-order Diffusion Sampling via Chebyshev Interpolation and Gauss--Seidel Iterations}

\author{
Bingyuan Wei\thanks{Computational Mathematics, Beihang University, Beijing, China. Email: BingyuanWei@buaa.edu.cn}
\and
Meng Huang \thanks{School of Mathematical Sciences, Beihang University, Beijing, 100191, China. Email: menghuang@buaa.edu.cn}
}

\begin{document}
\maketitle

\begin{abstract}
Higher-order ODE solvers have shown strong empirical promise for accelerating diffusion models through the probability flow ODE, but rigorous non-asymptotic guarantees for such acceleration remain limited. In this paper, we develop a Chebyshev--Gauss--Seidel higher-order sampler and establish a non-asymptotic convergence guarantee that allows the approximation order to grow logarithmically with the number of outer iterations. In the exact-score setting, up to logarithmic factors, the proposed sampler requires at most
\[
d^{1+o_T(1)}\varepsilon^{-1/K_1}
\]
score functions to approximate the target distribution on \(\mathbb{R}^d\) within total variation distance \(\varepsilon\), where \(o_T(1)\to 0\) as \(T\to\infty\) and \(K_1>0\) is a sufficiently large constant. The analysis assumes only a polynomial second-moment bound on the target distribution, thereby relaxing the bounded-support condition imposed in existing higher-order theory. Moreover, the guarantee is robust to score and Jacobian estimation errors and does not require higher-order smoothness assumptions on the score estimates. Numerical experiments on anisotropic Gaussian mixture benchmarks support the predicted improvement in the accuracy--cost tradeoff under finite score-evaluation budgets.
\end{abstract}

\keywords{diffusion models, probability flow ODE, higher-order sampler, Chebyshev--Lobatto interpolation}

\section{Introduction}
\subsection{Diffusion models}
Diffusion models have emerged as a central paradigm in modern generative modeling, achieving strong empirical performance across image synthesis, text generation, audio generation, and related tasks \cite{sohl2015deep,song2019generative,ho2020ddpm,song2020ddim,song2021sde,dhariwal2021diffusion,rombach2022latent,saharia2022photorealistic}. Two of the most influential diffusion frameworks are denoising diffusion probabilistic models (DDPM) \cite{ho2020ddpm} and denoising diffusion implicit models (DDIM) \cite{song2020ddim}. These methods generate high-quality samples by approximately reversing a progressive noising process. In contrast to alternative generative paradigms, such as generative adversarial networks (GANs) \cite{goodfellow2014gan}, variational autoencoders (VAEs) \cite{kingma2014vae}, and normalizing flows \cite{rezende2015nf}, which typically permit sample generation in a single forward pass or with substantially fewer iterative steps, diffusion samplers require a sequence of reverse-time updates, each of which generally involves evaluating a pretrained neural network for denoising or score estimation.

More specifically, a diffusion model is built on two stochastic processes in
\(\mathbb{R}^d\). The first is a forward process
\[
X_0 \xrightarrow{\text{add noise}} X_1 \xrightarrow{\text{add noise}} \cdots \xrightarrow{\text{add noise}} X_T,
\]
which begins with a sample drawn from the target data distribution and
gradually transforms it into a noise-like distribution according to a
prescribed variance schedule \(\{\beta_t\}_{t=1}^T\); see, e.g.,
\cite{ho2020ddpm,song2021sde}. When \(T\) is sufficiently large, the
distribution of \(X_T\) is typically close to a standard Gaussian
distribution. The goal of diffusion generative modeling is then to construct a
reverse process
\[
Y_T \xrightarrow{\text{denoise}} Y_{T-1} \xrightarrow{\text{denoise}} \cdots \xrightarrow{\text{denoise}} Y_0,
\]
which starts from pure noise and progressively recovers a sample whose
distribution is close to that of the target data, ideally so that
\[
Y_t \stackrel{d}{\approx} X_t,\qquad t=T,\ldots,0.
\]

The reverse-time dynamics are determined by the score functions of the forward
marginals, thereby linking diffusion sampling to score-based generative
modeling and reverse-time diffusion theory
\cite{song2019generative,anderson1982reverse,haussmann1986time,hyvarinen2005score,vincent2011connection}.
Within this framework, DDPM is commonly viewed as a discretization of the
reverse-time stochastic dynamics \cite{ho2020ddpm}, whereas DDIM is closely
connected to the deterministic probability flow ODE
\cite{song2020ddim,song2021sde}. Both formulations require repeated evaluation
of pretrained score or denoising networks during sampling. Consequently, even
after training, generation can remain computationally expensive, making it a
central goal of accelerated diffusion sampling to reduce the number of score
function evaluations while maintaining sampling accuracy.

\subsection{Training-free acceleration and the higher-order theory gap}

Training-free acceleration has become a central approach to fast diffusion
sampling. Rather than introducing an additional distillation or
consistency-training stage, it keeps the pretrained score estimates fixed and
improves the accuracy--cost tradeoff through more effective discretizations of
the reverse-time dynamics. For probability flow ODE samplers, this naturally
leads to higher-order ODE discretizations. Methods such as
DPM-Solver~\cite{lu2022dpmsolver}, DEIS~\cite{zhang2022deis},
UniPC~\cite{zhao2023unipc}, and DPM-Solver++ for guided
sampling~\cite{lu2022dpmsolverpp} have demonstrated substantial empirical
speedups while maintaining high sample quality; see also
classifier-free guidance~\cite{ho2022cfg}. These empirical successes make
higher-order probability flow ODE sampling a compelling target for rigorous
non-asymptotic analysis.

The theoretical understanding of such acceleration, however, remains relatively
limited. For SDE-based diffusion sampling, polynomial-time guarantees under
weak assumptions were first established in
\cite{chen2022sampling,lee2022polynomial,lee2023general,benton2024nearly}. For
deterministic samplers based on the probability flow ODE,
\cite{chen2023pflow} gave the first provable convergence guarantee, and
\cite{li2024sharp} later established a sharp first-order benchmark. Since the
present work focuses on training-free higher-order acceleration, we summarize in
Table~\ref{tab:ode_theory_comparison} the deterministic probability flow ODE
guarantees most relevant to our setting, with all displayed iteration
complexities suppressing logarithmic factors. In particular, Runge--Kutta
analyses \cite{huang2024pflow} improve the dependence on \(\varepsilon\), but
require compact support and higher-order smoothness assumptions. Meanwhile,
\cite{li2025faster} proves acceleration under comparatively weak
distributional and score-estimation assumptions, but still treats the
approximation order \(K\) as a constant and assumes bounded support of the
target distribution. This suggests that the current complexity theory for
higher-order probability flow ODE sampling remains improvable under weak
assumptions, which is precisely the motivation for this paper.

\begin{table}[H]
\centering
\caption{Comparison of deterministic probability flow ODE sampling guarantees
in total variation. Here \(s_\tau\) and \(\nabla s_\tau\) denote the estimated score
function and its Jacobian, while \(s_\tau^\star\) and \(\nabla s_\tau^\star\) denote
their exact counterparts, $D$ represents the radius of the data support and $L$ bounds certain higher-order derivatives of the score estimates.}
\label{tab:ode_theory_comparison}
\small
\setlength{\tabcolsep}{4pt}
\renewcommand{\arraystretch}{1.20}
\begin{tabular}{lllll}
\toprule
paper
&
target assumption
&
score / regularity assumption
&
iteration complexity
&
higher-order solver
\\
\midrule

\cite{li2024sharp}
&
$\mathbb P(\|X_0\|_2\le T^{c_R})=1$
&
$s_\tau\approx s_\tau^\star,\ 
\nabla s_\tau\approx \nabla s_\tau^\star$
&
$\max\{d^2,d/\varepsilon\}$
&
$\times$
\\

\cite{huang2024pflow}
&
$\mathbb P(\|X_0\|_2\le D)=1$
&
$s_\tau\approx s_\tau^\star,\quad s_\tau\in C^{p+1}$
&
$(LDd)^{1+1/p}\varepsilon^{-1/p}$
&
$\checkmark$ $p$-th order
\\

\cite{li2025faster}
&
$\mathbb P(\|X_0\|_2\le T^{c_R})=1$
&
$s_\tau\approx s_\tau^\star,\ 
\nabla s_\tau\approx \nabla s_\tau^\star$
&
$
\max\{d^2,\,
d^{1+2/K}\varepsilon^{-1/K}\}
$
&
$\checkmark$ $K$-th order(fixed $K$)
\\

\rowcolor{gray!15}
this work
&
$\mathbb E\|X_0\|_2^2\le T^{2c_R}$
&
$s_\tau\approx s_\tau^\star,\ 
\nabla s_\tau\approx \nabla s_\tau^\star$
&
$
d^{1+o_T(1)}\varepsilon^{-1/K_1}
$
&
$\checkmark$ $K$-th order($K\asymp\log T$)
\\

\bottomrule
\end{tabular}
\end{table}

\subsection{Main contributions}

As noted above, existing complexity guarantees for higher-order probability flow
ODE sampling rely on two key restrictions: the interpolation order \(K\) on
each reverse-time interval is treated as a fixed constant, and the target
distribution is assumed to satisfy the bounded-support condition
\(
\mathbb P(\|X_0\|_2 \le T^{c_R}) = 1.
\)
Under these assumptions, the iteration complexity required to achieve
\(\varepsilon\)-accuracy in total variation scales as
\(
d^{1+2/K}\varepsilon^{-1/K}.
\)
In this paper, we show that this complexity can be further improved under the
substantially weaker moment condition
\(
\mathbb E\|X_0\|_2^2 \le T^{2c_R}.
\)
Specifically, by developing a Chebyshev--Gauss--Seidel higher-order sampler
within the probability flow ODE framework of \cite{li2025faster}, we obtain the
improved complexity bound
\(
d^{1+o_T(1)}\varepsilon^{-1/K_1},
\)
where \(K_1\) is a sufficiently large constant.

Our approach has two main ingredients. First, we replace the equi-spaced nodes
used in \cite{li2025faster} with Chebyshev--Lobatto nodes, which allow the
interpolation order \(K\) to grow as \(\log T\) while preserving control of
error propagation. Second, we replace the Jacobi iteration in
\cite{li2025faster} with a Gauss--Seidel-type refinement scheme to approximate
the probability flow ODE solution more effectively. As a result, the number of
score function evaluations required to achieve a prescribed total variation
accuracy is reduced to
\[
d^{1+o_T(1)}\varepsilon^{-1/K_1},
\]
up to logarithmic factors, where \(o_T(1)\to 0\) as \(T\to\infty\).

In addition, our theory is robust to inexact score estimation: it requires only
score and Jacobian accuracy along the iterates and does not impose
higher-order smoothness assumptions on the learned score. Finally, numerical
experiments on anisotropic Gaussian mixture benchmarks support the predicted
improvement in the accuracy--cost tradeoff.

\subsection{Other related work}

Beyond the probability flow ODE literature discussed above, non-asymptotic
convergence theory for score-based diffusion sampling has been developed for
DDPM-type samplers and reverse-SDE methods under increasingly general
assumptions on the data distribution and score estimation error. These works
establish guarantees in total variation, Kullback--Leibler divergence, and
Wasserstein distance
\cite{lee2022polynomial,lee2023general,chen2022sampling,chen2022userfriendly,benton2024nearly,bruno2025logconcave,conforti2025kl,gao2025generalw2}.

For probability flow ODE samplers, recent studies have established
polynomial-time convergence with a corrector step, nearly dimension-linear
convergence for discrete-time implementations, prediction--correction
guarantees, Wasserstein convergence, and adaptation to intrinsic
low-dimensional structure
\cite{chen2023pflow,li2024sharp,pedrotti2024prediction,gao2025pfloww2,tang2025adaptivity,beyler2025wasserstein}.
Provable acceleration has further been investigated through training-free
accelerated DDIM/DDPM samplers, higher-order Runge--Kutta discretizations of
the probability flow ODE, higher-order Lagrange interpolation with successive
refinement, operator-splitting samplers, and first-order forward-value
evaluation
\cite{li2024accelerating,huang2024pflow,li2025faster,liu2026operatorsplitting,jiao2025firstorder}.
Recent work has also established high-accuracy diffusion sampling guarantees
with polylogarithmic dependence on the inverse accuracy under a variety of
algorithmic and structural assumptions
\cite{gatmiry2026highaccuracy,chen2026highaccuracy}.

\subsection{Notations}
Throughout this paper, we write \(f = O(g)\) or \(f \lesssim g\) if
\(|f| \le C |g|\) for some universal constant \(C > 0\). We write
\(f \asymp g\) if both \(f \lesssim g\) and \(g \lesssim f\) hold.
The notation \(\widetilde O(\cdot)\) suppresses logarithmic factors in the
relevant parameters, and \(o_T(1)\) denotes a quantity that tends to zero as
\(T \to \infty\). For any two probability measures \(P\) and \(Q\), their
total variation distance is defined by
\(
\operatorname{TV}(P,Q) := \frac{1}{2}\int |\mathrm{d}P - \mathrm{d}Q|.
\)
In addition, \(p_X(\cdot)\) and \(p_{X \mid Y}(\cdot \mid \cdot)\) denote the
probability density functions of \(X\) and of \(X\) conditional on \(Y\),
respectively. For any matrix \(A\), we use \(\|A\|\) and
\(\|A\|_{\mathrm{F}}\) to denote its spectral norm and Frobenius norm,
respectively. Finally, for any vector-valued function \(f\), we let
\(\frac{\partial f}{\partial x}\) denote the Jacobian matrix of \(f\).

\subsection{Organization}
The remainder of this paper is organized as follows. Section~\ref{sec:Preliminary}
presents the preliminaries on diffusion models and the probability flow ODE,
and Section~\ref{sec:algo} introduces the proposed Chebyshev--Gauss--Seidel
higher-order sampler. Section~\ref{sec:main_result} states the main result, and
Section~\ref{sec:promainresult} contains its proof. Section~\ref{sec:numerical_study}
reports the numerical experiments, while Section~\ref{sec:discussion}
concludes the paper with further discussion. The appendix collects the
auxiliary lemmas used in the analysis, as well as the detailed proofs of the
lemmas used in Section~\ref{sec:promainresult}.

\section{Preliminaries}
\label{sec:Preliminary}
In this section, we review the basic framework of diffusion generative models based on (Stein) score functions. The goal of a generative model is to produce samples whose distribution is close to an unknown target distribution \(p_{\text{data}}\) on \(\mathbb{R}^d\), given access to data drawn from \(p_{\text{data}}\). A diffusion generative model typically involves two stochastic processes: a forward process and a reverse process, which we describe below.

\paragraph{Forward process.} 
Starting from an initial sample \(X_0\) drawn from the target distribution
\(p_{\text{data}}\) on \(\mathbb{R}^d\), the forward process is defined by
\begin{equation}
    X_t = \sqrt{\alpha_t}\, X_{t-1} + \sqrt{1-\alpha_t}\, W_t,
    \qquad t=1,\ldots,T,
    \label{forward_process}
\end{equation}
where \(0 < \alpha_t < 1\) are prescribed noise-scheduling parameters, and
\(\{W_t\}_{t=1}^T\) is a sequence of independent standard Gaussian random vectors
in \(\mathbb{R}^d\), that is,
\[
W_t \stackrel{\text{i.i.d.}}{\sim} \mathcal{N}(0,I_d).
\]
Define
\begin{equation}
    \bar{\alpha}_t := \prod_{k=1}^t \alpha_k,
    \qquad t=1,\ldots,T.
\end{equation}
Then, for each \(t=1,\ldots,T\), it is straightforward to verify that
\begin{equation}
    X_t = \sqrt{\bar{\alpha}_t}\, X_0 + \sqrt{1-\bar{\alpha}_t}\, \bar{W}_t,
    \qquad \text{where } \bar{W}_t \sim \mathcal{N}(0,I_d).
\end{equation}
In particular, when \(\bar{\alpha}_T\) is sufficiently small, the distribution of
\(X_T\) is close to \(\mathcal{N}(0,I_d)\) for a broad class of data distributions:
\begin{equation}
\operatorname{Law}(X_T) \approx \mathcal{N}(0,I_d).
\label{eq:X_T}
\end{equation}
Diffusion models are closely connected to the framework of stochastic
differential equations (SDEs), and it is therefore useful to introduce a
continuous-time version \((\bar{X}_\tau)_{\tau \in [0,1]}\) of the forward
diffusion process. Specifically, consider the SDE
\begin{equation}
    \mathrm{d}\bar{X}_\tau
    = -\frac{1}{2(1-\tau)} \bar{X}_\tau\, \mathrm{d}\tau
    + \frac{1}{\sqrt{1-\tau}}\, \mathrm{d}B_\tau,
    \qquad
    \bar{X}_0 \sim p_{\text{data}},
    \qquad 0 \le \tau < 1,
    \label{eq:forwSDE}
\end{equation}
where \(B_\tau\) is a standard Brownian motion in \(\mathbb{R}^d\). One can
verify that the solution to this SDE satisfies
\begin{equation}
    \bar{X}_\tau = \sqrt{1-\tau}\, X_0 + \sqrt{\tau}\, Z,
    \qquad
    X_0 \sim p_{\text{data}},
    \quad
    Z \sim \mathcal{N}(0,I_d),
    \label{def_1}
\end{equation}
where \(X_0\) and \(Z\) are independent. In particular,
\begin{equation}
    \bar{X}_{1-\bar{\alpha}_t} \stackrel{\mathrm{d}}{=} X_t,
    \qquad t=1,\ldots,T.
    \label{eq:relas}
\end{equation}

\paragraph{Reverse process.}
A central goal of diffusion models is to construct a time-reversed process whose
marginal distributions coincide with, or closely approximate, those of the
forward process. More precisely, we seek a reverse-time process
\[
Y_T \to Y_{T-1} \to \cdots \to Y_1
\]
such that
\[
Y_t \stackrel{\mathrm{d}}{\approx} X_t,
\qquad t=1,\ldots,T.
\]

Using the relation between \(X_t\) and \(\bar{X}_\tau\) in \eqref{eq:relas}, we
can design a discrete reverse-time process \(\{Y_t\}_{t=1}^T\) through the
time-reversal of the continuous-time process \(\bar{X}_\tau\). Specifically,
for the forward process \eqref{eq:forwSDE}, classical results on reverse-time
SDEs~\cite{song2021sde} show that the corresponding probability flow ODE is
given as follows: for any \(\tau_0 \in (0,1)\), the process
\(\{Y_\tau^{\mathrm{ode}}\}_{\tau \in [0,\tau_0)}\) satisfies
\begin{equation*}
\mathrm{d}Y_\tau^{\mathrm{ode}}
=
-\frac{1}{2(1-\tau)}
\left(
Y_\tau^{\mathrm{ode}}
+
\nabla \log p_{\bar{X}_\tau}\!\left(Y_\tau^{\mathrm{ode}}\right)
\right)\mathrm{d}\tau,
\qquad
Y_{\tau_0}^{\mathrm{ode}} \sim \bar{X}_{\tau_0},
\qquad
0 \le \tau < \tau_0,
\end{equation*}
or equivalently,
\begin{equation}
\mathrm{d}\!\left(\frac{Y_\tau^{\mathrm{ode}}}{\sqrt{1-\tau}}\right)
=
-\frac{1}{2(1-\tau)^{3/2}}
\nabla \log p_{\bar{X}_\tau}\!\left(Y_\tau^{\mathrm{ode}}\right)\mathrm{d}\tau,
\qquad
Y_{\tau_0}^{\mathrm{ode}} \sim \bar{X}_{\tau_0},
\qquad
0 \le \tau < \tau_0,
\label{ODE}
\end{equation}
and has the same marginal distributions as
\(\{\bar{X}_\tau\}_{\tau \in [0,\tau_0)}\).

Importantly, the reverse process depends only on the gradient of the log-density
\(\nabla \log p_{\bar{X}_\tau}\), known as the \emph{score function} of
\(p_{\bar{X}_\tau}\). For the continuous-time forward process
\(\{\bar{X}_\tau\}_{\tau \in [0,1]}\), the exact score function at time
\(\tau \in (0,1]\) is defined by
\begin{equation}
s_\tau^\star(x) := \nabla \log p_{\bar{X}_\tau}(x),
\qquad x \in \mathbb{R}^d.
\label{score_define}
\end{equation}

If these score functions were available, one could initialize
\(Y_T \sim \mathcal{N}(0,I_d)\) in view of \eqref{eq:X_T}, and then generate
samples by numerically solving the probability flow ODE \eqref{ODE}. In
practice, however, the true score functions are unknown and must be learned
from data, typically using neural networks trained on samples from the target
distribution~\cite{ho2020ddpm,song2020ddim}. Consequently, the practical
reverse process \(Y_T \to Y_{T-1} \to \cdots \to Y_1\) is obtained by
approximately solving \eqref{ODE} using estimated score functions, giving rise
to a broad class of diffusion sampling algorithms~\cite{chen2023pflow,li2024minimal,li2024sharp,li2025faster}.

\section{Gauss--Seidel Refinement with Higher-Order Diffusion Sampling} \label{sec:algo}
In this section, we present the main ideas and the detailed procedure of our
proposed algorithm. At a high level, the method is a Gauss--Seidel-type
higher-order iterative scheme for approximately solving the ODE \eqref{ODE}. To
motivate the construction, recall that
\(X_t \stackrel{\mathrm{d}}{=} \bar{X}_{1-\bar{\alpha}_t}\). According to the
probability flow ODE \eqref{ODE}, the ideal reverse update from \(Y_t\) to
\(Y_{t-1}\) is given by \(Y_{t-1} = Y_{1-\bar{\alpha}_{t-1}}^{\mathrm{ode}}\),
where \(Y_\tau^{\mathrm{ode}}\) denotes the solution to \eqref{ODE} with
terminal condition \(Y_{1-\bar{\alpha}_t}^{\mathrm{ode}} = Y_t\). In
particular, we have
\begin{equation}
    \frac{Y_{t-1}}{\sqrt{\bar{\alpha}_{t-1}}}
    =
    \frac{Y_t}{\sqrt{\bar{\alpha}_t}}
    -
    \int_{1-\bar{\alpha}_t}^{1-\bar{\alpha}_{t-1}}
    \frac{1}{2(1-\tau)^{3/2}}
    s_\tau^\star\!\left(Y_\tau^{\mathrm{ode}}\right)\,\mathrm{d}\tau .
    \label{ODE_score}
\end{equation}
Directly implementing this update is, however, computationally prohibitive,
since the integral depends on a continuum of score evaluations, whereas in
practice we only have access to a finite collection of estimated score
functions. To address this issue, on each interval
\(\bigl[1-\bar{\alpha}_t,\,1-\bar{\alpha}_{t-1}\bigr]\) we approximate
\(s_\tau^\star\!\left(Y_\tau^{\mathrm{ode}}\right)\) by polynomial
interpolation. In this paper, we employ a higher-order approximation based on
\(K\) score evaluations at Chebyshev--Lobatto nodes, where \(K \ge 2\).
Specifically, for the interval
\(\bigl[1-\bar{\alpha}_t,\,1-\bar{\alpha}_{t-1}\bigr]\), define
\begin{equation}
    \tau_{t,0} = 1-\bar{\alpha}_t,
    \qquad
    \tau_{t,K-1} = 1-\bar{\alpha}_{t-1},
\end{equation}
and choose the Chebyshev--Lobatto nodes by
\begin{equation}
    \tau_{t,j}
    =
    \frac{\tau_{t,0} + \tau_{t,K-1}}{2}
    +
    \frac{\tau_{t,0} - \tau_{t,K-1}}{2}
    \cos\!\left(\frac{j\pi}{K-1}\right),
    \qquad
    0 \le j \le K-1.
    \label{eq:nodes}
\end{equation}
Then
\[
\tau_{t,K-1} < \tau_{t,K-2} < \cdots < \tau_{t,0}.
\]

Given the collection of points
\[
\left(
\tau_{t,j},
\,(1-\tau_{t,j})^{-3/2}
s_{\tau_{t,j}}^\star\!\bigl(Y_{\tau_{t,j}}^{\mathrm{ode}}\bigr)
\right),
\qquad
0 \le j \le K-1,
\]
we approximate \((1-\tau)^{-3/2} s_\tau^\star\!\left(Y_\tau^{\mathrm{ode}}\right)\)
by the degree-\((K-1)\) Lagrange interpolating polynomial through these \(K\)
points, namely,
\begin{equation}
\frac{1}{(1-\tau)^{3/2}} s_\tau^\star\!\left(Y_\tau^{\mathrm{ode}}\right)
\approx
\sum_{j=0}^{K-1} \psi_{t,j}(\tau)
\frac{s_{\tau_{t,j}}^\star\!\left(Y_{\tau_{t,j}}^{\mathrm{ode}}\right)}
{(1-\tau_{t,j})^{3/2}},
\qquad
\forall \tau \in [\tau_{t,K-1}, \tau_{t,0}],
\label{eq:appopoly}
\end{equation}
where \(\psi_{t,j}(\tau)\) are the Lagrange basis polynomials defined by
\begin{equation}
\psi_{t,j}(\tau)
:=
\frac{\prod_{j' \ne j} (\tau - \tau_{t,j'})}
     {\prod_{j' \ne j} (\tau_{t,j} - \tau_{t,j'})},
\qquad
0 \le j \le K-1.
\label{eq:basicfun}
\end{equation}
A closer inspection shows that the approximation in \eqref{eq:appopoly} still
cannot be used directly, because the values
\(\{Y_{\tau_{t,j}}^{\mathrm{ode}}\}_{j=0}^{K-1}\) are generally unavailable. To
overcome this difficulty, we propose a \emph{Gauss--Seidel refinement}
procedure that alternates between estimating
\(\{Y_{\tau_{t,j}}^{\mathrm{ode}}\}_{j=0}^{K-1}\) and approximating
\((1-\tau)^{-3/2} s_\tau^\star\!\left(Y_\tau^{\mathrm{ode}}\right)\). More
precisely, starting from an initial sequence
\(\{x_{\tau_{t,j}}^{(0)}\}_{j=0}^{K-1}\) as an approximation to
\(\{Y_{\tau_{t,j}}^{\mathrm{ode}}\}_{j=0}^{K-1}\), the Gauss--Seidel scheme
updates this estimate to
\(\{x_{\tau_{t,j}}^{(n+1)}\}_{j=0}^{K-1}\) at iteration
\(n = 0,1,\ldots,N-1\) as follows: for each \(j=1,\ldots,K-1\),
\begin{equation}
\frac{x_{\tau_{t,j}}^{(n+1)}}{\sqrt{1-\tau_{t,j}}}
=
\frac{x_{\tau_{t,0}}^{(0)}}{\sqrt{1-\tau_{t,0}}}
+
\sum_{k=0}^{j-1}
\frac{\gamma_{t,k}(\tau_{t,j})}{2(1-\tau_{t,k})^{3/2}}
s_{\tau_{t,k}}\!\left(x_{\tau_{t,k}}^{(n+1)}\right)
+
\sum_{k=j}^{K-1}
\frac{\gamma_{t,k}(\tau_{t,j})}{2(1-\tau_{t,k})^{3/2}}
s_{\tau_{t,k}}\!\left(x_{\tau_{t,k}}^{(n)}\right),
\label{eq:gs_update}
\end{equation}
where \(x_{\tau_{t,0}}^{(n)} = x_{\tau_{t,0}}^{(0)}\) for all \(n=1,\ldots,N\),  $s_{\tau_{t,k}}$ is the estimated score functions, 
and
\begin{equation}
\gamma_{t,k}(\tau_{t,j})
:=
\int_{\tau_{t,j}}^{\tau_{t,0}} \psi_{t,k}(\tau)\,\mathrm{d}\tau.
\label{eq:gammtk}
\end{equation}

We now summarize our higher-order method for diffusion models; the detailed procedure is given in Algorithm~\ref{alg:gs_sampler}.
\begin{itemize}
\item[(i)] \textbf{Initialization.} Sample $Y_T \sim \mathcal{N}(0,I_d)$, and set $x_{\tau_{T,j}}^{(0)} := Y_T$ for all $j = 0,\ldots,K-1$, where $\tau_{T,j}$ are defined in \eqref{eq:nodes}.
\item[(ii)] \textbf{Iterative update rule.} For $t = T,T-1,\ldots,2$, compute $\{x_{\tau_{t,j}}^{(n+1)}\}_{j=0}^{K-1}$ according to the update rule \eqref{eq:gs_update} for $n = 0,1,\ldots,N-1$. After $N$ iterations at step $t$, set $Y_{t-1} := x_{\tau_{t,K-1}}^{(N)}$ and initialize the next step by setting $x_{\tau_{t-1,j}}^{(0)} := Y_{t-1}$ for all $j = 0,\ldots,K-1$.
\end{itemize}

\begin{algorithm}[H]
\caption{Gauss--Seidel Refinement with Higher-Order Diffusion Sampling}
\label{alg:gs_sampler}
\begin{algorithmic}[1]
\Require $T,K,N$, interpolation nodes $\{\tau_{t,i}\}$, coefficients $\{\gamma_{t,j}(\tau_{t,i})\}$, score estimator $s_\tau(\cdot)$
\Ensure $Y_1$

\State Sample $Y_T \sim \mathcal{N}(0,I_d)$
\For{$t=T,T-1,\dots,2$}
    \State Set $x_{\tau_{t,j}}^{(0)} \gets Y_t$ for all $j=0,\ldots,K-1$
    \For{$n=0,\dots,N-1$}
        \State Compute $x_{\tau_{t,j}}^{(n+1)}$ via \eqref{eq:gs_update} sequentially for $j=1,\dots,K-1$
    \EndFor
    \State Set $Y_{t-1} \gets x_{\tau_{t,K-1}}^{(N)}$
\EndFor

\State \Return $Y_1$
\end{algorithmic}
\end{algorithm}

Algorithm~\ref{alg:gs_sampler} shows that, once the initial value \(Y_T\) is sampled, all intermediate iterates \(\{x_{\tau_{t,i}}^{(n)}\}\) are fully determined by the prescribed update rule together with the available score estimator. The total number of score function evaluations is \(O(TK(N+1))\). Since, as shown later, \(K \le c\log T\) and \(N=\lceil C_3K\log T\rceil\), it follows that the overall number of score function evaluations is \(\widetilde{O}(T)\).

\section{Main Results} \label{sec:main_result}
In this section, we establish convergence guarantees for Algorithm~\ref{alg:gs_sampler}. Throughout, let \(q_t := \operatorname{distribution}\left(X_t\right)\) denote the target distribution at time \(t\), where \(X_t\) is the forward process initialized from the data distribution, and let \(p_t := \operatorname{distribution}\left(Y_t\right)\) denote the distribution at time \(t\) generated by our algorithm. We begin by introducing several assumptions on the learning-rate schedule, the target distribution, and the estimated score functions.

\subsection{Assumptions}
For the forward process given in \eqref{forward_process}, the learning rates $\dkh{\alpha_t}_{t=1}^T$ are chosen as follows:
\begin{equation} \label{eq:beta1}
\beta_1 = 1-\alpha_1 = T^{-c_0},
\end{equation}
and for $t \ge 2$,
\begin{equation} \label{eq:learnrate}
\beta_t = 1-\alpha_t
= \frac{c_1 \log T}{T}
\min\left\{
\beta_1\left(1+\frac{c_1 \log T}{T}\right)^t,\,1 \right\},
\end{equation}
where $c_0,c_1>0$ are sufficiently large numerical constants.  This choice of
learning-rate schedule is standard in recent diffusion theory; see, for
example, \cite{benton2024nearly,liyan2024odt,li2025faster,huang2026denoising}.
Next, we impose a mild assumption on the target distribution
\(p_{\mathrm{data}}\).

\begin{assumption}[Target data distribution] \label{assup:1}
The target distribution $p_{\mathrm{data}}$  has a bounded second moment, namely,
\begin{equation*}
\mathbb{E}_{X_0 \sim p_{\mathrm{data}}}\bigl[\|X_0\|_2^2\bigr] \le T^{2c_R}
\end{equation*}
for some sufficiently large constant $c_R>0$.
\end{assumption}

\begin{remark}
Assumption~\ref{assup:1} requires only that the second moment of the target
distribution grow at most polynomially with the iteration number \(T\).
Since \(T\) itself typically scales polynomially with the dimension \(d\), this
condition allows the second moment to remain fairly large. In contrast, the
closely related work \cite{li2025faster} assumes the bounded-support condition
\[
\mathbb{P}\bigl(\|X_0\|_2 \le T^{c_R}\bigr)=1,
\]
which is strictly stronger than Assumption~\ref{assup:1}. Furthermore, our
assumption is more consistent with standard conditions in sampling theory and
more plausible for empirical data encountered in practice.
\end{remark}

As described in Section~\ref{sec:algo}, the true score function appearing in the integral is approximated by a finite collection of estimated score functions. Accordingly, the convergence behavior of our proposed algorithm is intrinsically tied to the accuracy of these score estimates.

\begin{assumption}[Score accuracy] \label{assup:2}
Assume that the score estimates satisfy
\[
\frac{1}{T(N+1)K}\sum_{t=1}^{T}
\mathbb{E}_{Y_t\sim q_t}
\bigl[\varepsilon_{\mathrm{score},t}^2(Y_t)\bigr]
\le
\varepsilon_{\mathrm{score}}^2,
\]
where
\[
\varepsilon_{\mathrm{score},t}^2(Y_t)
:=
\sum_{i=0}^{K-1}\sum_{n=0}^{N}
\left\|
s_{\tau_{t,i}}\!\left(x_{\tau_{t,i}}^{(n)}\right)
-
s_{\tau_{t,i}}^\star\!\left(x_{\tau_{t,i}}^{(n)}\right)
\right\|_2^2
\]
and $x_{\tau_{t,i}}^{(n)}$ are the points generated by \eqref{eq:gs_update},  with  $x_{\tau_{t,0}}^{(n)}=Y_t$ for all $n=0,1,\ldots,N$.
\end{assumption}

For ODE-based samplers, an \(\ell_2\) bound on the score estimation error alone is generally insufficient for convergence analysis. We therefore impose an additional assumption on the Jacobian estimation error.
\begin{assumption}[Jacobian accuracy] \label{assup:3}
Assume that, for all \(1 \le t \le T\) and \(0 \le i \le K-1\), each score
function \(s_{\tau_{t,i}}(\cdot)\) is continuously differentiable. Moreover,
assume that
\[
\frac{1}{T(N+1)K}\sum_{t=1}^{T}
\mathbb{E}_{Y_t\sim q_t}
\bigl[\varepsilon_{\mathrm{Jacobi},t}^2(Y_t)\bigr]
\le
\varepsilon_{\mathrm{Jacobi}}^2,
\]
where
\[
\varepsilon_{\mathrm{Jacobi},t}^2(Y_t)
:=
\sum_{i=0}^{K-1}\sum_{n=0}^{N}
\left\|
\frac{\partial s_{\tau_{t,i}}\!\left(x_{\tau_{t,i}}^{(n)}\right)}{\partial x}
-
\frac{\partial s_{\tau_{t,i}}^\star\!\left(x_{\tau_{t,i}}^{(n)}\right)}{\partial x}
\right\|^2 
\]
and $x_{\tau_{t,i}}^{(n)}$ denotes the sequence of points generated by
\eqref{eq:gs_update}, with \(x_{\tau_{t,0}}^{(n)} = Y_t\) for every
\(n=0,1,\ldots,N\).
\end{assumption}

\subsection{Convergence guarantees}

We now state the main convergence result for Algorithm~\ref{alg:gs_sampler}.

\begin{theorem} \label{main_theorem}
Suppose  that Assumptions \ref{assup:1}, \ref{assup:2}, \ref{assup:3} hold.
Then there exist sufficiently large constants $C_2, C_3, K_1, M_0>1$ and a sufficiently small constant $c>0$ such that, whenever
\[
T \ge C_2 d\log^4 T,
\qquad
N=\lceil C_3K \log T\rceil,
\qquad
2\le K\le c\log T,
\]
the output of Algorithm~\ref{alg:gs_sampler} satisfies
\begin{equation*}
\mathrm{TV}(q_1,p_1)
\lesssim d^2K\log^2 T
\left(\frac{M_0Kd\log^2 T}{T}\right)^K
+ d\log^{5/2}T\sqrt{\log K}\,
\varepsilon_{\mathrm{Jacobi}}
+ d^{3/2}\log^{7/2}T\log K\,
\varepsilon_{\mathrm{score}}
+ T^{-K_1}.
\end{equation*}
\end{theorem}

\begin{remark}
Theorem~\ref{main_theorem} shows that the total variation error consists of four components: the discretization error arising from the higher-order approximation of the continuous reverse process, the Jacobian estimation error, the score estimation error, and the initialization error. We next compare our result with the most closely related existing work.\begin{enumerate}
    \item \textbf{Iteration complexity.} 
    Given the Jacobian estimation error $\varepsilon_{\mathrm{Jacobi}}$ and the score estimation error $\varepsilon_{\mathrm{score}}$,  for any target accuracy $\varepsilon \ge d\log^{5/2}T\sqrt{\log K}\,\varepsilon_{\mathrm{Jacobi}}
+
d^{3/2}\log^{7/2}T\log K\,\varepsilon_{\mathrm{score}}$,  Algorithm \ref{alg:gs_sampler} achieves 
$\mathrm{TV}(q_1,p_1)\lesssim \varepsilon$ within
\begin{equation} \label{eq:iteration_complexity}
T= \widetilde{O}\left(
\max\left\{
d^{1+2/K}\varepsilon^{-1/K},
\varepsilon^{-1/K_1}
\right\}
\right).
\end{equation}
In particular, our analysis allows $K\asymp\log T$.  In this regime,
$d^{1+2/K}=d^{1+o_T(1)}$, and
$\varepsilon^{-1/K}=\varepsilon^{-o_T(1)}$,
so that $T$ can be small as
\[
\widetilde{O}\bigl(d^{1+o_T(1)}\varepsilon^{-1/K_1}\bigr)
\]
which significantly improves upon the best known result $\widetilde{O}\left(\max\left\{d^2,\,d^{1+2/K}\varepsilon^{-1/K} \right\}
\right)$ in  \cite{li2025faster}, where \(K\) is treated as a fixed constant.. 
\item  \textbf{Mild target distribution and score assumptions.}  Our convergence guarantee requires only a finite second-moment condition on the target distribution, and therefore applies to a broad class of data distributions. In particular, it avoids the bounded-support assumption $\mathbb{P}(\|X_0\|_2\le T^{c_R})=1$ imposed in \cite{li2025faster}, as well as other restrictive assumptions such as smoothness or log-concavity. Moreover, compared to other higher-order ODE-based analyses (e.g., \cite{huang2024pflow}), our framework does not require bounded higher-order derivatives or higher-order Lipschitz continuity of the score function. Instead, it relies only on first-order accuracy of the score and its Jacobian along the sampler trajectory. 
\end{enumerate}
\end{remark}

\begin{remark}
We note that the bounds on the Jacobian estimation error and the score estimation error in Theorem~\ref{main_theorem} are larger than those reported in \cite{li2025faster}. This difference arises because our analysis only requires \(T \gtrsim  d \log^4 T\), whereas \cite{li2025faster} assumes the stronger condition \(T \gtrsim  d^2 \log^3 T\). In fact, if one imposes the stronger requirement \(T \gtrsim  d^2 \log^4 T\), our bounds on both the Jacobian and score estimation errors match those in \cite{li2025faster}.
\end{remark}

\section{Proof of the main result} \label{sec:promainresult}
In this section, we present the proofs of the main results. Throughout, we assume that
\begin{equation}\label{eq:small_error}
\varepsilon_{\mathrm{score}}
\lesssim (d^{3/2} \log^{9/2} T)^{-1}
\quad \text{and} \quad
\varepsilon_{\mathrm{Jacobi}}
\lesssim (d \log^3 T)^{-1}.
\end{equation}
Otherwise, Theorem~\ref{main_theorem} follows immediately, since the total variation distance is always bounded above by \(1\). Recall from \eqref{eq:X_T} that \(\operatorname{Law}(X_T)\) is close to
\(\mathcal{N}(0,I_d)\), and that Algorithm~\ref{alg:gs_sampler} is initialized
with \(Y_T \sim \mathcal{N}(0,I_d)\). The following lemma shows that the
distributions of \(X_T\) and \(Y_T\) are indeed close in total variation
distance, with an error that can be controlled explicitly.

\begin{lemma}\cite[Lemma 2]{li2024minimal}
\label{lemma1}
Suppose that \(T\) is sufficiently large, and let \(K_1>0\) be any fixed
constant. Then
\begin{equation}
\bigl(\operatorname{TV}(q_T,p_T)\bigr)^2
\le \frac{1}{2}\operatorname{KL}(q_T \| p_T)
\lesssim \frac{1}{T^{2K_1}}.
\label{first_err}
\end{equation}
\end{lemma}

\subsection{Main steps for proving Theorem 1}
Before proceeding, we introduce some notation. Fix an initial value
\(Y_T \in \mathbb{R}^d\) in Algorithm~\ref{alg:gs_sampler}, and for each
\(1 \le t \le T\), define
\[
x_t := x_{\tau_{t,0}}^{(0)},
\]
where \(x_{\tau_{t,0}}^{(0)}\) is generated by
Algorithm~\ref{alg:gs_sampler}. On each interval
\(\bigl[1-\bar{\alpha}_t,\,1-\bar{\alpha}_{t-1}\bigr]\), let
\(\{x_\tau^\star\}_{\tau \in [\tau_{t,K-1},\tau_{t,0}]}\) denote the exact
probability flow ODE trajectory associated with \eqref{ODE}, with
\(\tau_0 := 1-\bar{\alpha}_t\) and initial condition
\(Y_{\tau_0}^{\mathrm{ode}} := x_t\), and set
\[
x_{t-1}^\star := x_{\tau_{t,K-1}}^\star.
\]
Likewise, let \(Y_{t-1}^\star\) denote the solution to \eqref{ODE} at
\(\tau = 1-\bar{\alpha}_{t-1}\) with 
\(\tau_0 := 1-\bar{\alpha}_t\) and initial condition
\(Y_{\tau_0}^{\mathrm{ode}} := Y_t\). With this notation, we now outline the
proof strategy, which consists of several steps.

\paragraph{Step 1: controlling the density ratio.}
With Lemma \ref{lemma1} in hand, the  basic idea to establish the convergence rate is to connect  the density ratio $\frac{p_{X_{t-1}}(x_{t-1})}{p_{Y_{t-1}}(x_{t-1})}$ at the $(t-1)$-th step with its counterpart part  $\frac{p_{X_t}(x_t)}{p_{Y_t}(x_t)}$ at the $t$-th step. To this end, we begin with the identity
\begin{align}
\frac{p_{X_{t-1}}(x_{t-1})}{p_{Y_{t-1}}(x_{t-1})}
&=
\frac{
p_{\sqrt{\alpha_t}X_{t-1}}(\sqrt{\alpha_t}x_{t-1})
}{
p_{\sqrt{\alpha_t}X_{t-1}}(\sqrt{\alpha_t}x_{t-1}^{\star})
}
\cdot
\frac{
p_{\sqrt{\alpha_t}X_{t-1}}(\sqrt{\alpha_t}x_{t-1}^{\star})
}{
p_{X_t}(x_t)
}
\notag\\
&\quad\cdot
\left[
\frac{
p_{\sqrt{\alpha_t}Y_{t-1}}(\sqrt{\alpha_t}x_{t-1})
}{
p_{\sqrt{\alpha_t}Y_{t-1}^{\star}}(\sqrt{\alpha_t}x_{t-1}^{\star})
}
\cdot
\frac{
p_{\sqrt{\alpha_t}Y_{t-1}^{\star}}(\sqrt{\alpha_t}x_{t-1}^{\star})
}{
p_{Y_t}(x_t)
}
\right]^{-1}
\cdot
\frac{p_{X_t}(x_t)}{p_{Y_t}(x_t)}.
\label{eq:step1_density_decomp_gs}
\end{align}
Since $Y_{t-1}^{\star}$ is obtained by transporting $Y_t$ through the exact
probability-flow ODE \eqref{ODE_score} on $[\tau_{t,K-1},\tau_{t,0}]$, it implies
\begin{equation}
\frac{
p_{\sqrt{\alpha_t}Y_{t-1}^{\star}}(\sqrt{\alpha_t}x_{t-1}^{\star})
}{
p_{Y_t}(x_t)
}
=
\frac{
p_{\sqrt{\alpha_t}X_{t-1}}(\sqrt{\alpha_t}x_{t-1}^{\star})
}{
p_{X_t}(x_t)
}.
\label{eq:exact_transport_cancel_gs}
\end{equation}
Substituting \eqref{eq:exact_transport_cancel_gs} into \eqref{eq:step1_density_decomp_gs} gives
\begin{equation*}
\frac{p_{X_{t-1}}(x_{t-1})}{p_{Y_{t-1}}(x_{t-1})}
=
\frac{
p_{\sqrt{\alpha_t}X_{t-1}}(\sqrt{\alpha_t}x_{t-1})
}{
p_{\sqrt{\alpha_t}X_{t-1}}(\sqrt{\alpha_t}x_{t-1}^{\star})
}
\cdot
\frac{
p_{\sqrt{\alpha_t}Y_{t-1}^{\star}}(\sqrt{\alpha_t}x_{t-1}^{\star})
}{
p_{\sqrt{\alpha_t}Y_{t-1}}(\sqrt{\alpha_t}x_{t-1})
}
\cdot
\frac{p_{X_t}(x_t)}{p_{Y_t}(x_t)}.
\end{equation*}

It has been shown in Lemma~6 in \cite{li2025faster} that the right hand side can be well controlled by the distance between $x_{t-1}$ and $x_{t-1}^\star$ and their Jacobian matrices, as stated below.
\begin{lemma}[Lemma~6 in \cite{li2025faster}] \label{lemma2}
For each $t=2,\ldots,T$,  it holds
\begin{enumerate}
\item[(a)]
\begin{equation}
\frac{
p_{\sqrt{\alpha_t}X_{t-1}}\!\left(\sqrt{\alpha_t}\,x_{t-1}\right)
}{
p_{\sqrt{\alpha_t}X_{t-1}}\!\left(\sqrt{\alpha_t}\,x_{t-1}^\star\right)
}
=
\exp\!\left(
O\!\left(
\frac{\|x_{t-1}-x_{t-1}^\star\|_2^2}{1-\bar{\alpha}_{t-1}}
+
\sqrt{\frac{d\log T}{1-\bar{\alpha}_{t-1}}}\|x_{t-1}-x_{t-1}^\star\|_2\,
\right)
\right).
\label{eq:lemma2_a}
\end{equation}

\item[(b)] Assume  that
\begin{equation}
\|(J_{\tau_{t,K-1}}^{(N)})^{-1}\|\lesssim 1,
\label{eq:lemma2_b_cond}
\end{equation}
then
\begin{equation}
\frac{
p_{\sqrt{\alpha_t}Y_{t-1}^{\star}}\!\left(\sqrt{\alpha_t}\,x_{t-1}^\star\right)
}{
p_{\sqrt{\alpha_t}Y_{t-1}}\!\left(\sqrt{\alpha_t}\,x_{t-1}\right)
}
=
\exp\!\left(
O\!\left(
d\,\|J_{\tau_{t,K-1}}^{(N)}-J_{\tau_{t,K-1}}^\star\|
\right)
\right).
\label{eq:lemma2_b}
\end{equation}
\end{enumerate}
Here, 
$J_{\tau_{t,i}}^{(n)} :=
\frac{\partial (x_{\tau_{t,i}}^{(n)}/\sqrt{1-\tau_{t,i}})}
{\partial (x_{\tau_{t,0}}/\sqrt{1-\tau_{t,0}})},
J_{\tau_{t,i}}^\star
:=
\frac{\partial (x_{\tau_{t,i}}^\star/\sqrt{1-\tau_{t,i}})}
{\partial (x_{\tau_{t,0}}^\star/\sqrt{1-\tau_{t,0}})}$ and we regard $x_{\tau_{t,i}}^{(n)}$ and $x_{\tau_{t,i}}^\star$ as functions of $x_{\tau_{t,0}}$, since the Gauss--Seidel iterates and the exact ODE trajectory are determined by the rescaled initial value
$x_{\tau_{t,0}}/\sqrt{1-\tau_{t,0}}$.
\end{lemma}

In view of Lemma~\ref{lemma2}, we obtain
\begin{equation}
\frac{p_{X_{t-1}}(x_{t-1})}{p_{Y_{t-1}}(x_{t-1})}
=
\exp\!\left(
O\!\left(
\frac{\|x_{t-1}-x_{t-1}^\star\|_2^2}{1-\bar{\alpha}_{t-1}}
+
\sqrt{\frac{d\log T}{1-\bar{\alpha}_{t-1}}}\|x_{t-1}-x_{t-1}^\star\|_2\,
+
d\,\|J_{\tau_{t,K-1}}^{(N)}-J_{\tau_{t,K-1}}^\star\|
\right)
\right)
\frac{p_{X_t}(x_t)}{p_{Y_t}(x_t)}.
\label{eq:step1_generic_recursion_gs}
\end{equation}

By Algorithm~\ref{alg:gs_sampler},  we know $x_{t-1}=x_{\tau_{t-1, 0}}^{(0)}=x_{\tau_{t, K-1}}^{(N)}$.
Therefore, it suffices to control
\[
\|x_{\tau_{t, K-1}}^{(N)}-x_{\tau_{t,K-1}}^\star\|_2
\qquad\text{and}\qquad
\|J_{\tau_{t,K-1}}^{(N)}-J_{\tau_{t,K-1}}^\star\|.
\]
A key observation is that, for the probability ODE flow, once $Y_T$ in Algorithm~\ref{alg:gs_sampler} is fixed, all quantities \(x_{\tau_{t,i}}^{(n)}\), \(x_{\tau_{t,i}}^\star\), and \(x_\tau^\star\) are fully determined. In the initialization stage, we take $Y_T \sim \mathcal{N}(0,I_d)$. We will show that for ``typical'' points $Y_T \in \mathbb{R}^d$, the above two quantities are sufficiently small. We say that a point $Y_T \in \R^d$ is \emph{typical} if
\begin{equation} \label{eq:gs_event_E}
Y_T \in E_t := E_{1,t} \cap E_{2,t} \cap E_{3,t},
\end{equation}
where the events $E_{1,t}$, $E_{2,t}$, and $E_{3,t}$ are defined as follows:
\begin{equation*}
E_{1,t} := \dkh{ Y_T :  \frac{d\sqrt{\log K}\log T}{T}
\sqrt{\sum_{i=0}^{K-1}\sum_{n=0}^{N}
\Bigl(\varepsilon_{\mathrm{Jacobi},t,i}^{(n)}(x_{\tau_{t,i}}^{(n)})\Bigr)^2}
+
\frac{d^{3/2}\log K\log^2T}{T}
\sqrt{\sum_{i=0}^{K-1}\sum_{n=0}^{N}
\Bigl(\varepsilon_{\mathrm{score},t,i}^{(n)}(x_{\tau_{t,i}}^{(n)})\Bigr)^2} \le c_3}
\end{equation*}
\begin{equation*}
E_{2,t} := \dkh{ Y_T\in\mathbb R^d: 
-\log p_{\bar X_{\tau_{t,i}}}
\!\left(
\lambda x_{\tau_{t,i}}^{(n)}
+(1-\lambda)x_{\tau_{t,i}}^\star
\right)
\le C_4 d\log T,
\text{for all } 0\le i\le K-1,\ 0\le n\le N,\ \lambda\in[0,1] }
\end{equation*}
\begin{equation*}
E_{3,t} := \dkh{ Y_T\in\mathbb R^d:
-\log p_{\bar X_\tau}(x_\tau^\star)
\le C_4 d\log T,  \quad
\text{for all }\quad\tau\in[\tau_{t,K-1},\tau_{t,0}] }.
\end{equation*}
Here, \(C_4 > 0\) is a sufficiently large absolute constant,  \(c_3 > 0\) is a sufficiently small constant, and we denote
\begin{equation} \label{def_vareps}
\varepsilon_{\mathrm{score},t,i}^{(n)}
\bigl(x_{\tau_{t,i}}^{(n)}\bigr)
:=
\left\|
s_{\tau_{t,i}}\!\left(x_{\tau_{t,i}}^{(n)}\right)
-
s_{\tau_{t,i}}^\star\!\left(x_{\tau_{t,i}}^{(n)}\right)
\right\|_2,\quad
\varepsilon_{\mathrm{Jacobi},t,i}^{(n)}
\bigl(x_{\tau_{t,i}}^{(n)}\bigr)
:=
\left\|
\frac{\partial s_{\tau_{t,i}}\!\left(x_{\tau_{t,i}}^{(n)}\right)}{\partial x}
-
\frac{\partial s_{\tau_{t,i}}^\star\!\left(x_{\tau_{t,i}}^{(n)}\right)}{\partial x}
\right\|.
\end{equation}

The next lemma shows that under the event $E_t$, the terms $\|x_{t-1}-x_{\tau_{t,K-1}}^\star\|_2$ are well controlled. 

\begin{lemma} \label{lemma3}
For each $t$, on the event $E_t$,  when $K\lesssim \log T$ and $T \ge C_2 d\log^4 T$ for some sufficiently large constant $C_2>0$, it holds
\begin{equation*}
\|x_{\tau_{t,i}}^{(N)}-x_{\tau_{t,i}}^\star\|_2^2
\lesssim \log K\,\tau_{t,0}^2 \frac{\log ^2 T}{T^2}
\sum_{n=0}^{N}\sum_{j=0}^{K-1}
\Bigl(\varepsilon_{\mathrm{score}, t, j}^{(n)}(x_{\tau_{t,j}}^{(n)})\Bigr)^2
+  \frac{d \tau_{t,0} \log ^3 T}{T^2}
\left(\frac{M_0Kd \log ^2 T}{T}\right)^{2K}
\end{equation*}
for all $0\le i\le K-1$ and where $M_0>1$ is a large enough constant.
\end{lemma}
\begin{proof}
See Section \ref{lemma3_proof}.
\end{proof}

Similarly, under the event $E_t$, the terms $\|J_{\tau_{t,i}}^{(N)}-J_{\tau_{t,i}}^\star\|^2$ are well controlled. 

\begin{lemma}  \label{lemma4}
For each $t$, on the event $E_t$, when $K\lesssim \log T$ and $T \ge C_2 d\log^4 T$ for some sufficiently large constant $C_2>0$, the Jacobian  satisfies
\begin{equation} \label{lemma4_lemma2_condition}
\left\|(J_{\tau_{t, K-1}}^{(N)})^{-1}\right\| \lesssim 1
\end{equation}
and
\begin{align} \label{eq:lemma4_gs_bound}
\|J_{\tau_{t,i}}^{(N)}-J_{\tau_{t,i}}^\star\|^2
&\lesssim  K\log^2 K\,\frac{d^3\log^7 T}{T^4} \sum_{n=0}^{N}\sum_{j=0}^{K-1}
\Bigl(\varepsilon_{\mathrm{score},t,j}^{(n)}(x_{\tau_{t,j}}^{(n)})\Bigr)^2
\nonumber\\
&\quad+
 \log K \tau_{t,0}^2 \frac{\log^2 T}{T^2}
\sum_{n=0}^{N}\sum_{j=0}^{K-1}
\Bigl(\varepsilon_{\mathrm{Jacobi},t,j}^{(n)}(x_{\tau_{t,j}}^{(n)})\Bigr)^2
+  \left(\frac{M_0Kd\log^2 T}{T}\right)^{2K+2}
\end{align}
for all $0\le i\le K-1$ and where $M_0>1$ is a large enough constant.
\end{lemma}
\begin{proof}
See Section \ref{lemma3_proof}.
\end{proof}

Putting Lemma~\ref{lemma3} and  Lemma \ref{lemma4} into 
\eqref{eq:step1_generic_recursion_gs}, we obtain that on the Event $E_t$, it holds
\begin{equation}
\frac{p_{X_{t-1}}(x_{t-1})}{p_{Y_{t-1}}(x_{t-1})}
=
\exp\!\left(
O\!\left(
\widetilde{\xi}_t(x_t)
+
d\left(\frac{M_0Kd\log^2T}{T}\right)^{K+1}
\right)
\right)
\frac{p_{X_t}(x_t)}{p_{Y_t}(x_t)},
\label{eq:gs_density_ratio_final_compact}
\end{equation}
provided that \(T\gtrsim d\log^4T\), where
\begin{equation}
\widetilde{\xi}_t(x_t)
:=
\frac{d\sqrt{\log K}\log T}{T}
\sqrt{\sum_{i=0}^{K-1}\sum_{n=0}^{N}
\Bigl(\varepsilon_{\mathrm{Jacobi},t,i}^{(n)}(x_{\tau_{t,i}}^{(n)})\Bigr)^2}
+
\frac{d^{3/2}\log K\log^2T}{T}
\sqrt{\sum_{i=0}^{K-1}\sum_{n=0}^{N}
\Bigl(\varepsilon_{\mathrm{score},t,i}^{(n)}(x_{\tau_{t,i}}^{(n)})\Bigr)^2}.
\label{eq:gs_xi_t_definition}
\end{equation}

\paragraph{Step 2: decomposing the total variation distance.}
Define 
\begin{equation} \label{eq:def_E_step2_lemma}
\mathcal{E} := \left\{ x \in \mathbb{R}^d :\; q_1(x) > \max\bigl\{p_1(x),\, \exp(-c_6 d \log T)\bigr\} \right\}.
\end{equation}
Informally, the set $\mathcal{E}$ collects those typical points $x \in \mathbb{R}^d$ at which $q_1(x)$ is not exponentially small. The next lemma shows that the total variation distance between $q_1$ and $p_1$ is primarily controlled by their discrepancy on $\mathcal{E}$.

\begin{lemma}
\label{lemma5}
Assume that Assumption~\ref{assup:1} holds. Then, for any fixed constant
\(K_1>0\), if \(c_6\) in \eqref{eq:def_E_step2_lemma} is chosen sufficiently
large, it holds that
\begin{equation*}
\operatorname{TV}(q_1,p_1)
\le
\mathbb{E}_{Y_1 \sim p_1}\left[
\left(\frac{q_1(Y_1)}{p_1(Y_1)}-1\right)\mathbf{1}\{Y_1 \in \mathcal E\}
\right]
+
C T^{-2K_1}
+
\exp(-c_2 d\log T),
\end{equation*}
where \(C>0\) is an absolute constant and \(c_2>0\) is a constant depending on \(c_6\), \(c_R\).
\end{lemma}
\begin{proof}
The proof is deferred to Appendix~\ref{lemma5_proof}.
\end{proof}

Define
\begin{equation}
\widetilde{S}_t(Y_T)
:=
\sum_{k=2}^t \widetilde{\xi}_k(x_k),
\qquad t \ge 2,
\qquad
\widetilde{S}_1(Y_T):=0,
\label{eq:def_tilde_St}
\end{equation}
where \(\{x_k\}\) are generated by Algorithm~\ref{alg:gs_sampler} with initial
sample \(Y_T\), and \(\widetilde{\xi}_k(x_k)\) is defined in
\eqref{eq:gs_xi_t_definition}. Next, define
\begin{equation}
\mathcal I_1 :=
\left\{ Y_T\in\mathbb R^d : \widetilde{S}_T(Y_T)\le c_3 \right\},
\label{eq:def_I1}
\end{equation}
where \(c_3>0\) is a sufficiently small constant. Thus, \(\mathcal I_1\)
represents the set of initial points for which the accumulated score error
along the backward trajectory is well controlled.
With the above notation, Lemma~\ref{lemma5} immediately yields
\begin{eqnarray}
\operatorname{TV}(q_1,p_1)
&\le& \mathbb{E}_{Y_1 \sim p_1}\left[\left(\frac{q_1\left(Y_1\right)}{p_1\left(Y_1\right)}-1\right) \mathbf{1}\left\{Y_1 \in \mathcal E\right\}\right]+C T^{-2K_1}+\exp(-c_2\,d\log T)
\nonumber\\ 
&=&
\underbrace{
\mathbb{E}_{Y_T \sim p_T}\!\left[
\left(\frac{q_1(Y_1)}{p_1(Y_1)}-1\right)
\mathbf{1}\{Y_1\in \mathcal E,\;Y_T\in\mathcal I_1\}
\right]
}_{\alpha_1}
+
\underbrace{
\mathbb E_{Y_T\sim p_T}\!\left[
\left(\frac{q_1(Y_1)}{p_1(Y_1)}-1\right)
\mathbf 1\{Y_1\in \mathcal E,\;Y_T\in \mathcal I_2\}
\right]
}_{\alpha_2}
\nonumber\\
&&+C T^{-2K_1}+\exp(-c_2\,d\log T),
\label{eq:alp12}
\end{eqnarray}
where the second equality follows from the fact that \(Y_1\) is determined
entirely by \(Y_T\) through the deterministic update rule, and
\(\mathcal I_2 := \mathbb{R}^d \setminus \mathcal I_1\). As we will show later,
for any \(Y_T \in \mathcal I_1\) such that \(x_1 \in \mathcal E\), or
equivalently \(Y_1 \in \mathcal E\), one can conclude that \(Y_T\) is a typical
point in \(E_t\). In this case, we may apply
\eqref{eq:gs_density_ratio_final_compact} to control \(\alpha_1\). On the other
hand, for any \(Y_T \in \mathcal I_2\) with \(Y_1 \in \mathcal E\), we will show
that its contribution to the total variation distance is negligible, so that
\(\alpha_2\) is also well controlled.

\paragraph{Step 3: bounding \(\alpha_1\) and $\alpha_2$.}
For any fixed point $Y_T\in\mathbb{R}^d$, define
\begin{equation}
\tau(Y_T) :=
\max\left\{
2\le t\le T+1:\;
\widetilde{S}_{t-1}(Y_T)\le c_3
\right\}.
\label{eq:def_tau_xt}
\end{equation}

 The next lemma shows that we can control the density ratio up to the $(\tau(Y_T) -1)$-th iteration.

\begin{lemma} \label{lemma6}
For any \(Y_T\in \mathcal{I}_1\) such that \(x_1\in \mathcal E\), the following hold.   For every \(2\le t\le \tau(Y_T)-1\),
\begin{equation}
\frac{q_1(x_1)}{p_1(x_1)} =
\left(
1+
O\!\left(
\sum_{t<\tau(Y_T)}\widetilde{\xi}_t(x_t)
+
d^2\log^2 T\,K
\left(
\frac{M_0Kd\log^2 T}{T}
\right)^K
\right)
\right)
\frac{q_{\tau(Y_T)-1}(x_{\tau(Y_T)-1})}{p_{\tau(Y_T)-1}(x_{\tau(Y_T)-1})},
\label{eq:lemma6_first_here}
\end{equation}
and
\begin{equation}
\frac{q_\ell(x_\ell)}{2p_\ell(x_\ell)}
\le
\frac{q_1(x_1)}{p_1(x_1)}
\le
\frac{2q_\ell(x_\ell)}{p_\ell(x_\ell)},
\qquad
\forall\,\ell\le \tau(Y_T)-1.
\label{eq:lemma6_second_here}
\end{equation}
Here, \(M_0>1\) is a sufficiently large constant. 
\end{lemma}
\begin{proof}
The result follows from Lemmas~\ref{lemma14} and \ref{lemma15} by the same
argument as in the proof of Lemma~10 of \cite{li2025faster}. We therefore omit
the details for brevity.
\end{proof}

By definition, for every \(Y_T \in \mathcal I_1\), we have \(\tau(Y_T)=T+1\). Combining this with Lemma~\ref{lemma6}, we obtain
\begin{align*}
 \alpha_1 &=  \mathbb{E}_{Y_T \sim p_T}\!\Biggl[
\Biggl(
\Bigl(
1+\sum_{t=2}^T\widetilde{\xi}_t(x_t)
+d^2\log^2TK\Bigl(\frac{M_0Kd\log^2T}{T}\Bigr)^K
\Bigr)
\frac{q_T(Y_T)}{p_T(Y_T)}
-1 \Biggr)
\mathbf{1}\{Y_1\in \mathcal E,\;Y_T\in\mathcal I_1\}
\Biggr]
\nonumber\\
&= \int \Biggl[ \Biggl( \Bigl( 1+\sum_{t=2}^{T}\widetilde{\xi}_t(x_t) +d^2\log^2TK\Bigl(\frac{M_0Kd\log^2T}{T}\Bigr)^K \Bigr)
q_T(x_T) -p_T(x_T) \Biggr) \mathbf{1}\{x_1\in \mathcal  E,\;x_T\in\mathcal I_1\} \Biggr] \,\mathrm{d}x_T \nonumber\\
& \le \int |q_T(x_T)-p_T(x_T)|\,\mathrm{d}x_T+ \int \sum_{t=2}^{T}\widetilde{\xi}_t(x_t)\,q_T(x_T)\, \mathbf{1}\{x_1\in \mathcal E,\;x_T\in\mathcal I_1\} \,\mathrm{d}x_T + d^2\log^2 TK \left( \frac{M_0Kd\log^2 T}{T} \right)^K. 
\end{align*}
For the first term on the right-hand side, Lemma~\ref{lemma1} yields
\[
\int |q_T(x_T)-p_T(x_T)|\,\mathrm{d}x_T =2 \operatorname{TV}\left(q_T,p_T\right) \lesssim \frac1{T^{K_1}}.
\]
For the second term, observe that
\[
\begin{aligned}
& \int \widetilde{S}_T(x_T)\,q_T(x_T)\,\mathbf{1}\{x_1\in \mathcal E,\;x_T\in\mathcal I_1\}\,\mathrm{d}x_T \\
=& \sum_{t=2}^T
\mathbb{E}_{Y_T\sim p_T}\!\left[
\frac{\log T}{T}
\left(d\sqrt{\log K}\,\varepsilon_{\mathrm{Jacobi},t}(Y_t)
+d^{3/2}\log K\log T\,\varepsilon_{\mathrm{score},t}(Y_t)\right)
\frac{q_T(Y_T)}{p_T(Y_T)}
\mathbf{1}\{Y_1\in \mathcal E,\;Y_T\in\mathcal I_1\}
\right] \\
\stackrel{(\ref{eq:lemma6_second_here})}{\le}&
4\sum_{t=2}^T
\mathbb{E}_{Y_T\sim p_T}\!\left[
\frac{\log T}{T}
\left(d\sqrt{\log K}\,\varepsilon_{\mathrm{Jacobi},t}(Y_t)
+d^{3/2}\log K\log T\,\varepsilon_{\mathrm{score},t}(Y_t)\right)
\frac{q_t(Y_t)}{p_t(Y_t)}
\mathbf{1}\{Y_1\in \mathcal E,\;Y_T\in\mathcal I_1\}
\right] \\
\le&
4\sum_{t=2}^T
\mathbb{E}_{Y_t\sim p_t}\!\left[
\frac{\log T}{T}
\left(d\sqrt{\log K}\,\varepsilon_{\mathrm{Jacobi},t}(Y_t)
+d^{3/2}\log K\log T\,\varepsilon_{\mathrm{score},t}(Y_t)\right)
\frac{q_t(Y_t)}{p_t(Y_t)}
\right] \\
=&
4\sum_{t=2}^T
\mathbb{E}_{Y_t\sim q_t}\!\left[
\frac{\log T}{T}
\left(d\sqrt{\log K}\,\varepsilon_{\mathrm{Jacobi},t}(Y_t)
+d^{3/2}\log K\log T\,\varepsilon_{\mathrm{score},t}(Y_t)\right)
\right]\\
\lesssim & \quad d\log^{5/2}T\sqrt{\log K}\,\varepsilon_{\mathrm{Jacobi}}
+d^{3/2}\log^{7/2}T\log K\varepsilon_{\mathrm{score}},
\end{aligned}
\]
where the last inequality follows from Assumptions~\ref{assup:2} and~\ref{assup:3}, together with the bounds
\begin{align*}
\frac1T\sum_t \mathbb{E}_{Y_t\sim q_t} \bigl[\varepsilon_{\mathrm{score},t}(Y_t)\bigr]
&\le \sqrt{ \frac1T\sum_t \mathbb{E}_{Y_t\sim q_t} \bigl[\varepsilon_{\mathrm{score},t}^2(Y_t)\bigr] } \lesssim
\sqrt{(N+1)K}\,\varepsilon_{\mathrm{score}} \asymp \varepsilon_{\mathrm{score}}\log^{3/2} T, \\
\frac1T\sum_t \mathbb{E}_{Y_t\sim q_t} \bigl[\varepsilon_{\mathrm{Jacobi},t}(Y_t)\bigr]
&\le \sqrt{ \frac1T\sum_t \mathbb{E}_{Y_t\sim q_t} \bigl[\varepsilon_{\mathrm{Jacobi},t}^2(Y_t)\bigr] }
\lesssim \sqrt{(N+1)K}\,\varepsilon_{\mathrm{Jacobi}} \asymp \varepsilon_{\mathrm{Jacobi}}\log^{3/2} T.
\end{align*}
Combining  the above estimates,  we arrive at
\[
\alpha_1 \lesssim
\frac1{T^{K_1}}+d^2\log^2 TK\left(\frac{M_0Kd\log^2 T}{T}\right)^K
+d\log^{5/2}T\sqrt{\log K}\,\varepsilon_{\mathrm{Jacobi}}
+d^{3/2}\log^{7/2}T\log K\varepsilon_{\mathrm{score}}.
\]

Using techniques similar to those in \cite{li2024sharp}, the term \(\alpha_2\) can also be controlled effectively, as stated in the following lemma.

\begin{lemma} \label{lemma7}
Suppose that Assumptions~\ref{assup:1}, \ref{assup:2}, and \ref{assup:3} hold. If
\(K \lesssim \log T\) and \(T \ge C_2 d\log^4 T\), where \(M_0, K_1>0\) are sufficiently large constants and \(c_2\) is the constant
appearing in Lemma~\ref{lemma5}, then
\[
\alpha_2 \lesssim d^2\log^2 TK\left(\frac{M_0Kd\log^2 T}{T}\right)^K +d\log^{5/2}T\sqrt{\log K}\,\varepsilon_{\mathrm{Jacobi}} +d^{3/2}\log^{7/2}T\log K\varepsilon_{\mathrm{score}} + e^{-c_2 d\log T} + \frac1{T^{K_1}}.
\]
\end{lemma}
\begin{proof}
The proof follows arguments similar to those of Lemma~11 in \cite{li2024sharp}, with minor modifications. Specifically, the proof in \cite{li2024sharp} relies on a bounded-support assumption on \(X_0\). However, a careful inspection shows that this condition can be replaced by Assumption~\ref{assup:1} combined with Markov's inequality. This yields the similar result up to an additional negligible error \(T^{-K_1} + \exp(-c_2 d \log T)\). We therefore omit the proof for brevity.
\end{proof}

\paragraph{Step 4: controlling $\operatorname{TV}(q_1,p_1)$.}
Substituting the bounds for \(\alpha_1\) and \(\alpha_2\) into
\eqref{eq:alp12}, we obtain the desired result
\begin{eqnarray*}
   &&  \operatorname{TV}(q_1,p_1)  \\
    &\lesssim &  d^2\log^2 TK\left(\frac{M_0Kd\log^2 T}{T}\right)^K
+d\log^{5/2}T\sqrt{\log K}\,\varepsilon_{\mathrm{Jacobi}}
+d^{3/2}\log^{7/2}T\log K\varepsilon_{\mathrm{score}}+e^{-c_2 d \log T}+\frac1{T^{K_1}} \\
   &  \asymp &d^2\log^2 TK\left(\frac{M_0Kd\log^2 T}{T}\right)^K
+d\log^{5/2}T\sqrt{\log K}\,\varepsilon_{\mathrm{Jacobi}}
+d^{3/2}\log^{7/2}T\log K\varepsilon_{\mathrm{score}}+\frac1{T^{K_1}}.
\end{eqnarray*}

\section{Numerical experiments}
\label{sec:numerical_study}

In this section, we numerically compare our method with the Li-HEROISM
baseline~\cite{li2025faster}. Following Huang et al.~\cite{huang2024pflow}, we
consider a finite Gaussian mixture target distribution of the form
\begin{equation}
p_{\mathrm{data}}(x)
=
\sum_{\ell=1}^{M}
w_\ell \mathcal N(x;m_\ell,C_\ell),
\qquad
\sum_{\ell=1}^{M}w_\ell=1,
\label{eq:num_target_gmm}
\end{equation}
where \(M\) denotes the number of mixture components, \(x \in \mathbb{R}^d\),
\(m_{\ell} \in \mathbb{R}^d\), and \(C_\ell \in \mathbb{R}^{d \times d}\). Under
the forward process \eqref{def_1}, the marginal density of \(\bar X_\tau\)
remains a Gaussian mixture, as observed in \cite{huang2024pflow}:
\begin{equation}
p_{\bar X_\tau}(x)
=
\sum_{\ell=1}^{M}
w_\ell
\mathcal N
\left(
x;
\sqrt{1-\tau}\,m_\ell,\,
(1-\tau)C_\ell+\tau I_d
\right).
\label{eq:num_forward_density}
\end{equation}
Consequently, for each \(x \in \mathbb{R}^d\), the exact score function used in
the experiments is available in closed form:
\begin{equation}
s_\tau^\star(x)
=
-\sum_{\ell=1}^{M}
\frac{
w_\ell
\mathcal N
\left(
x;
\sqrt{1-\tau}\,m_\ell,\,
(1-\tau)C_\ell+\tau I_d
\right)
}{
p_{\bar X_\tau}(x)
}
\left((1-\tau)C_\ell+\tau I_d\right)^{-1}
\left(x-\sqrt{1-\tau}\,m_\ell\right).
\label{eq:num_exact_score}
\end{equation}

In practice, the score function is typically trained  by a neural network via score matching on progressively corrupted data. In the present
numerical study, however, we bypass the training stage and instead evaluate the
samplers using an imperfect analytical score parameterized by a scalar
\(\delta\):
\begin{equation}
s_\tau(x)
=
s_\tau^\star(x)
+
\delta \eta(x).
\label{eq:num_imperfect_score}
\end{equation}
Following the artificial-error protocol of Huang et al.~\cite{huang2024pflow},
we consider the perturbations
\[
\eta_{\mathrm{const}}(x)=\frac{1}{\sqrt d}\mathbf{1},
\qquad
\eta_{\mathrm{lin}}(x)=\frac{x-m_0}{\sqrt d},
\qquad
\eta_{\mathrm{sin}}(x)=\frac{\sin(x)\odot(x-m_0)}{\sqrt d}\in\mathbb{R}^d,
\]
where \(m_0=\mathbb E[X_0]\), the sine function is applied coordinatewise,
\(\odot\) denotes coordinatewise multiplication, and
\(\mathbf{1}\in\mathbb{R}^d\) is the all-ones vector.

The reverse dynamics are governed by the probability flow ODE \eqref{ODE}. We
use the same two-phase schedule \eqref{eq:beta1}--\eqref{eq:learnrate} as in
the theoretical analysis, with \(c_0=1\) and \(c_1=0.5\), and linearly rescale
the raw endpoints \(\tau_t=1-\bar\alpha_t\) to the fixed numerical interval
\([\tau_{\min},\tau_{\max}]=[10^{-3},0.999]\). Thus, \(\tau_{\max}=0.999\) is
used as a numerical approximation to the fully noised endpoint \(\tau=1\),
whereas \(\tau_{\min}=10^{-3}\) serves as a numerical approximation to the data
endpoint \(\tau=0\). To initialize the reverse sampler, we draw \(J\) particles
from \(Y_T\sim \mathcal N(0,I_d)\) at \(\tau_{\max}\) and evolve the
deterministic reverse process down to \(\tau_{\min}\).

With \(T\) outer iterations, \(K\) interpolation nodes, \(N\) refinement
rounds, and cached score evaluations, the total numbers of score function
evaluations are
\begin{equation}
\mathrm{SFE}_{\mathrm{Li}}
=
T\bigl(1+(K-1)N\bigr),
\qquad
\mathrm{SFE}_{\mathrm{Ours}}
=
T\bigl(K+(K-1)N\bigr).
\label{eq:num_sfe_count}
\end{equation}
Here, \(\mathrm{SFE}\) denotes the total number of score function evaluations
over all outer iterations. For each fixed time point \(\tau\), one batched
evaluation of \(s_\tau(\cdot)\) over all \(J\) particles is counted as a single
score function evaluation. The goal of the experiments is to compare the
finite-budget performance of the two methods under the same target
distribution, score perturbation, and total score-evaluation budget.

\begin{figure}[H]
\centering
\includegraphics[width=0.33\linewidth]{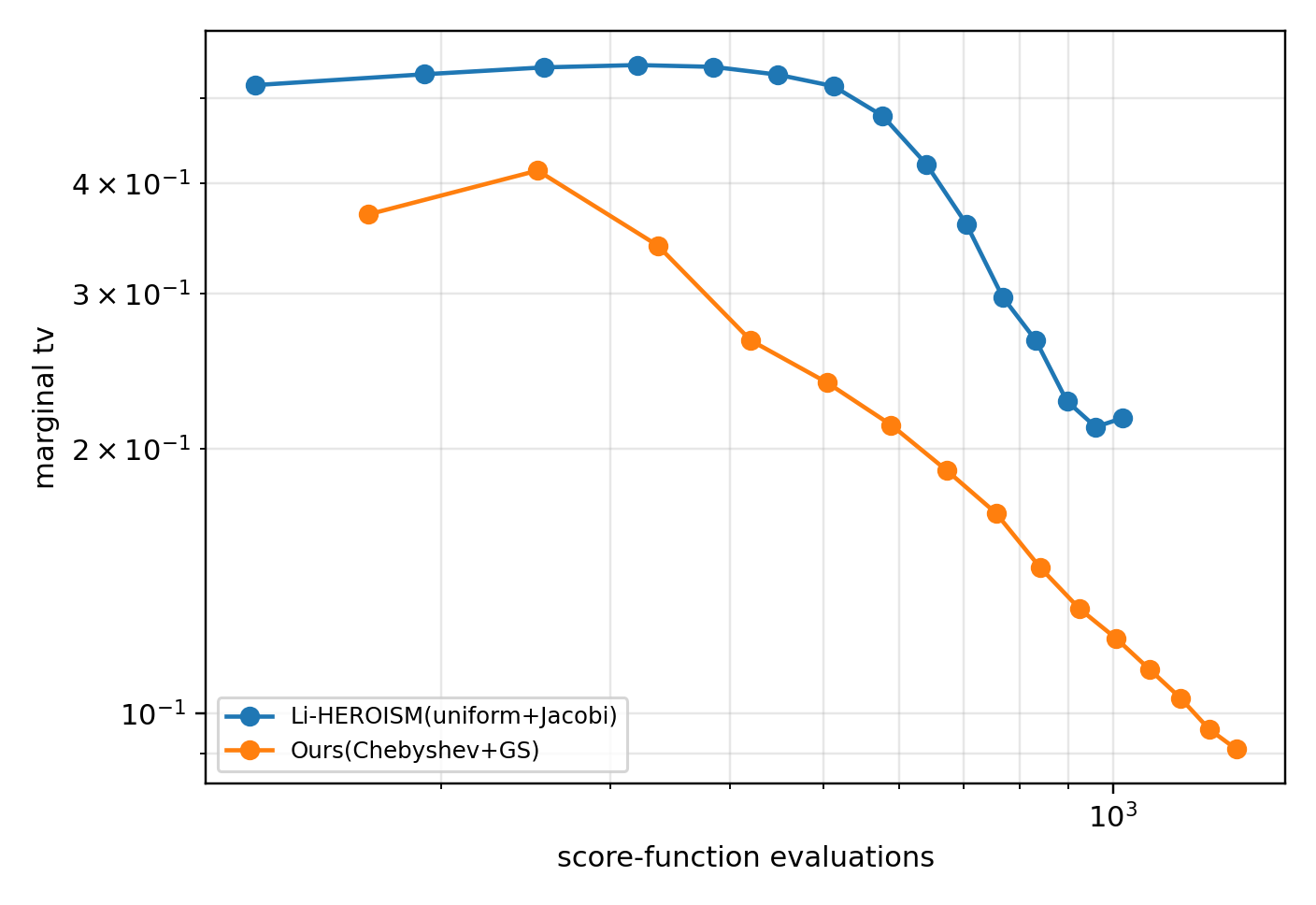}
\includegraphics[width=0.33\linewidth]{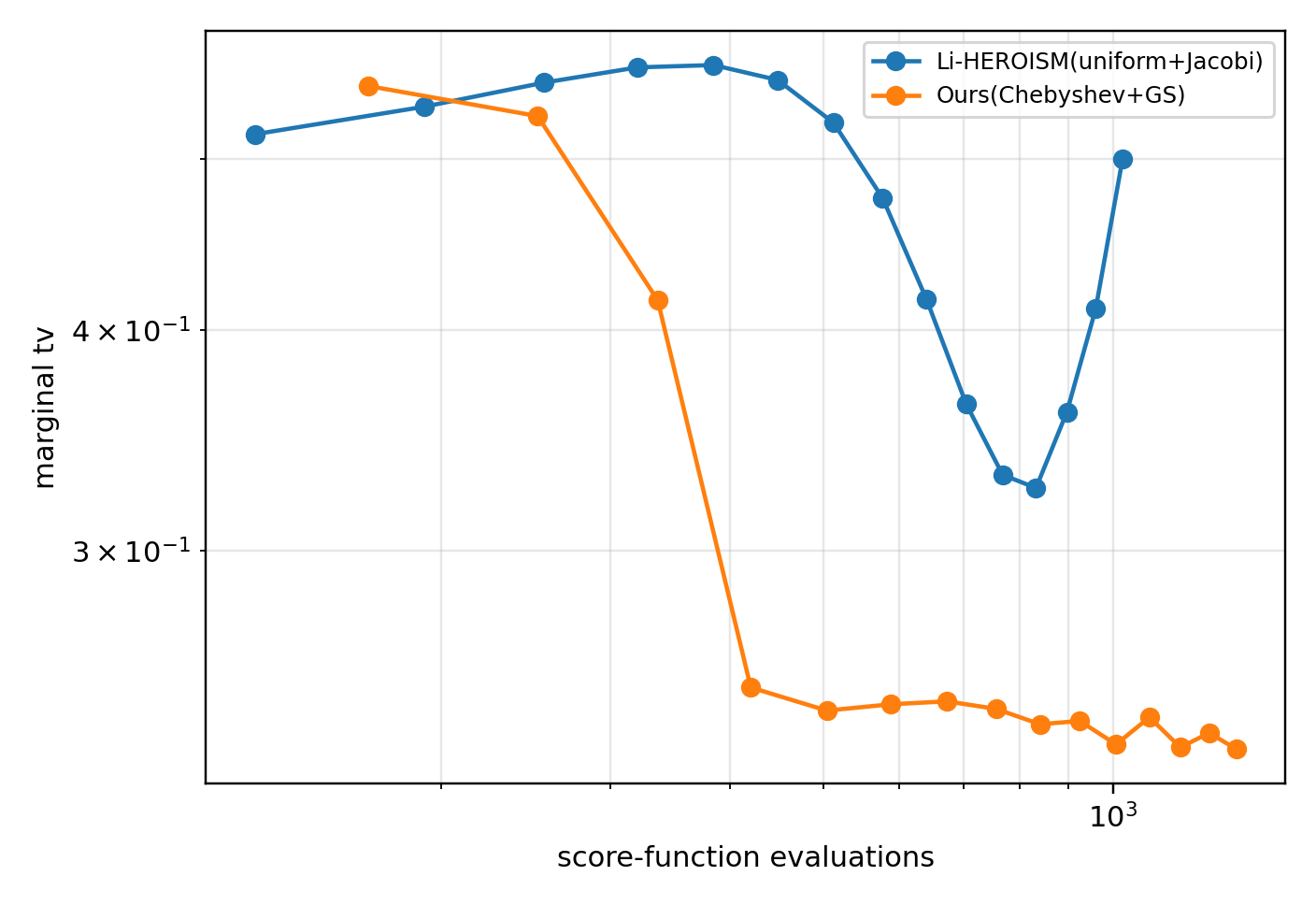}
\includegraphics[width=0.33\linewidth]{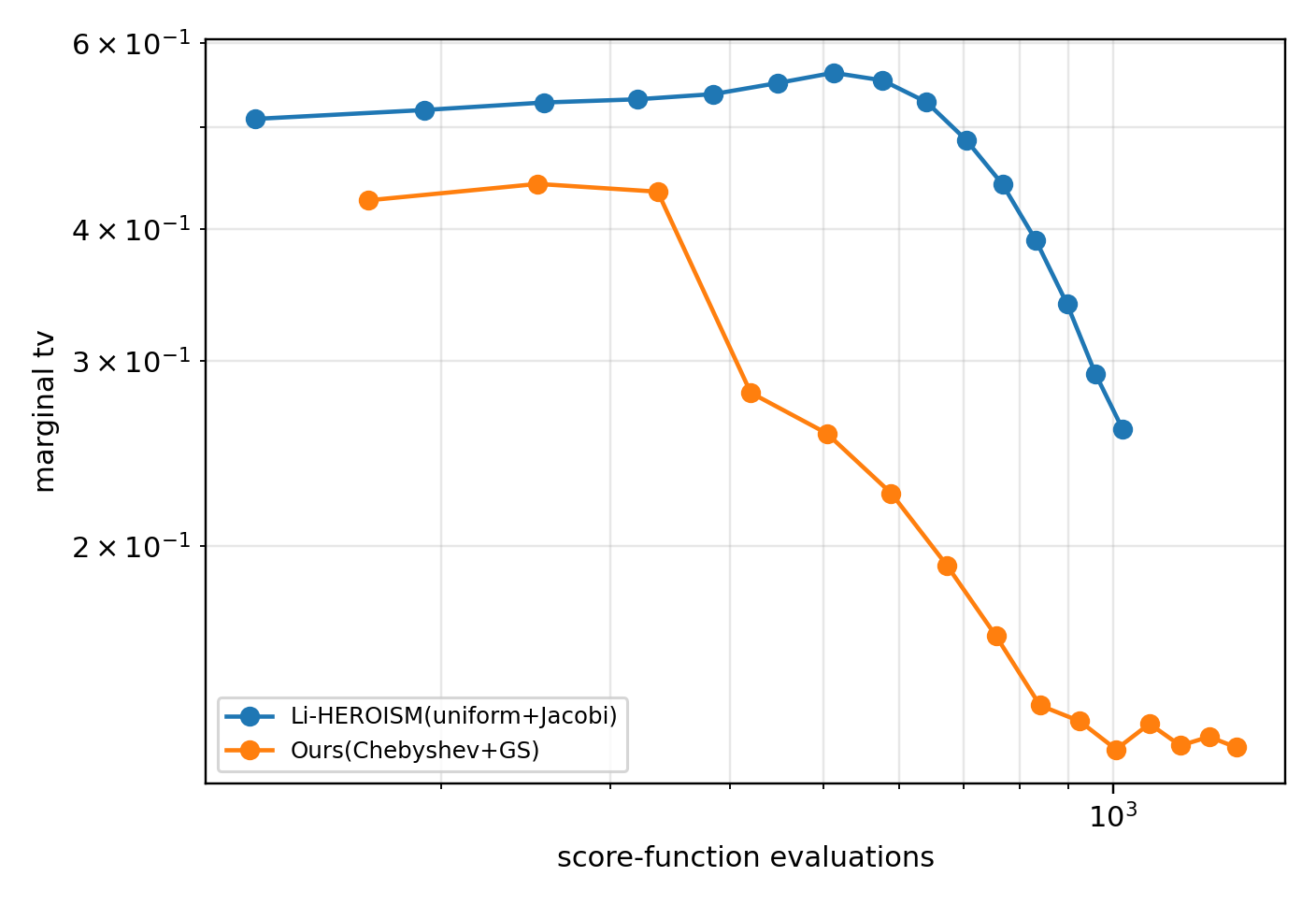}
\caption{One-dimensional total variation (TV) comparison with \(\delta=0.05\).
From left to right: $\eta_{\mathrm{const}}$, $\eta_{\mathrm{lin}}$, and $\eta_{\mathrm{sin}}$ score perturbations. Each
panel plots TV distance versus total $\mathrm{SFE}$ for Li-HEROISM
and the proposed GS sampler.}
\label{fig:num_1d_sfe}
\end{figure}
\subsection{Test $d=1$}
\label{subsec:num_1d}

We first consider the one-dimensional three-component Gaussian mixture
\eqref{eq:num_target_gmm} with
\[
w=[0.1 ; 0.4 ; 0.5],\qquad
m=[-6.0 ; 4.0 ; 6.0],\qquad
C=[0.25 ; 0.25 ; 0.25].
\]
This target follows the one-dimensional setting of Huang et
al.~\cite{huang2024pflow} and is used to assess whether the samplers can
accurately recover the locations and relative weights of well-separated modes.
We sample \(J=5\times 10^4\) particles from \(\mathcal{N}(0,1)\), fix \(K=6\)
and \(N=3\), and vary the number of outer iterations \(T\). Under the
\(\mathrm{SFE}\) convention in \eqref{eq:num_sfe_count}, this yields
\(\mathrm{SFE}_{\mathrm{Li}}=16T\) and
\(\mathrm{SFE}_{\mathrm{Ours}}=21T\).

For density visualization, we estimate the empirical density of the generated
samples at \(\tau_{\min}\), denoted by \(\widehat p_{\tau_{\min}}\), using a
kernel density estimator (KDE) with bandwidth selected according to Silverman's
rule~\cite{silverman2018density}. Since \(\tau_{\min}=10^{-3}\) approximates
the data endpoint, \(p_{\bar X_{\tau_{\min}}}\) and
\(\widehat p_{\tau_{\min}}\) serve as numerical proxies for \(q_1\) and
\(p_1\), respectively. Following the evaluation protocol of Huang et
al.~\cite{huang2024pflow}, we approximate the total variation distance on
\([-10,10]\) using a composite midpoint quadrature with \(1000\) subintervals.

Figure~\ref{fig:num_1d_sfe} shows that, under the same \(\mathrm{SFE}\)
budget, the proposed Gauss-Seidel sampler achieves a faster reduction in total variation
distance than Li-HEROISM across all three score-perturbation settings.
Figure~\ref{fig:num_1d_density} compares the estimated densities with the
reference density \(p_{\bar X_{\tau_{\min}}}\). Table~\ref{tab:num_1d_constant_twox_sfe}
further shows that our method often attains smaller total variation errors even
when Li-HEROISM is allowed twice as many \(\mathrm{SFEs}\), demonstrating the
efficiency of the proposed sampler.

\begin{figure}[H]
\centering
\includegraphics[width=0.48\linewidth]{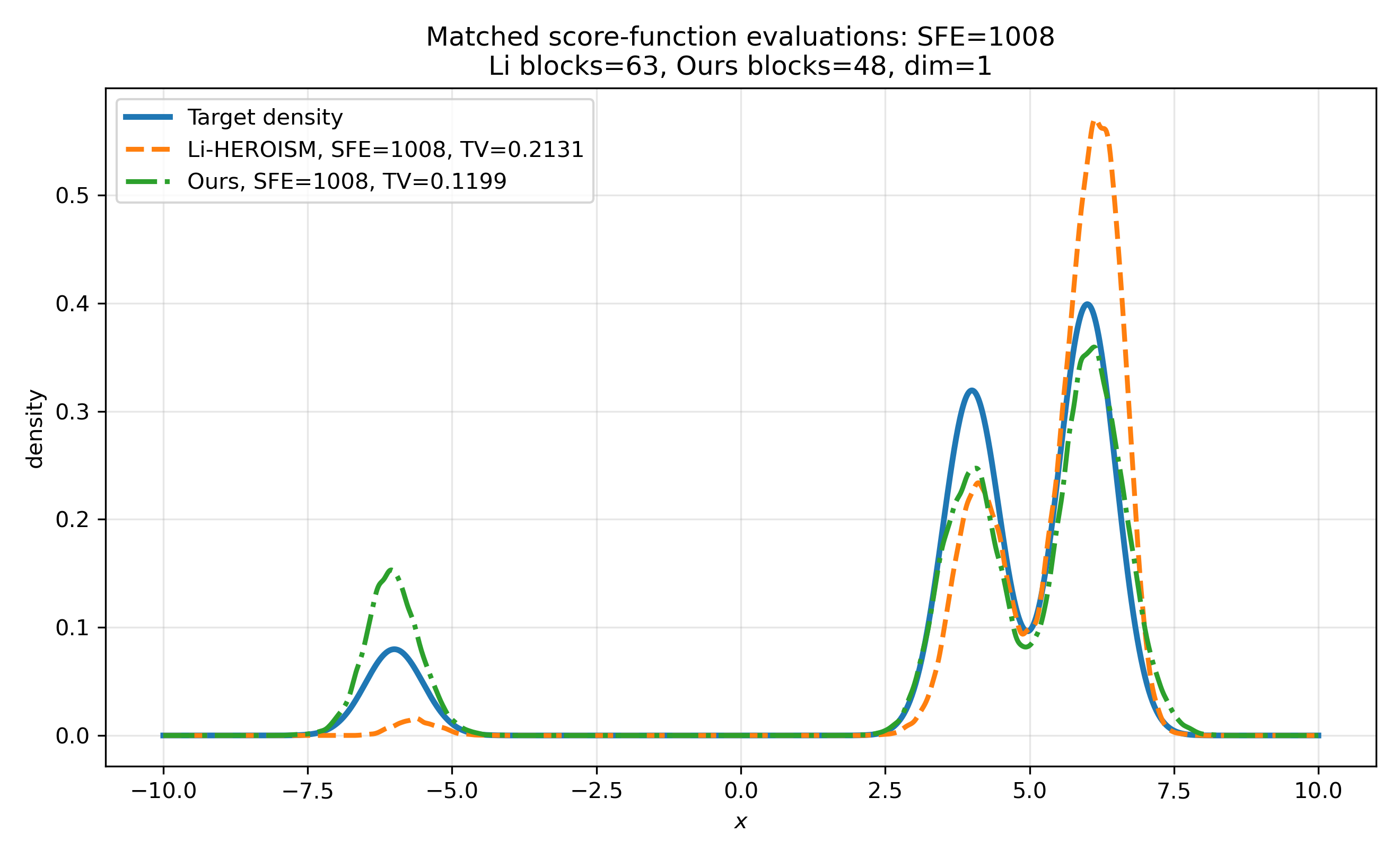}
\includegraphics[width=0.48\linewidth]{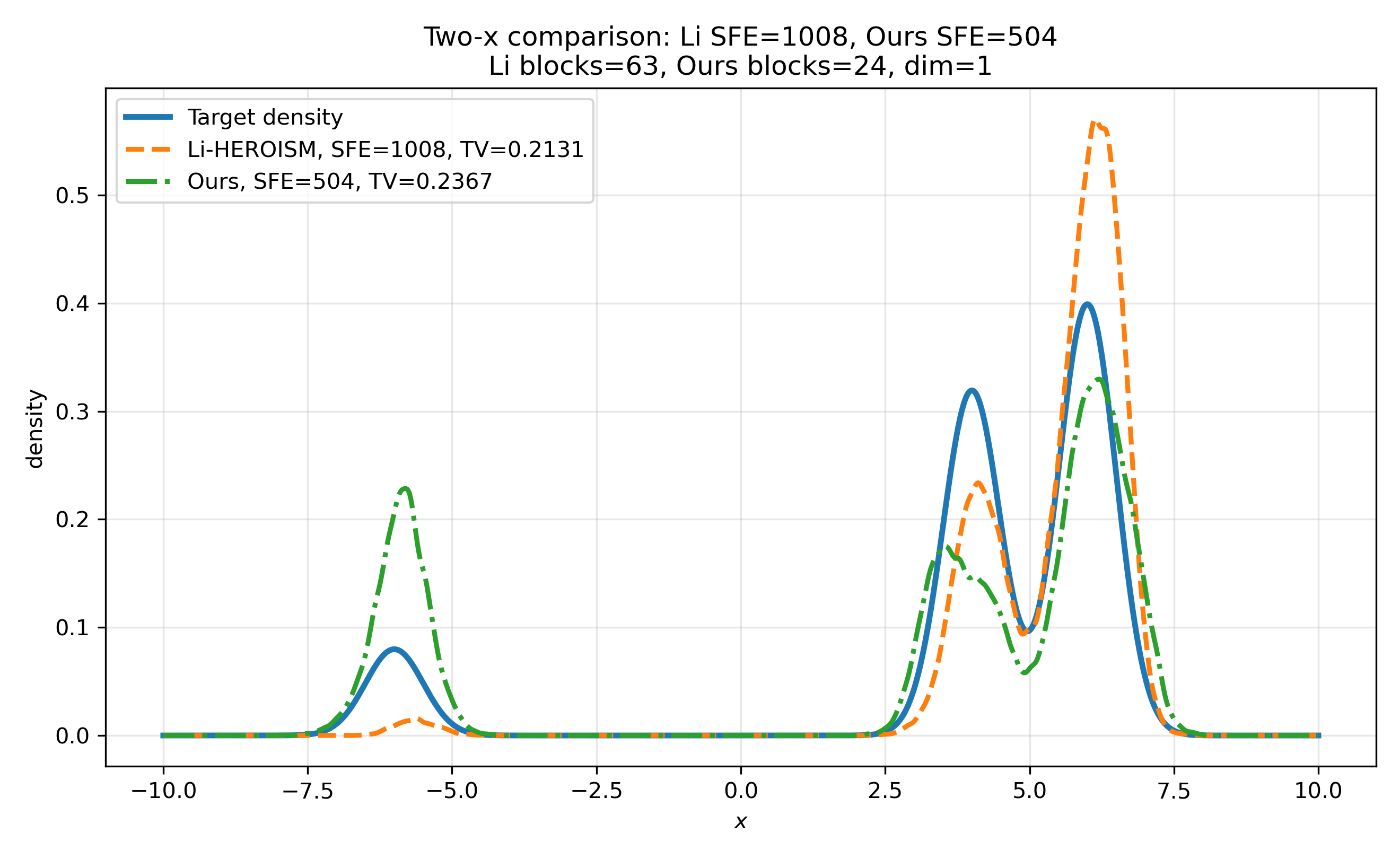}
\caption{One-dimensional density comparison under the \(\eta_{\mathrm{const}}\) score perturbation
 with \(\delta=0.05\).
Left: matched $\mathrm{SFE}$. Right: Li-HEROISM with twice the
$\mathrm{SFE}$ of the proposed Gauss-Seidel sampler.}
\label{fig:num_1d_density}
\end{figure}
\begin{table}[H]
\centering
\caption{One-dimensional experiment with $\eta_{\mathrm{const}}$ score perturbation, where
Li-HEROISM is allowed twice as many $\mathrm{SFEs}$ as the proposed
Gauss-Seidel method. The perturbation is
\(\eta(x)=\mathbf 1/\sqrt d\) with
\(\delta=0.05\). }
\label{tab:num_1d_constant_twox_sfe}
\small
\begin{tabular}{cccccc}
\toprule
\(T_{\mathrm{Ours}}\) &
$\mathrm{SFE}_{\mathrm{Ours}}$ &
\(T_{\mathrm{Li}}\) &
$\mathrm{SFE}_{\mathrm{Li}}$ &
$\mathrm{TV}_{\mathrm{Li}}$ &
$\mathrm{TV}_{\mathrm{Ours}}$ \\
\midrule
8  & \textbf{168} & 21 & \textbf{336}  & 0.5446 & \textbf{0.3689} \\
16 & \textbf{336} & 42 & \textbf{672}  & 0.3901 & \textbf{0.3401} \\
24 & \textbf{504} & 63 & \textbf{1008} & \textbf{0.2131} & 0.2367 \\
\bottomrule
\end{tabular}
\end{table}
\begin{figure}[H]
\centering
\includegraphics[width=0.33\linewidth]{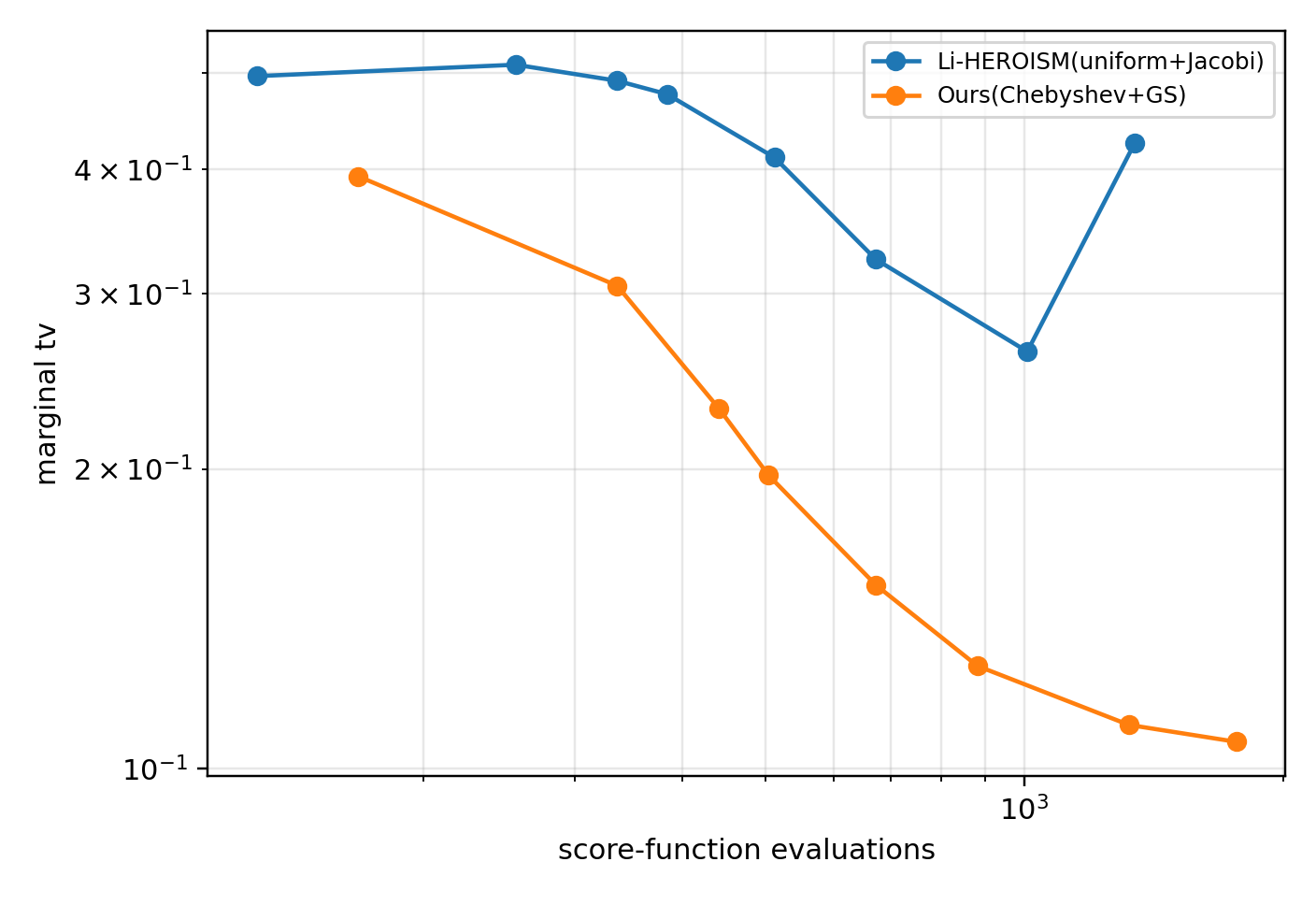}
\includegraphics[width=0.33\linewidth]{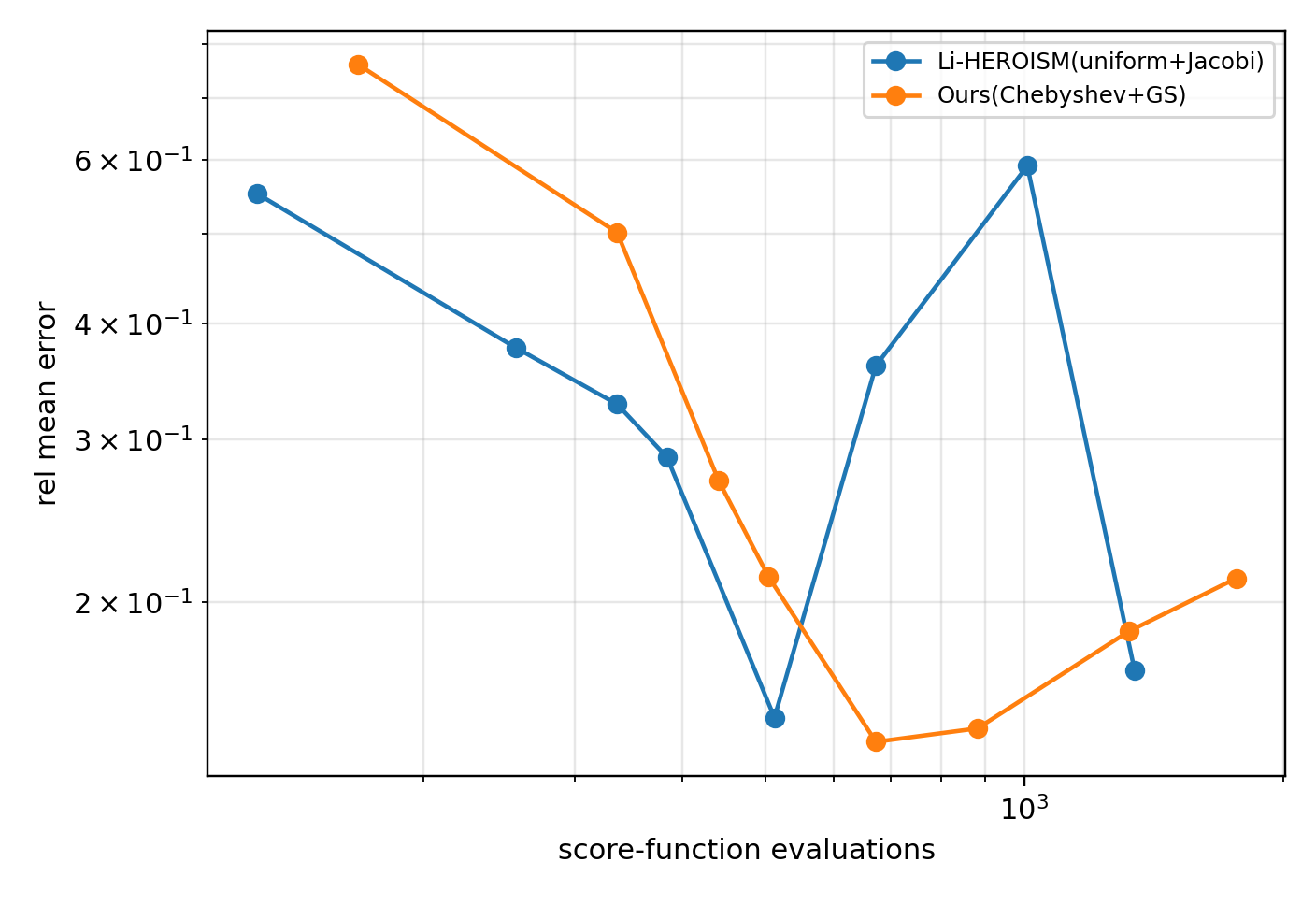}
\includegraphics[width=0.33\linewidth]{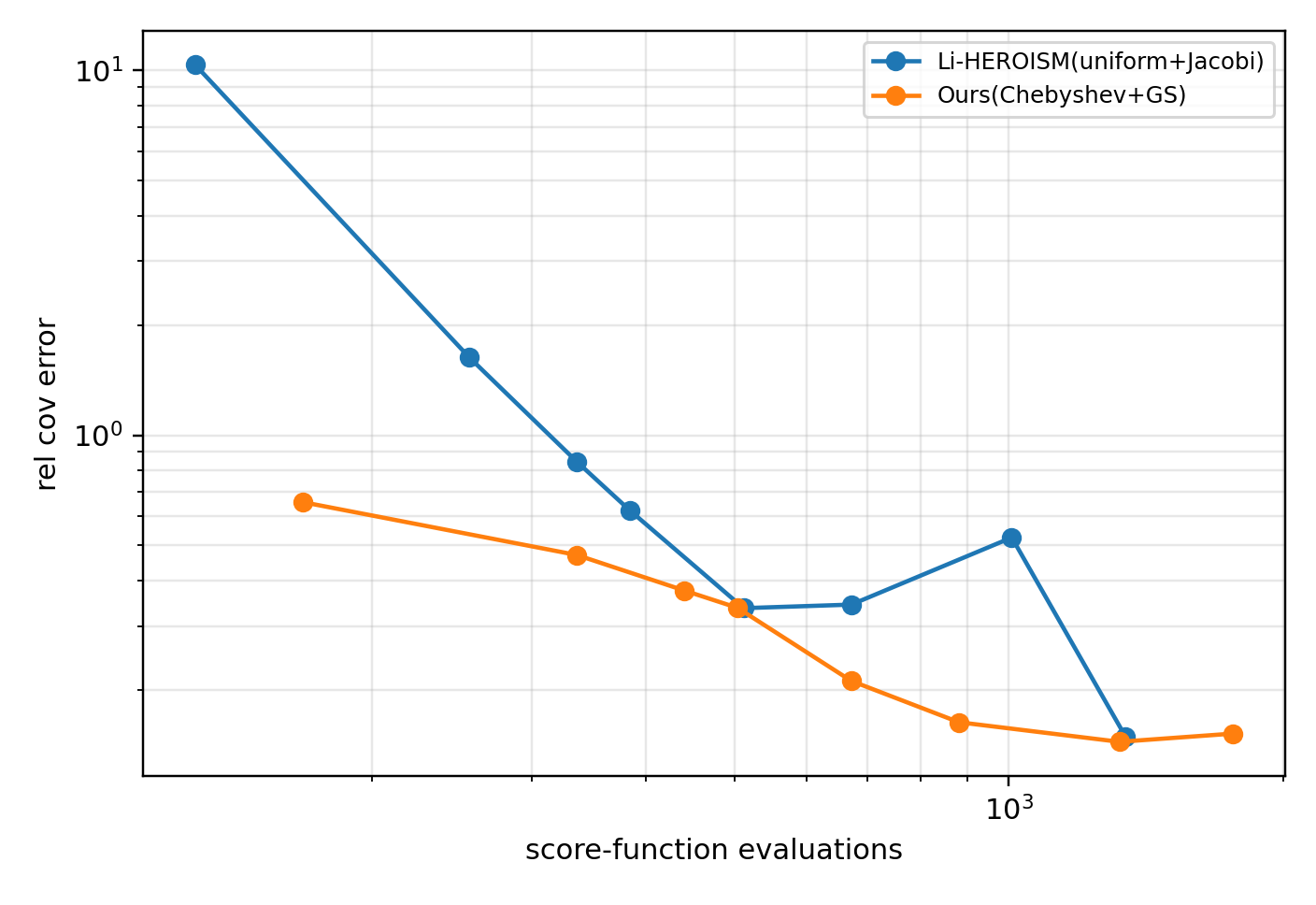}
\caption{High-dimensional experiment for \(d=128\) under the
\(\eta_{\mathrm{lin}}\) artificial score perturbation with \(\delta=0.05\). Error versus total SFE. Left: first-coordinate marginal total variation (TV). Middle: full-dimensional
relative mean error. Right: full-dimensional relative covariance error.}
\label{fig:num_d128_sfe}
\end{figure}
\subsection{Test $d=128$}
\label{subsec:num_highdim}

Following the high-dimensional Gaussian-mixture setting in Huang et
al.~\cite{huang2024pflow}, we next test a \(d=128\) anisotropic
Gaussian mixture \eqref{eq:num_target_gmm} with \(M=5\) components. The mixture weights $w_{\ell}$ are sampled from
\(\operatorname{Uniform}(0,1)\) and normalized. The distribution means and
covariance matrices are generated by
\[
m_\ell\sim \mathcal N(0,3^2I_d),
\qquad
C_\ell
=
\frac18
\left(
\frac{W_\ell^\top W_\ell}{d}
+
I_d
\right),
\qquad
(W_\ell)_{ij}\sim \mathcal N(0,1),
\qquad
W_\ell \in \mathbb{R}^{d \times d}
\]
with \(d=128\). 
We use the same algorithmic setup as in the one-dimensional test, except that
the number of particles is \(J=2\times 10^4\). Since full-dimensional total variation (TV) estimation from particles is unreliable in high
dimension, we follow Huang et al.~\cite{huang2024pflow} and compute the
first-coordinate marginal TV, estimated with the same KDE bandwidth and
fixed-grid quadrature rule as in the one-dimensional test.
To complement this one-dimensional marginal metric, we also compute the
full-dimensional relative mean error and relative covariance error:
\[
\mathrm{Err}_{\mathrm{mean}}
=
\frac{\|\widehat m-m_{\tau_{\min}}\|_2}{\|m_{\tau_{\min}}\|_2},
\qquad
\mathrm{Err}_{\mathrm{cov}}
=
\frac{\|\widehat\Sigma-\Sigma_{\tau_{\min}}\|_{\mathrm F}}
{\|\Sigma_{\tau_{\min}}\|_{\mathrm F}}.
\]
Here \(\widehat m\in\mathbb{R}^d\) and \(\widehat\Sigma\in\mathbb{R}^{d\times d}\) are the empirical mean and covariance
computed directly from the generated particles, while \(m_{\tau_{\min}}\) and
\(\Sigma_{\tau_{\min}}\) are the exact mean and covariance of
\(p_{\bar X_{\tau_{\min}}}\).
Figure~\ref{fig:num_d128_sfe} shows that the proposed sampler reduces the
marginal TV more steadily and substantially faster than Li-HEROISM, with more
stable covariance behavior. Although the relative mean error is mixed, it is
only an auxiliary moment diagnostic. Table~\ref{tab:num_d128_twox_sfe} further
shows that our method attains smaller marginal TV in all tested cases even when
Li-HEROISM uses twice as many $\mathrm{SFEs}$, demonstrating its
finite-budget efficiency.
\begin{table}[H]
\centering
\caption{High-dimensional experiment when Li-HEROISM uses twice as many
score function evaluations as the proposed Gauss--Seidel sampler. The score
perturbation is $\eta_{\mathrm{lin}}$ with \(\delta=0.05\).}
\label{tab:num_d128_twox_sfe}
\begin{tabular}{ccccccccc}
\toprule
\(T_{\mathrm{Ours}}\) & $\mathrm{SFE}_{\mathrm{Ours}}$ &
\(T_{\mathrm{Li}}\) & $\mathrm{SFE}_{\mathrm{Li}}$ &
$\mathrm{TV}_{1,\mathrm{Li}}$ & $\mathrm{TV}_{1,\mathrm{Ours}}$ &
$\mathrm{MeanErr}_{\mathrm{Li}}$ & $\mathrm{MeanErr}_{\mathrm{Ours}}$ &
$\mathrm{CovErr}_{\mathrm{Li}}$ / $\mathrm{CovErr}_{\mathrm{Ours}}$ \\
\midrule
8  & 168 & 21 & 336  & 0.4911 & \textbf{0.3936} & 0.3271 & 0.7598 & 0.8456 / \textbf{0.6556} \\
16 & 336 & 42 & 672  & 0.3248 & \textbf{0.3056} & 0.3600 & 0.5010 & 0.3434 / 0.4695 \\
24 & 504 & 63 & 1008 & 0.2624 & \textbf{0.1971} & 0.5911 & \textbf{0.2132} & 0.5244 / \textbf{0.3365} \\
32 & 672 & 84 & 1344 & 0.4251 & \textbf{0.1528} & 0.1688 & \textbf{0.1415} & 0.1491 / 0.2121 \\
\bottomrule
\end{tabular}
\end{table}

\section{Discussion and future directions}
\label{sec:discussion}

In this paper, we develop a Chebyshev--Gauss--Seidel higher-order sampler for
training-free acceleration of diffusion models through the probability flow ODE,
and establish a convergence guarantee that allows the approximation order \(K\)
to grow with the number of iterations \(T\). In the exact-score setting, by
taking \(K \asymp \log T\), the proposed sampler achieves a target total
variation accuracy \(\varepsilon\) using at most
\[
\widetilde O\!\left(d^{1+o_T(1)}\varepsilon^{-1/K_1}\right)
\]
score function evaluations, where \(o_T(1)\to 0\) as \(T\to\infty\) and
\(K_1>0\) is an absolute constant. Our guarantee holds under a polynomial
second-moment condition on the target distribution and accommodates inexact
score estimation, without imposing higher-order smoothness assumptions on the
score estimates.

Several questions remain open. First, it would be interesting to understand
whether the intrinsic low-dimensional structure of the target distribution can
be leveraged to further accelerate the proposed sampler. Second, although our
analysis avoids higher-order smoothness assumptions on the score estimates, it
still requires control of the Jacobian estimation error; an important direction
for future work is to determine whether this requirement can be weakened.
Finally, our numerical results indicate that interpolation nodes and
successive-refinement schemes play a significant role in the performance of
higher-order samplers. This naturally raises the question of whether a unified
non-asymptotic theory can rigorously explain the impact of these algorithmic
design choices.

%

\appendix
\section{Auxiliary Lemmas}
\label{app:auxiliary_estimates}

This section gathers the auxiliary estimates that will be used repeatedly in the theoretical analysis and in the subsequent density-ratio propagation arguments.
\begin{lemma}
\label{lemma8}
Assume that \(K \ge 2\), and let
\[
z_j = \cos\!\left(\frac{j\pi}{K-1}\right), \qquad 0 \le j \le K-1,
\]
denote the Chebyshev--Lobatto nodes on \([-1,1]\). Let \(\tau_{t,j}\) be
defined as in \eqref{eq:nodes}, and let \(\psi_{t,j}\) be given by
\eqref{eq:basicfun}. Then
\[
\sup_{\tau \in [\tau_{t,K-1}, \tau_{t,0}]}
\sum_{j=0}^{K-1} |\psi_{t,j}(\tau)|
\le \frac{2}{\pi}\log K + 1.
\]
\end{lemma}

\begin{proof}
The proof is given in Section~\ref{lemma8_proof}.
\end{proof}

\begin{lemma}  \label{lemma9}
For large enough $T$, one has
\begin{subequations}
\begin{equation}
\alpha_t \ge 1-\frac{c_1\log T}{T}\ge \frac12,
\qquad 1\le t\le T.
\label{eq:23a}
\end{equation}
\begin{equation}
\frac12\frac{1-\alpha_t}{1-\bar{\alpha}_t}
\le
\frac12\frac{1-\alpha_t}{\alpha_t-\bar{\alpha}_t}
\le
\frac{1-\alpha_t}{1-\bar{\alpha}_{t-1}}
\le
\frac{4c_1\log T}{T},
\qquad 2\le t\le T.
\label{eq:23b}
\end{equation}
\begin{equation}
1
\le
\frac{1-\bar{\alpha}_t}{1-\bar{\alpha}_{t-1}}
\le
1+\frac{4c_1\log T}{T},
\qquad 2\le t\le T.
\label{eq:23c}
\end{equation}
\begin{equation}
\bar{\alpha}_T \le \frac{1}{T^{C_1}}.
\label{eq:23d}
\end{equation}
\begin{equation}
\frac{\bar{\alpha}_{t+1}}{1-\bar{\alpha}_{t+1}}
\le
\frac{\bar{\alpha}_t}{1-\bar{\alpha}_t}
\le
\frac{4\bar{\alpha}_{t+1}}{1-\bar{\alpha}_{t+1}},
\qquad 1\le t<T.
\label{eq:23e}
\end{equation}
\begin{equation}
\left|
\frac{\tau_{t,i_1}-\tau_{t,i_2}}{\tau_{t,i_3}(1-\tau_{t,i_4})}
\right|
\le
8c_1\frac{\log T}{T},
\qquad
2\le t\le T,\quad
0\le i_1,i_2,i_3,i_4\le K-1.
\label{eq:23f}
\end{equation}
\begin{equation}
|\gamma_{t,j}(\tau_{t,i})|
\le
2(\tau_{t,0}-\tau_{t,i}),
\qquad 0\le i,j\le K-1.
\label{eq:23g}
\end{equation}
\begin{equation}
1-\tau_{t,i}\asymp 1-\tau_{t,j},
\qquad
\tau_{t,i}\asymp \tau_{t,j},
\qquad
0\le i,j\le K-1,\quad 2\le t\le T.
\label{eq:23h}
\end{equation}
\begin{equation}
\sum_{j=0}^{K-1}| \gamma_{t,j}(\tau_{t,i})|
\le
C_4(\tau_{t,0}-\tau_{t,i})\log K,
\qquad 0\le i\le K-1.
\label{eq:23i}
\end{equation}
\begin{equation}
\sum_{j=0}^{K-1}|A_{ji}^{(t)}|
\le
C_4
\frac{(\tau_{t,0}-\tau_{t,i})\log K}{(1-\tau_{t,0})^{3/2}},
\qquad
0\le i\le K-1,\quad 2\le t\le T.
\label{eq:23j}
\end{equation}
\begin{equation}
\sum_{j=0}^{K-1}|A_{ji}^{(t)}|^2
\le
C_5
\frac{(\tau_{t,0}-\tau_{t,i})^2\log K}{(1-\tau_{t,0})^{3}},
\qquad
0\le i\le K-1,\quad 2\le t\le T.
\label{eq:23k}
\end{equation}
\end{subequations}
Here, $\gamma_{t,j}(\tau_{t,i})$ is given in \eqref{eq:gammtk},  $c_1$ is defined in \eqref{eq:learnrate}, $C_1>0$ is a sufficiently large constant, $C_4, C_5>0$ are universal constants, and $A_{j i}^{(t)}:=\frac{1}{2} \gamma_{t, j}\left(\tau_{t, i}\right)\left(1-\tau_{t, j}\right)^{-3 / 2}$ for convenience. 
\end{lemma}
\begin{proof}
See Section~\ref{lemma9_proof}.
\end{proof}

Define
\begin{equation}    \label{assu_1}
    \theta_\tau(x):=\max \left\{-\frac{\log p_{\bar{X}_\tau}(x)}{d \log T}, c_6\right\},
\end{equation}
where $c_6>0$ is a sufficiently large constant.  The next lemma provides a quantitative tail estimate for the conditional distribution of $X_0$ given the continuous-time forward variable $\bar{X}_\tau$, where $\bar{X}_\tau$ is given in \eqref{def_1}.

\begin{lemma}
\label{lemma10}
Suppose that $\tau\ge T^{-c_0}$ where $c_0$ is defined in \eqref{eq:beta1}. Assume that  Assumption~\ref{assup:1} holds.
Then for any constant $c_5\ge 2$,  one has
\[
\mathbb{P}\!\left(
\left\|\sqrt{1-\tau}X_0-y\right\|_2
>
5c_5\sqrt{\theta_\tau(y)\,d\,\tau\log T}
\,\middle|\, \bar{X}_\tau=y
\right)
\le
\exp\!\left(-c_5^2\theta_\tau(y)d\log T\right).
\]
Moreover, there exist absolute constants $C_6,C_7,C_8,C_9>0$ such that
\begin{align}
\mathbb{E}\!\left[\left\|\sqrt{1-\tau}X_0-y\right\|_2 \,\middle|\, \bar{X}_\tau=y\right]
&\le
C_6\sqrt{\theta_\tau(y)\,d\,\tau\log T}, \label{eq:mod_lemma2_ct_m1}\\
\mathbb{E}\!\left[\left\|\sqrt{1-\tau}X_0-y\right\|_2^2 \,\middle|\, \bar{X}_\tau=y\right]
&\le
C_7\,\theta_\tau(y)\,d\,\tau\log T, \label{eq:mod_lemma2_ct_m2}\\
\mathbb{E}\!\left[\left\|\sqrt{1-\tau}X_0-y\right\|_2^3 \,\middle|\, \bar{X}_\tau=y\right]
&\le
C_8\bigl(\theta_\tau(y)\,d\,\tau\log T\bigr)^{3/2}, \label{eq:mod_lemma2_ct_m3}\\
\mathbb{E}\!\left[\left\|\sqrt{1-\tau}X_0-y\right\|_2^4 \,\middle|\, \bar{X}_\tau=y\right]
&\le
C_9\bigl(\theta_\tau(y)\,d\,\tau\log T\bigr)^2. \label{eq:mod_lemma2_ct_m4}
\end{align}
\end{lemma}
\begin{proof}
See Section~\ref{lemma10_proof}.
\end{proof}

By Tweedie's formula \cite{efron2011tweedie}, the score function \eqref{score_define} can also be expressed as
\begin{equation} \label{eq:scorefin}
    s_\tau^{\star}(x)=-\frac{1}{\tau}\left(x-\sqrt{1-\tau} \mathbb{E}\left[X_0 \mid \bar{X}_\tau=x\right]\right),
\end{equation}
which further implies
\begin{equation*}
    s_\tau^{\star}(\sqrt{1-\tau} x)=-\frac{\sqrt{1-\tau}}{\tau} \int_{x_0} p_{X_0 \mid \bar{X}_\tau}\left(x_0 \mid \sqrt{1-\tau} x\right)\left(x-x_0\right) \mathrm{d} x_0.
\end{equation*}
Based on Lemma \ref{lemma10}, it is straightforward to obtain the following results.
\begin{lemma}
\label{lemma11}
For any $\tau\ge T^{-c_0}$ where $c_0$ is defined in \eqref{eq:beta1}, one has
\begin{align}
\|s_\tau^\star(x)\|_2
&\lesssim
\,\sqrt{\frac{d\,\theta_\tau(x)\log T}{\tau}},
\label{eq:lemma11_score_bound}\\
\left\|\frac{\partial s_\tau^\star(x)}{\partial x}\right\|
&\lesssim
\,\frac{d\,\theta_\tau(x)\log T}{\tau}.
\label{eq:lemma11_score_jac_bound}
\end{align}
Moreover, under the condition
\[
-\log p_{\bar X_\tau}\bigl(\lambda x_1+(1-\lambda)x_2\bigr)\lesssim \,d\log T,
\qquad \forall \lambda\in[0,1],
\]
one has
\begin{equation}
\left\|
\frac{\partial s_\tau^\star(x_1)}{\partial x}
-
\frac{\partial s_\tau^\star(x_2)}{\partial x}
\right\|
\lesssim
\,\sqrt{\frac{d^3\log^3 T}{\tau^3}}\,\|x_1-x_2\|_2 .
\label{eq:lemma11_score_jac_lip_bound}
\end{equation}
\end{lemma}
\begin{proof}
The first estimate follows directly from Tweedie's formula  and the first-moment bound in Lemma~\ref{lemma10}. The second estimate follows from the \eqref{eq:scorefin} together with the second-moment bound in Lemma~\ref{lemma10}. The third estimate is standard and follows, for instance, from the argument of \cite[Claim~(88)]{li2024sharp}.
\end{proof}


\begin{lemma} \cite[Lemma 4]{li2025faster} \label{lemma12}
Suppose that $-\log p_{\bar X_{\tau_e}}(x_{\tau_e}^\star)\le \theta d\log T$  for some $T^{-c_0} \le  \tau_e <1$ where $c_0$ is defined in \eqref{eq:beta1} and some $\theta>1$, then for every $\tau'$ satisfying $|\tau'-\tau_e|\le c_4\tau_e(1-\tau_e)$, one has
\[
-\log p_{\bar X_{\tau'}}(x_{\tau'}^\star)\le 2\theta d\log T.
\]
Here, $\bar{X}_\tau$ is given in \eqref{def_1} and $c_4>0$ is some sufficiently small constant.
\end{lemma}


\begin{lemma} \label{lemma13}
Assume that $\tau\in[\tau_{t,K-1},\tau_{t,0}]$ and  $-\log p_{\bar{X}_\tau}(x_\tau^\star)\le \theta\, d\log T$
for some $\theta\ge c_6$ where $c_6$ is defined in \eqref{assu_1}.  Set
\[
u_\tau^\star:=\frac{x_\tau^\star}{\sqrt{1-\tau}} \quad \mbox{and} \quad J_\tau^\star := \frac{\partial u_\tau^\star}{\partial u_{\tau_{t,0}}^\star} = \frac{\partial (x_\tau^\star/\sqrt{1-\tau})} {\partial (x_{\tau_{t,0}}^\star/\sqrt{1-\tau_{t,0}})}.
\]
Then there exist universal constants $M,C_{\mathrm s},C_{\mathrm J}>0$,
independent of $k,K,\tau,d,T,\theta$ where $M=\max\{M_1,M_2\}$ in proof, such that for every integer
$0\le k\le K$,
\begin{align}
&\left\|
\frac{\partial^k}{\partial\tau^k}
\left(
\frac{s_\tau^\star(x_\tau^\star)}{(1-\tau)^{3/2}}
\right)
\right\|_2
\le
C_{\mathrm s}\,(MK)^k k!\,
\sqrt{\frac{d\theta\log T}{\tau(1-\tau)^3}}
\left(
\frac{d\theta\log T}{\tau(1-\tau)}
\right)^k, \label{eq:lemma131}
\\[1ex]
&\left\|
\frac{\partial^k}{\partial\tau^k}
\left[
\frac{1}{(1-\tau)^{3/2}}
\frac{\partial s_\tau^\star(x_\tau^\star)}
{\partial u_\tau^\star}\,
J_\tau^\star
\right]
\right\|
\le
C_{\mathrm J}\,(MK)^{k+1} k!\,
\left(
\frac{d\theta\log T}{\tau(1-\tau)}
\right)^{k+1}. \label{eq:lemma132}
\end{align}
\end{lemma}
\begin{proof}
See Section~\ref{proof_lemma13}.
\end{proof}


\begin{lemma}  \label{lemma14}
Let
$\theta_t:=\theta_{\tau_{t,0}}(x_{\tau_{t,0}})$, where $\theta_{\tau_{t,0}}(x_{\tau_{t,0}})$ is defined in \eqref{assu_1}.
Assume that
\begin{equation} \label{eq:lemma14_condition_here}
C_{10}\left\{
\frac{\theta_t d\log^2 T}{T}
+
\frac{\sqrt{\theta_t d\log^3 T \log K\sum_{m,j}\left(\varepsilon_{\mathrm{score}, t, j}^{(m)}\left(x_{\tau_{t, j}}^{(m)}\right)\right)^2}\,
}{T}
\right\}\le 1
\end{equation}
for some sufficiently large constant $C_{10}>0$, where $\sum_{m,j}:=\sum_{m=0}^{N}\sum_{j=0}^{K-1}$ and $\varepsilon_{\mathrm{score}, t, j}^{(m)}(\cdot)$ is defined in \eqref{def_vareps}.
Then for all $0\le i\le K-1$, $0\le n\le N-1$, and all $\lambda\in[0,1]$,
\begin{equation}
-\log p_{\bar{X}_{\tau_{t,i}}}
\Bigl(
\lambda x_{\tau_{t,i}}^{(n+1)}+(1-\lambda)x_{\tau_{t,i}}^\star
\Bigr)
\le 2.1\,d\theta_t\log T,
\label{eq:lemma14_typical_line_final}
\end{equation}
and
\begin{equation*}
\log
\frac{
p_{\sqrt{\frac{1-\tau_{t,0}}{1-\tau_{t,i}}}\bar{X}_{\tau_{t,i}}}
\!\left(
\sqrt{\frac{1-\tau_{t,0}}{1-\tau_{t,i}}}\,x_{\tau_{t,i}}^{(n+1)}
\right)
}{
p_{\bar{X}_{\tau_{t,0}}}(x_{\tau_{t,0}})
}
\le
\frac{4c_1 d\log T}{T}
+
C_{10}\left\{
\frac{d^2\theta_t^2\log^4 TK}{T^2}
+
\frac{
\sqrt{d\theta_t\log^3T\,\log K\sum_{m,j}\left(\varepsilon_{\mathrm{score}, t, j}^{(m)}\left(x_{\tau_{t, j}}^{(m)}\right)\right)^2}\;
}{T}
\right\}.
\end{equation*}
\end{lemma}
\begin{proof}
See Section~\ref{proof_lemma14}.
\end{proof}


\begin{lemma} \label{lemma15}
Recall that $q_t$ is the distribution of $X_t$, where $X_t$ in \eqref{eq:relas}. If
$-\log q_1(x_1)\le c_6 d\log T$ where $c_6$ is defined in \eqref{assu_1} and $T\ge C_2d\log^4 T$,
then for all integers $1\le \ell<\tau(x_T)$ where $\tau(x_T)$ is defined in \eqref{eq:def_tau_xt}, one has
\begin{equation}
-\log q_\ell(x_\ell)\le 2c_6 d\log T,
\label{eq:lemma15_conclusion_here}
\end{equation}
provided that $c_6>5c_1$.
\end{lemma}
\begin{proof}
The proof of Lemma~\ref{lemma15} follows the same argument as that of
Lemma~9 in \cite{li2025faster}, with the only modification being that we invoke
Lemma~\ref{lemma14} in place of the corresponding estimate used there. This
change requires the conditions \(T \ge C_2 d\log^4 T\) and \(K \lesssim \log T\)
to obtain the desired conclusion. We omit the details for brevity.
\end{proof}

\section{Proofs of lemmas in Section \ref{sec:promainresult}}
\label{app:proof_local_approximation}
For $\tau_{t-1,0} \leq \tau<\tau_{t,0}$, we let $x_\tau^\star$ denote the solution of \eqref{ODE} with the initial condition
$x_{\tau_{t,0}}^\star=x_{\tau_{t,0}}$. In particular, with iteration \eqref{eq:gs_update} we obtain
\[
x_{\tau_{t,0}}^{(n)}=x_{\tau_{t,0}}^\star=x_{\tau_{t,0}},
\qquad\text{for all }n\ge 0.
\]
For convenience, we define
$A_{ji}^{(t)}
:=
\frac{1}{2}\gamma_{t,j}\left(\tau_{t,i}\right)
\left(1-\tau_{t,j}\right)^{-3/2}$, where
$\gamma_{t,j}\left(\tau_{t,i}\right)
$ is defined in \eqref{eq:gammtk}.

\subsection{Proof of Lemma~\ref{lemma3}}
\label{lemma3_proof}
\begin{proof}
Fix $0\le i\le K-1$.
By the Gauss--Seidel update \eqref{eq:gs_update},
\[
\frac{x_{\tau_{t,i}}^{(n+1)}}{\sqrt{1-\tau_{t,i}}}
=
\frac{x_{\tau_{t,0}}}{\sqrt{1-\tau_{t,0}}}
+
\sum_{j<i} A_{ji}^{(t)}\, s_{\tau_{t,j}}\!\left(x_{\tau_{t,j}}^{(n+1)}\right)
+
\sum_{j\ge i} A_{ji}^{(t)}\, s_{\tau_{t,j}}\!\left(x_{\tau_{t,j}}^{(n)}\right).
\]
On the other hand, by integrating the rescaled probability flow ODE along the exact path, we obtain
\begin{align*}
\frac{x_{\tau_{t,i}}^\star}{\sqrt{1-\tau_{t,i}}}
=
\frac{x_{\tau_{t,0}}^\star}{\sqrt{1-\tau_{t,0}}}
+
\int_{\tau_{t,i}}^{\tau_{t,0}}
\frac{s_\tau^\star(x_\tau^\star)}{2(1-\tau)^{3/2}}
d\tau .
\end{align*}
Subtracting the above two identities yields
\begin{align}
x_{\tau_{t,i}}^{(n+1)}-x_{\tau_{t,i}}^\star
&=
\underbrace{\sqrt{1-\tau_{t,i}}\Biggl[\sum_{j<i} A_{ji}^{(t)}
\Bigl(
s_{\tau_{t,j}}\!\left(x_{\tau_{t,j}}^{(n+1)}\right)
-
s_{\tau_{t,j}}^\star\!\left(x_{\tau_{t,j}}^\star\right)
\Bigr)
+
\sum_{j\ge i} A_{ji}^{(t)}
\Bigl(
s_{\tau_{t,j}}\!\left(x_{\tau_{t,j}}^{(n)}\right)
-
s_{\tau_{t,j}}^\star\!\left(x_{\tau_{t,j}}^\star\right)
\Bigr)\Biggr]}_{:=G_1}
\nonumber\\
&\quad+\underbrace{\sqrt{1-\tau_{t,i}}\Biggl[\sum_{j=0}^{K-1}A_{ji}^{(t)}
s_{\tau_{t,j}}^\star(x_{\tau_{t,j}}^\star)-
\int_{\tau_{t,i}}^{\tau_{t,0}}
\frac{s_\tau^\star(x_\tau^\star)}{2(1-\tau)^{3/2}}
d\tau \Biggr]}_{:=G_2}.
\label{eq:gs_stage1_main}
\end{align}

\paragraph{Bound for $G_1$.} A simple calculation gives
\begin{align}
G_1
=&\underbrace{\sqrt{1-\tau_{t,i}}\Biggl[\sum_{j<i} A_{ji}^{(t)}
\Bigl(
s_{\tau_{t,j}}\!\left(x_{\tau_{t,j}}^{(n+1)}\right)
- s_{\tau_{t,j}}^\star\!\left(x_{\tau_{t,j}}^{(n+1)}\right)
\Bigr)
+ \sum_{j\ge i} A_{ji}^{(t)}
\Bigl(
s_{\tau_{t,j}}\!\left(x_{\tau_{t,j}}^{(n)}\right)
- s_{\tau_{t,j}}^\star\!\left(x_{\tau_{t,j}}^{(n)}\right)
\Bigr)\Biggr]}_{:=G_{11}}\nonumber\\
+&\underbrace{\sqrt{1-\tau_{t,i}}\Biggl[\sum_{j<i} A_{ji}^{(t)}
\Bigl(
s_{\tau_{t,j}}^\star\!\left(x_{\tau_{t,j}}^{(n+1)}\right)
-
s_{\tau_{t,j}}^\star\!\left(x_{\tau_{t,j}}^\star\right)
\Bigr)
+
\sum_{j\ge i} A_{ji}^{(t)}
\Bigl(
s_{\tau_{t,j}}^\star\!\left(x_{\tau_{t,j}}^{(n)}\right)
-
s_{\tau_{t,j}}^\star\!\left(x_{\tau_{t,j}}^\star\right)
\Bigr)\Biggr]}_{:=G_{12}}.
\label{80}
\end{align}
For the term $G_{11}$, the triangle inequality and Cauchy--Schwarz yield
\begin{eqnarray*}
\|G_{11}\|_2^2 &\le & 
(1-\tau_{t,i}) \Bigl(\sum_{j=0}^{K-1}|A_{ji}^{(t)}|^2\Bigr) \Biggl[ \sum_{j<i} \Bigl(\varepsilon_{\mathrm{score},t,j}^{(n+1)}(x_{\tau_{t,j}}^{(n+1)})\Bigr)^2 + \sum_{j\ge i} \Bigl(\varepsilon_{\mathrm{score},t,j}^{(n)}(x_{\tau_{t,j}}^{(n)})\Bigr)^2\Biggr] \\
&\lesssim & 
\log K\cdot \frac{(\tau_{t,0}-\tau_{t,i})^2}{(1-\tau_{t,0})^2}
\Biggl[\sum_{j=0}^{K-1}
\Bigl(\varepsilon_{\mathrm{score},t,j}^{(n)}(x_{\tau_{t,j}}^{(n)})\Bigr)^2+\sum_{j=0}^{K-1}
\Bigl(\varepsilon_{\mathrm{score},t,j}^{(n+1)}(x_{\tau_{t,j}}^{(n+1)})\Bigr)^2\Biggr] \\
&\lesssim &
\log K\,\tau_{t,i}^2\frac{\log^2T}{T^2}
\Biggl[\sum_{j=0}^{K-1}
\Bigl(\varepsilon_{\mathrm{score},t,j}^{(n)}(x_{\tau_{t,j}}^{(n)})\Bigr)^2+\sum_{j=0}^{K-1}
\Bigl(\varepsilon_{\mathrm{score},t,j}^{(n+1)}(x_{\tau_{t,j}}^{(n+1)})\Bigr)^2\Biggr],
\label{eq:gs_stage1_score}
\end{eqnarray*}
where the second inequality follows from \eqref{eq:23k} and the fact $1-\tau_{t,i}\asymp1-\tau_{t,0}$, the last inequality comes from \eqref{eq:23f} that by setting $i_1=0$, $i_2=i$, $i_3=i$, and $i_4=0$.

For the term $G_{12}$,  on the event $E_t$, it holds
\begin{align*}
   \|G_{12}\|_2^2
\le&2 \left(1-\tau_{t,i}\right)\left(\left\|\sum_{j<i} A_{ji}^{(t)}\left(s_{\tau_{t, j}}^{\star}\left(x_{\tau_{t, j}}^{(n+1)}\right)-s_{\tau_{t, j}}^{\star}\left(x_{\tau_{t, j}}^{\star}\right)\right)\right\|_2^2+\left\|\sum_{j \geq i} A_{ji}^{(t)}\left(s_{\tau_{t, j}}^{\star}\left(x_{\tau_{t, j}}^{(n)}\right)-s_{\tau_{t, j}}^{\star}\left(x_{\tau_{t, j}}^{\star}\right)\right)\right\|_2^2\right)\nonumber\\
\lesssim& (1-\tau_{t,i})\left(\sum_{j=0}^{K-1}\left|A_{ji}^{(t)}\right|\right)^2\left(\frac{d\log T}{\tau_{t,0}}\right)^2\left(\max _{0 \leq \ell \leq i-1}\left\|x_{\tau_{t, \ell}}^{(n+1)}-x_{\tau_{t, \ell}}^{\star}\right\|_2^2+\max _{0 \leq \ell \leq K-1}\left\|x_{\tau_{t, \ell}}^{(n)}-x_{\tau_{t, \ell}}^{\star}\right\|_2^2\right)\nonumber\\
\lesssim& \log^2K\left(\frac{d\log^2T}{T}\right)^2\left(\max _{0 \leq \ell \leq i-1}\left\|x_{\tau_{t, \ell}}^{(n+1)}-x_{\tau_{t, \ell}}^{\star}\right\|_2^2+\max _{0 \leq \ell \leq K-1}\left\|x_{\tau_{t, \ell}}^{(n)}-x_{\tau_{t, \ell}}^{\star}\right\|_2^2\right),
\end{align*}
where the third line follows from the typical region that $-\log p_{\bar{X}_{\tau_{t,j}}}\left(\lambda x_{\tau_{t,j}}^{(n)}+(1-\lambda) x_{\tau_{t,j}}^{\star}\right) \leq C_4 d \log T$, for all $j, n$ together with Lemma~\ref{lemma11}, namely $\left\|s_{\tau_{t, j}}^{\star}\left(x_{\tau_{t, j}}^{(n)}\right)-s_{\tau_{t, j}}^{\star}\left(x_{\tau_{t, j}}^{\star}\right)\right\|_2 \leq C \frac{d \log T}{\tau_{t, j}}\left\|x_{\tau_{t, j}}^{(n)}-x_{\tau_{t, j}}^{\star}\right\|_2$, and the last line holds by $\sum_{j=0}^{K-1}\left|A_{j i}^{(t)}\right| \leq C_4 \frac{\left(\tau_{t, 0}-\tau_{t,i}\right) \log K}{\left(1-\tau_{t,0}\right)^{3 / 2}}$ in Lemma~\ref{lemma9} and $\tau_{t,0} \asymp \tau_{t,i}$.

\paragraph{Bound for $G_2$.}
For any \(\tau_{t,i}\), the Lagrange remainder  bound gives 
\begin{equation} \label{eq:exact_flow_integral_cheb}
\norm{G_2}^2 \le \xkh{1-\tau_{t,i}} \cdot \left(
\int_{\tau_{t,i}}^{\tau_{t,0}}
\frac{1}{K!}
\sup_{\tau_{t,i}\le \tau\le \tau_{t,0}}
\left\|
\frac{\partial^K}{\partial \tau^K}
\frac{s_\tau^\star(x_\tau^\star)}
{2(1-\tau)^{3/2}}
\right\|_2
\left|
\prod_{j=0}^{K-1}(\tau-\tau_{t,j})
\right|
\,d\tau
\right)^2.
\end{equation}
For all $\tau_{t,i}\le \tau\le \tau_{t,0}$, according to Lemma \ref{lemma13}, it holds
\[
\sup_{\tau_{t,i}\le \tau\le \tau_{t,0}}
\left\|
\frac{\partial^K}{\partial \tau^K}
\frac{s_\tau^\star(x_\tau^\star)}
{2(1-\tau)^{3/2}}
\right\|_2
\le
C_s(MK)^K K!
\sqrt{
\frac{d\theta\log T}{\tau_{t,i}(1-\tau_{t,i})^3}
}
\left(
\frac{d\theta\log T}{\tau_{t,i}(1-\tau_{t,i})}
\right)^K .
\]
Next, set 
\[
\tau
=
\frac{\tau_{t,0}+\tau_{t,K-1}}{2}
+
\frac{\tau_{t,0}-\tau_{t,K-1}}{2}z,
\qquad z\in[-1,1],
\]
then 
\[
\prod_{j=0}^{K-1}(\tau-\tau_{t,j})
=
\left(
\frac{\tau_{t,0}-\tau_{t,K-1}}{2}
\right)^K
\prod_{j=0}^{K-1}(z-z_j),
\]
where 
\[
z_j=\cos\left(\frac{j\pi}{K-1}\right),
\qquad 0\le j\le K-1.
\]
Using the fact that 
\[
\prod_{j=0}^{K-1}(z-z_j)= 2^{2-K}(z^2-1)U_{K-2}(z) \quad \mbox{and} \quad \sup_{z\in[-1,1]}
\left|
(z^2-1)U_{K-2}(z)
\right|
\le 1,
\]
where \(U_{K-2}\) is the Chebyshev polynomial of the second kind, we obtain
\begin{equation} \label{eq:tauminj}
\sup_{\tau_{t,K-1}\le \tau\le \tau_{t,0}}
\left|
\prod_{j=0}^{K-1}(\tau-\tau_{t,j})
\right|
\le
4^{1-K}
(\tau_{t,0}-\tau_{t,K-1})^K .
\end{equation}
Substituting the preceding two estimates into
\eqref{eq:exact_flow_integral_cheb}  yields
\begin{eqnarray*}
\norm{G_2}^2 &\lesssim &  (\tau_{t,0}-\tau_{t,i})^2 \frac{d\theta\log T}{\tau_{t,i}(1-\tau_{t,i})^2} \left(
\frac{MKd\theta(\tau_{t,0}-\tau_{t,K-1})\log T}{4\tau_{t,i}(1-\tau_{t,i})} \right)^{2K}\\
& \lesssim& \frac{d\theta\,\tau_{t,i}\log^3 T}{T^2}
\left(
\frac{c_1MK d\theta\log^2 T}{4T}
\right)^{2K} \\
&\lesssim & \frac{d\tau_{t,i}\log^3 T}{T^2}
\left( \frac{M_0K d\log^2 T}{T}
\right)^{2K},
\end{eqnarray*}
where the second inequality comes from the fact that 
$\tau_{t,0}-\tau_{t,i}
\le
c_1\tau_{t,i}(1-\tau_{t,i})\frac{\log T}{T} $ and 
$ \tau_{t,0}-\tau_{t,K-1}
\le c_1\tau_{t,i}(1-\tau_{t,i})\frac{\log T}{T}$, and the last inequality follows from the typical condition that  \(\theta\le C_\theta\) for an absolute constant \(C_\theta>0\) and  $M_0\ge \frac{c_1 M C_\theta}{4}$.

Putting the estimates of $G_{11}$, $G_{12}$,  and $G_{2}$ into \eqref{eq:gs_stage1_main}, with $\tau_{t,0}\ge \tau_{t,i}, 0\le i \le K-1$ we arrive at
\begin{align}
\|x_{\tau_{t,i}}^{(n+1)}-x_{\tau_{t,i}}^\star\|_2^2
&\le
\bar C\log^2 K\left(\frac{d\log^2T}{T}\right)^2
\Biggl[
\max_{0\le \ell\le i-1}\|x_{\tau_{t,\ell}}^{(n+1)}-x_{\tau_{t,\ell}}^\star\|_2^2
+
\max_{0\le \ell\le K-1}\|x_{\tau_{t,\ell}}^{(n)}-x_{\tau_{t,\ell}}^\star\|_2^2
\Biggr]
\nonumber\\
&\quad+
\bar C\log K\,\tau_{t,0}^2\frac{\log^2T}{T^2}
\left[\sum_{j=0}^{K-1}\left(\varepsilon_{\text {score }, t, j}^{(n)}\left(x_{\tau_{t, j}}^{(n)}\right)\right)^2+\sum_{j=0}^{K-1}\left(\varepsilon_{\text {score }, t, j}^{(n+1)}\left(x_{\tau_{t, j}}^{(n+1)}\right)\right)^2\right]
\nonumber\\
&\quad+
\bar C\frac{d\tau_{t,0} \log^3 T}{T^2}
\left(\frac{M_0Kd\log^2 T}{T}\right)^{2K}
\label{eq:gs_stage1_recursion}
\end{align}
for any fixed \(0\le i\le K-1\).  Here, $\bar C$ is an universal constant. Note that \(K\le c\log T\) and  \(T\ge C_2 d\log^4T\) for a
sufficiently large constant  \(C_2>0\). Therefore,  it holds
\[
\bar C\log ^2 K\left(\frac{d \log ^2 T}{T}\right)^2 \le \frac{1}{4}.
\]
By induction, one can easily check that for every \(0\le i\le K-1\),
\begin{align}
\|x_{\tau_{t,i}}^{(n+1)}-x_{\tau_{t,i}}^\star\|_2^2
&\le
\frac{1}{3}
\max_{0\le \ell\le K-1}
\|x_{\tau_{t,\ell}}^{(n)}-x_{\tau_{t,\ell}}^\star\|_2^2
+
\frac{4\bar C}{3}
\log K\,\tau_{t,0}^2\frac{\log^2T}{T^2}
\sum_{m=n}^{n+1}\sum_{j=0}^{K-1}
\Bigl(\varepsilon_{\mathrm{score},t,j}^{(m)}(x_{\tau_{t,j}}^{(m)})\Bigr)^2
\nonumber\\
&\quad+
 \frac{4\bar C}{3}\frac{d\tau_{t,0} \log^3 T}{T^2}
\left(\frac{M_0Kd\log^2 T}{T}\right)^{2K}.
\label{eq:gs_stage1_induction}
\end{align}
Indeed, for  \(i=0\), \eqref{eq:gs_stage1_induction} holds immediate since
\(x_{\tau_{t,0}}^{(n+1)}=x_{\tau_{t,0}}^\star=x_{\tau_{t,0}}\).  Next,  assume that \eqref{eq:gs_stage1_induction} holds for all
\(0\le i\le k_0\), where \(0\le k_0\le K-2\).  Using the hypothesis conditions, one can see that \eqref{eq:gs_stage1_induction} holds for \(i=k_0+1\). With \eqref{eq:gs_stage1_induction}  in place, we obtain 
\begin{align}
\max_{0\le r\le K-1}\|x_{\tau_{t,r}}^{(N)}-x_{\tau_{t,r}}^\star\|_2^2
&\le
3^{-N}
\max_{0\le r\le K-1}\|x_{\tau_{t,r}}^{(0)}-x_{\tau_{t,r}}^\star\|_2^2
+
C\log K\,\tau_{t,0}^2\frac{\log^2T}{T^2}
\sum_{m=0}^{N}\sum_{j=0}^{K-1}
\Bigl(\varepsilon_{\mathrm{score},t,j}^{(m)}(x_{\tau_{t,j}}^{(m)})\Bigr)^2
\nonumber\\
&\quad+
C \frac{d\tau_{t,0} \log^3 T}{T^2}
\left(\frac{M_0Kd\log^2 T}{T}\right)^{2K}.
\label{eq:gs_stage1_after_iter}
\end{align}
Here, $C=3\bar{C}$. Finally, as shown in  \cite{li2025faster}, one has
\[
\|x_{\tau_{t,i}}^\star-x_{\tau_{t,i}}^{(0)}\|_2^2
\le \left(\bar{\alpha}_{t-1} \sqrt{1-\alpha_t} \sqrt{\frac{1-\alpha_t}{1-\bar{\alpha}_{t-1}}} \sqrt{d \log T}\right)^2
\le
\frac{4c_1^2d\log^3T}{T^2},
\qquad 0\le i\le K-1.
\]
Since $T\ge C_2d\log^4T$, $K\le c\log T$, and $N\ge C_3K\log T$ for a sufficiently large constant $C_3>0$ , it holds
\[
3^{-N}
\max_{0\le r\le K-1}\|x_{\tau_{t,r}}^{(0)}-x_{\tau_{t,r}}^\star\|_2^2
\le
C \frac{d\tau_{t,0} \log^3T}{T^2}
\left(\frac{M_0Kd\log^2 T}{T}\right)^{2K}.
\]
Substituting this into \eqref{eq:gs_stage1_after_iter} yields
\[
\max_{0\le r\le K-1}\|x_{\tau_{t,r}}^{(N)}-x_{\tau_{t,r}}^\star\|_2^2
\le
C \log K\,\tau_{t,0}^2 \frac{\log ^2 T}{T^2}
\sum_{m=0}^{N}\sum_{j=0}^{K-1}
\Bigl(\varepsilon_{\mathrm{score}, t, j}^{(m)}(x_{\tau_{t,j}}^{(m)})\Bigr)^2
+
C \frac{d \tau_{t,0} \log ^3 T}{T^2}
\left(\frac{M_0Kd \log ^2 T}{T}\right)^{2K}.
\]
This completes the proof.
\end{proof}

\subsection{Proof of Lemma~\ref{lemma4}}
\begin{proof}
Recall that 
\[ J_{\tau_{t,i}}^{(n)} :=
\frac{\partial (x_{\tau_{t,i}}^{(n)}/\sqrt{1-\tau_{t,i}})}
{\partial (x_{\tau_{t,0}}/\sqrt{1-\tau_{t,0}})}, \qquad 
J_{\tau_{t,i}}^\star
:=
\frac{\partial (x_{\tau_{t,i}}^\star/\sqrt{1-\tau_{t,i}})}
{\partial (x_{\tau_{t,0}}^\star/\sqrt{1-\tau_{t,0}})}.
\]
It then follows from \eqref{eq:gs_update} that 
\begin{align}
J_{\tau_{t,i}}^{(n+1)}
=
I
&+
\sum_{j<i}
A_{ji}^{(t)}
\frac{\partial s_{\tau_{t,j}}\!\left(x_{\tau_{t,j}}^{(n+1)}\right)}
{\partial \left(x_{\tau_{t,j}}^{(n+1)}/\sqrt{1-\tau_{t,j}}\right)}
J_{\tau_{t,j}}^{(n+1)}
+
\sum_{j\ge i}
A_{ji}^{(t)}
\frac{\partial s_{\tau_{t,j}}\!\left(x_{\tau_{t,j}}^{(n)}\right)}
{\partial \left(x_{\tau_{t,j}}^{(n)}/\sqrt{1-\tau_{t,j}}\right)}
J_{\tau_{t,j}}^{(n)} .
\label{eq:gs_jacobian_update}
\end{align}
Similarly, note that the exact rescaled probability flow ODE is
\[
\frac{x_{\tau_{t,i}}^\star}{\sqrt{1-\tau_{t,i}}}
=
\frac{x_{\tau_{t,0}}^\star}{\sqrt{1-\tau_{t,0}}}
-
\int_{\tau_{t,0}}^{\tau_{t,i}}
\frac{s_\tau^\star(x_\tau^\star)}
{2(1-\tau)^{3/2}}
\,d\tau .
\]
Therefore,
\begin{equation}
J_{\tau_{t,i}}^\star
=
I
-
\int_{\tau_{t,0}}^{\tau_{t,i}}
\frac{1}{2(1-\tau)^{3/2}}
\frac{\partial s_\tau^\star(x_\tau^\star)}
{\partial \left(x_\tau^\star/\sqrt{1-\tau}\right)}
J_\tau^\star
\,d\tau .
\label{eq:exact_jacobian_flow}
\end{equation}
Combining the two previous equations, we can obtain
\begin{eqnarray*}
&& J_{\tau_{t,i}}^{(n+1)}-J_{\tau_{t,i}}^\star \\
&= & 
\sum_{j<i} A_{ji}^{(t)}
\Biggl[
\frac{\partial s_{\tau_{t,j}}(x_{\tau_{t,j}}^{(n+1)})}
{\partial (x_{\tau_{t,j}}^{(n+1)}/\sqrt{1-\tau_{t,j}})}
J_{\tau_{t,j}}^{(n+1)}
-
\frac{\partial s_{\tau_{t,j}}^\star(x_{\tau_{t,j}}^\star)}
{\partial (x_{\tau_{t,j}}^\star/\sqrt{1-\tau_{t,j}})}
J_{\tau_{t,j}}^\star
\Biggr]
\nonumber\\
&& +
\sum_{j\ge i} A_{ji}^{(t)}
\Biggl[
\frac{\partial s_{\tau_{t,j}}(x_{\tau_{t,j}}^{(n)})}
{\partial (x_{\tau_{t,j}}^{(n)}/\sqrt{1-\tau_{t,j}})}
J_{\tau_{t,j}}^{(n)}
-
\frac{\partial s_{\tau_{t,j}}^\star(x_{\tau_{t,j}}^\star)}
{\partial (x_{\tau_{t,j}}^\star/\sqrt{1-\tau_{t,j}})}
J_{\tau_{t,j}}^\star
\Biggr]
\nonumber\\
&& +
\int_{\tau_{t,0}}^{\tau_{t,i}}
\Biggl[
\frac{1}{2(1-\tau)^{3/2}}
\frac{\partial s_\tau^\star(x_\tau^\star)}
{\partial (x_\tau^\star/\sqrt{1-\tau})}
J_\tau^\star
-
\sum_{j=0}^{K-1}\psi_{t,j}(\tau)\,
\frac{1}{2(1-\tau_{t,j})^{3/2}}
\frac{\partial s_{\tau_{t,j}}^\star(x_{\tau_{t,j}}^\star)}
{\partial (x_{\tau_{t,j}}^\star/\sqrt{1-\tau_{t,j}})}
J_{\tau_{t,j}}^\star
\Biggr]d\tau \nonumber \\
&=& \underbrace{\sum_{j<i} A_{ji}^{(t)}
\Biggl(
\frac{\partial s_{\tau_{t,j}}(x_{\tau_{t,j}}^{(n+1)})}
{\partial (x_{\tau_{t,j}}^{(n+1)}/\sqrt{1-\tau_{t,j}})}
-
\frac{\partial s_{\tau_{t,j}}^\star(x_{\tau_{t,j}}^\star)}
{\partial (x_{\tau_{t,j}}^\star/\sqrt{1-\tau_{t,j}})}
\Biggr)
J_{\tau_{t,j}}^{(n+1)}}_{H_1}\\
&& +\underbrace{\sum_{j<i} A_{ji}^{(t)} \frac{\partial s_{\tau_{t, j}}^{\star}\left(x_{\tau_{t, j}}^{\star}\right)}{\partial\left(x_{\tau_{t, j}}^{\star} / \sqrt{1-\tau_{t, j}}\right)}\left(J_{\tau_{t, j}}^{(n+1)}-J_{\tau_{t, j}}^{\star}\right)}_{H_2} \\
&& +\underbrace{\sum_{j \geq i} A_{ji}^{(t)}\left(\frac{\partial s_{\tau_{t, j}}\left(x_{\tau_{t, j}}^{(n)}\right)}{\partial\left(x_{\tau_{t, j}}^{(n)} / \sqrt{1-\tau_{t, j}}\right)}-\frac{\partial s_{\tau_{t, j}}^{\star}\left(x_{\tau_{t, j}}^{\star}\right)}{\partial\left(x_{\tau_{t, j}}^{\star} / \sqrt{1-\tau_{t, j}}\right)}\right) J_{\tau_{t, j}}^{(n)}}_{H_3}
\end{eqnarray*}
\begin{eqnarray*}
&&+\underbrace{\sum_{j \geq i} A_{ji}^{(t)} \frac{\partial s_{\tau_{t, j}}^{\star}\left(x_{\tau_{t, j}}^{\star}\right)}{\partial\left(x_{\tau_{t, j}}^{\star} / \sqrt{1-\tau_{t, j}}\right)}\left(J_{\tau_{t, j}}^{(n)}-J_{\tau_{t, j}}^{\star}\right)}_{H_4}\nonumber \\
&& + \underbrace{\int_{\tau_{t,0}}^{\tau_{t,i}}
\Biggl[
\frac{1}{2(1-\tau)^{3/2}}
\frac{\partial s_\tau^\star(x_\tau^\star)}
{\partial (x_\tau^\star/\sqrt{1-\tau})}
J_\tau^\star
- \sum_{j=0}^{K-1}\psi_{t,j}(\tau)\,
\frac{1}{2(1-\tau_{t,j})^{3/2}}
\frac{\partial s_{\tau_{t,j}}^\star(x_{\tau_{t,j}}^\star)}
{\partial (x_{\tau_{t,j}}^\star/\sqrt{1-\tau_{t,j}})}
J_{\tau_{t,j}}^\star
\Biggr]d\tau}_{H_5}.
\end{eqnarray*}
To begin with, we claim that on the event $E_t$, for all $0 \le n \le N$ and $0 \le j \le K-1$, it holds
\begin{equation}
    \|J_{\tau_{t,j}}^{(n)}\|\le \widetilde{C}
    \label{claim1}
\end{equation}
for some universal constant $\widetilde{C}>0$.

\paragraph{Bound for $H_1$ and $H_3$.}
Observe that 
\begin{align*}
H_1&=\underbrace{\sum_{j<i} A_{j i}^{(t)}\left(\frac{\partial s_{\tau_{t, j}}\left(x_{\tau_{t, j}}^{(n+1)}\right)}{\partial\left(x_{\tau_{t, j}}^{(n+1)} / \sqrt{1-\tau_{t, j}}\right)}-\frac{\partial s_{\tau_{t, j}}^{\star}\left(x_{\tau_{t, j}}^{(n+1)}\right)}{\partial\left(x_{\tau_{t, j}}^{(n+1)} / \sqrt{1-\tau_{t, j}}\right)}\right) J_{\tau_{t, j}}^{(n+1)}}_{H_{11}}\\
&+ \underbrace{\sum_{j<i} A_{j i}^{(t)}\Biggl(
\frac{\partial s_{\tau_{t,j}}^\star(x_{\tau_{t,j}}^{(n+1)})}
{\partial (x_{\tau_{t,j}}^{(n+1)}/\sqrt{1-\tau_{t,j}})}
- \frac{\partial s_{\tau_{t,j}}^\star(x_{\tau_{t,j}}^\star)}
{\partial (x_{\tau_{t,j}}^\star/\sqrt{1-\tau_{t,j}})}
\Biggr) J_{\tau_{t,j}}^{(n+1)}}_{H_{12}},
\end{align*}
Note that  from \eqref{claim1} that $\|J_{\tau_{t,j}}^{(n)}\|\le \widetilde{C}$. This gives
\begin{align*}
\|H_{11}\|^2&\lesssim \left(\sum_{j=0}^{K-1}\left|A_{j i}^{(t)}\right|^2\right)(1-\tau_{t,0})\left[\sum_{j<i}\left(\varepsilon_{\text {Jacobi }, t, j}^{(n+1)}\left(x_{\tau_{t, j}}^{(n+1)}\right)\right)^2+\sum_{j \geq i}\left(\varepsilon_{\text {Jacobi }, t, j}^{(n)}\left(x_{\tau_{t, j}}^{(n)}\right)\right)^2\right]\nonumber\\
&\le
C \log K\tau_{t,i}^2\frac{\log^2T}{T^2}
\sum_{m=n}^{n+1}\sum_{\ell=0}^{K-1}
\left(\varepsilon_{\mathrm{Jacobi},t,\ell}^{(m)}(x_{\tau_{t,\ell}}^{(m)})\right)^2,
\end{align*}
where the first inequality follows from the Cauchy-Schwarz inequality and the second inequality follows by \eqref{eq:23k} in Lemma~\ref{lemma9}.  For the term $H_{12}$,  since $E_t$ holds,  Lemma~\ref{lemma11} yields
\[
\left\| \frac{\partial s_{\tau_{t,j}}^\star(x)}
{\partial (x/\sqrt{1-\tau_{t,j}})}
-\frac{\partial s_{\tau_{t,j}}^\star(y)}
{\partial (y/\sqrt{1-\tau_{t,j}})}
\right\| \le
C\sqrt{1-\tau_{t,j}}\sqrt{\frac{d^3\log^3T}{\tau_{t,j}^3}}\,
\|x-y\|.
\]
Therefore,  we have
\begin{align*}
\|H_{12}\|^2 &\lesssim
\left(\sum_{j=0}^{K-1}|A_{ji}^{(t)}|^2\right) 
\sum_{j<i}
(1-\tau_{t,j})\frac{d^3\log^3T}{\tau_{t,j}^3}
\|x_{\tau_{t,j}}^{(n+1)}-x_{\tau_{t,j}}^\star\|_2^2
\nonumber\\
&\lesssim
\log K\,
\frac{\log^2T}{T^2}
\frac{d^3\log^3T}{\tau_{t,K-1}}
\sum_{j<i}\|x_{\tau_{t,j}}^{(n+1)}-x_{\tau_{t,j}}^\star\|_2^2,
\end{align*}
Applying Lemma \ref{lemma3}, one has
\begin{align*}
\|H_{12}\|^2  &\le
C'\frac{d^3 \log^5 T K\log K}{T^2\tau_{t,K-1}}
\left(
\frac{1}{3^{n}} 4 c_1^2 \frac{d \log ^3 T}{T^2}
+
C\log K \tau_{t, 0}^2 \frac{\log ^2 T}{T^2}
\sum_{m=0}^{n+1} \sum_{j=0}^{K-1}
\left(\varepsilon_{\mathrm{score}, t, j}^{(m)}
\left(x_{\tau_{t, j}}^{(m)}\right)\right)^2
\right.
\nonumber\\
&\qquad\left.
+
C \frac{d \tau_{t, 0} \log ^3 T}{T^2}
\left(\frac{M_0 K d \log ^2 T}{T}\right)^{2 K}
\right).
\end{align*}
Combining the bounds for $H_{11}$ and $H_{12}$, we get an upper bound for $H_1$.  A similar bound for $H_3$ can be given, and we omit it.

\paragraph{Bound for $H_2$ and $H_4$.}
Again by Lemma~\ref{lemma11} and the fact that  $\tau_{t,j}\asymp\tau_{t,i}$, one has
\[
\left\|
\frac{\partial s_{\tau_{t,j}}^\star(x_{\tau_{t,j}}^\star)}
{\partial (x_{\tau_{t,j}}^\star/\sqrt{1-\tau_{t,j}})}
\right\| 
\lesssim
\sqrt{1-\tau_{t,i}}\frac{d\log T}{\tau_{t,i}}.
\]
Therefore, 
\begin{align*}
\|H_2\|^2
&\le C'
\left(\sum_{j<i}|A_{ji}^{(t)}|^2\right)
\sum_{j<i}
(1-\tau_{t,j})\frac{d^2\log^2T}{\tau_{t,j}^2}
\|J_{\tau_{t,j}}^{(n+1)}-J_{\tau_{t,j}}^\star\|^2
\nonumber\\
&\le
C' \log K\left(\frac{d\log^2T}{T}\right)^2
K\max _{0\le j\le i-1}\|J_{\tau_{t,j}}^{(n+1)}-J_{\tau_{t,j}}^\star\|^2,
\end{align*}
Similarly, we have
\begin{align*}
\|H_4\|^2
&\le
C' \log K\left(\frac{d\log^2T}{T}\right)^2
K \max _{0\le j\le K-1}\|J_{\tau_{t,j}}^{(n)}-J_{\tau_{t,j}}^\star\|^2.
\end{align*}

\paragraph{Bound for $H_5$.}
According to the Lagrange interpolation, one has
\begin{eqnarray*}
\|H_5\|^2 &\le & \left(\int_{\tau_{t, i}}^{\tau_{t, 0}} \frac{1}{K!} \sup _{\tau_{t, i} \leq \tau \leq \tau_{t, 0}}\left\|\frac{\partial^K}{\partial \tau^K} \frac{\partial s_\tau^\star(x_\tau^\star)}
{\partial (x_\tau^\star/\sqrt{1-\tau})}J_\tau^\star\right\|\left|\prod_{j=0}^{K-1}\left(\tau-\tau_{t, j}\right)\right| d \tau\right)^2 \\
&\lesssim & \left(\frac{M_0K d\log^2T}{T}\right)^{2K+2},
\end{eqnarray*}
where the second inequality comes from \eqref{eq:tauminj}, the fact that  $\tau\asymp\tau_{t,i}$ and $1-\tau\asymp 1-\tau_{t,i}$, and Lemma \ref{lemma13} that 
\[
\sup_{\tau_{t,K-1}\le \tau\le \tau_{t,0}}
\left\|
\frac{\partial^K}{\partial \tau^K}
\left[
\frac{1}{2(1-\tau)^{3/2}}
\frac{\partial s_\tau^\star(x_\tau^\star)}
{\partial (x_\tau^\star/\sqrt{1-\tau})}
J_\tau^\star
\right]
\right\|
\le
C_{\mathrm J}(MK)^{K+1} K!
\left(\frac{d\theta\log T}{\tau_{t,i}(1-\tau_{t,i})}\right)^{K+1}.
\]

\paragraph{Collect all the bounds.}
Combining the estimates for $H_1, H_2, H_3, H_4$ and $H_5$ with $\tau_{t,0}\ge \tau_{t,i}, 0\le i \le K-1$, we obtain that for any fixed $0\le i\le K-1$, it holds
\begin{align*}
\|J_{\tau_{t,i}}^{(n+1)}-J_{\tau_{t,i}}^\star\|^2
&\le
C'K\log K\left(\frac{d\log^2T}{T}\right)^2
\Biggl[
\max _{0\le j \le i-1}\|J_{\tau_{t,j}}^{(n+1)}-J_{\tau_{t,j}}^\star\|^2
+
\max _{0\le j \le K-1}\|J_{\tau_{t,j}}^{(n)}-J_{\tau_{t,j}}^\star\|^2
\Biggr]
\nonumber\\
&\quad+
C'\frac{d^3 \log^5 T K\log K}{T^2\tau_{t,K-1}}
\left(
\frac{1}{3^{n}} 4 c_1^2 \frac{d \log ^3 T}{T^2}
+
C\log K \tau_{t, 0}^2 \frac{\log ^2 T}{T^2}
\sum_{m=0}^{n+1} \sum_{j=0}^{K-1}
\left(\varepsilon_{\mathrm{score}, t, j}^{(m)}
\left(x_{\tau_{t, j}}^{(m)}\right)\right)^2
\right.
\nonumber\\
&\qquad\left.
+
C \frac{d \tau_{t, 0} \log ^3 T}{T^2}
\left(\frac{M_0 K d \log ^2 T}{T}\right)^{2 K}
\right)
+C' \left(\frac{M_0K d\log^2T}{T}\right)^{2K+2}
\nonumber\\
&\quad+
C' \log K\tau_{t,0}^2\frac{\log^2T}{T^2}
\Biggl[\sum_{\ell=0}^{K-1}
\left(\varepsilon_{\mathrm{Jacobi},t,\ell}^{(n)}(x_{\tau_{t,\ell}}^{(n)})\right)^2+\sum_{\ell=0}^{K-1}
\left(\varepsilon_{\mathrm{Jacobi},t,\ell}^{(n+1)}(x_{\tau_{t,\ell}}^{(n+1)})\right)^2\Biggr].
\end{align*}
By the same arguments to \eqref{eq:gs_stage1_induction} and \eqref{eq:gs_stage1_after_iter}, we obtain
\begin{align}
\max_{0\le i\le K-1}
\|J_{\tau_{t,i}}^{(N)}-J_{\tau_{t,i}}^\star\|^2 &\le
7^{-N}
\max_{0\le i\le K-1}
\|J_{\tau_{t,i}}^{(0)}-J_{\tau_{t,i}}^\star\|^2
+
C
\frac{d^3\tau_{t,0}\log^7T\log^2K}{T^4}
\sum_{m=0}^{N}\sum_{j=0}^{K-1}
\left(\varepsilon_{\mathrm{score},t,j}^{(m)}
\left(x_{\tau_{t,j}}^{(m)}\right)\right)^2
\nonumber\\
+
2C\log K&\,\tau_{t,0}^2\frac{\log^2T}{T^2}
\sum_{m=0}^{N}\sum_{j=0}^{K-1}
\left(\varepsilon_{\mathrm{Jacobi},t,j}^{(m)}
\left(x_{\tau_{t,j}}^{(m)}\right)\right)^2
+
C\left(\frac{M_0K d\log^2T}{T}\right)^{2K+2}.
\label{eq:gs_stage2_iter_expanded}
\end{align}
We claim that 
\begin{equation}
    \max _{0 \leq j \leq K-1}\left\|J_{\tau_{t, j}}^{(0)}-J_{\tau_{t, j}}^{\star}\right\|^2 \leq C\left(\frac{dK \log ^2 T}{T}\right)^2.
    \label{initialization}
\end{equation}
Note that \(T\ge C_2d\log^4T\), \(K\le c\log T\) and \(N\ge C_3K\log T\) for some sufficiently large constant \(C_3>0\). 
Combining \eqref{eq:gs_stage2_iter_expanded} and   \eqref{initialization} gives
\begin{align}
    \|J_{\tau_{t,i}}^{(N)}-J_{\tau_{t,i}}^\star\|^2
&\le
C K\log^2 K\,\frac{d^3\log^7 T}{T^4}
\sum_{n=0}^{N}\sum_{j=0}^{K-1}
\Bigl(\varepsilon_{\mathrm{score},t,j}^{(n)}(x_{\tau_{t,j}}^{(n)})\Bigr)^2
\nonumber\\
&\quad+
C \log K \tau_{t,0}^2 \frac{\log^2 T}{T^2}
\sum_{n=0}^{N}\sum_{j=0}^{K-1}
\Bigl(\varepsilon_{\mathrm{Jacobi},t,j}^{(n)}(x_{\tau_{t,j}}^{(n)})\Bigr)^2
+
C\,
\left(\frac{M_0Kd\log^2 T}{T}\right)^{2K+2}.
\label{eq:gs_stage2_final_closed}
\end{align}
This proves \eqref{eq:lemma4_gs_bound}. It remains to prove the claims \eqref{claim1} and  \eqref{initialization}.

\paragraph{Proof of claim \eqref{initialization}.}
Note that  $ x_{\tau_{t,j}}^{(0)}=x_{\tau_{t,0}}, \quad 0\le j\le K-1$. It immediately gives 
\[
J_{\tau_{t,j}}^{(0)}
=
\frac{\partial (x_{\tau_{t,j}}^{(0)}/\sqrt{1-\tau_{t,j}})}
{\partial (x_{\tau_{t,0}}/\sqrt{1-\tau_{t,0}})} =
\sqrt{\frac{1-\tau_{t,0}}{1-\tau_{t,j}}}\,I.
\]
Using $1-\tau_{t,j}\asymp 1-\tau_{t,i}$, one has
\begin{equation}
\|J_{\tau_{t,j}}^{(0)}\|\le c_7,
\qquad
\|J_{\tau_{t,j}}^{(0)}-I\|=
\Abs{\frac{1-\tau_{t,0}}{1-\tau_{t,j}}-1}
=
\frac{\tau_{t,0}-\tau_{t,j}}{1-\tau_{t,j}}
\le
c_8\frac{\log T}{T},
\label{eq:gs_stage2_J0_bounds}
\end{equation}
where $c_7,c_8>0$ are universal constants. For convenience,  define
\[
A(\tau)
:=
\frac{1}{2(1-\tau)^{3/2}}
\frac{\partial s_\tau^\star(x_\tau^\star)}
{\partial (x_\tau^\star/\sqrt{1-\tau})}.
\]
Similar to \eqref{eq:exact_jacobian_flow}, we have
\begin{equation} \label{eq:exaode1}
J_\tau^\star
=
I+\int_\tau^{\tau_{t,0}}A(u)J_u^\star\,du,
\qquad \tau\in[\tau_{t,K-1},\tau_{t,0}].
\end{equation}
According to Lemma \ref{lemma13}  with $k=0$, one has
\[
\|A(u)\|
\le
C_JM\frac{dK\log T}{u(1-u)},
\qquad u\in[\tau_{t,K-1},\tau_{t,0}].
\]
Therefore,  we obtain, for every $\tau\in[\tau_{t,K-1},\tau_{t,0}]$,
\[
\|J_\tau^\star\| \le
1+\int_\tau^{\tau_{t,0}} C_JM\frac{dK\log T}{u(1-u)}\|J_u^\star\|\,du \le \exp\!\left(
C_JM\int_\tau^{\tau_{t,0}}
\frac{dK\log T}{u(1-u)}\,du
\right),
\]
where the last inequality comes from  Gronwall's inequality. Using the fact that  $u\asymp \tau_{t,0}$ and
$1-u\asymp 1-\tau_{t,0}$ on $[\tau_{t,K-1},\tau_{t,0}]$, together with \eqref{eq:23f} in Lemma~\ref{lemma9}, we have
\[
\int_\tau^{\tau_{t,0}}\frac{dK\log T}{u(1-u)}\,du
\le
c_1\frac{dK\log^2T}{T}.
\]
Since $T\ge C_2d\log^4T$ and $K\le c\log T$, it implies
\begin{equation}
\sup_{\tau\in[\tau_{t,K-1},\tau_{t,0}]}
\|J_\tau^\star\| \le \exp\left(C_JMc_1\frac{dK\log^2T}{T}\right) \le c_8.
\label{eq:gs_stage2_Jstar_uniform}
\end{equation}
Substituting \eqref{eq:gs_stage2_Jstar_uniform} into \eqref{eq:exaode1} gives 
\begin{equation} \label{eq:Jstar}
\|J_{\tau_{t,j}}^\star-I\|
\le
\int_{\tau_{t,j}}^{\tau_{t,0}}
\|A(u)\|\,\|J_u^\star\|\,du
\le
c_8\int_{\tau_{t,j}}^{\tau_{t,0}}
C_JM\frac{dK\log T}{u(1-u)}\,du
\le
c_9\frac{dK\log^2T}{T} \le \frac12
\end{equation}
for all $0\le j\le K-1$, where $c_9$ is an universal constant. Finally, by the triangle inequality and using \eqref{eq:gs_stage2_J0_bounds},
\[
\|J_{\tau_{t,j}}^{(0)}-J_{\tau_{t,j}}^\star\|
\le
\|J_{\tau_{t,j}}^{(0)}-I\|
+
\|J_{\tau_{t,j}}^\star-I\|
\le
c_8\frac{\log T}{T}
+
c_9\frac{dK\log^2T}{T}
\le
c_{10}\frac{dK\log^2T}{T}
\]
holds for all $0\le j\le K-1$, where $c_{10}>0$ is an universal  constant.  This proves the  claim~\eqref{initialization}.

\paragraph{Proof of claim \eqref{claim1}.}
The proof is inductive in nature.  When $n=0$, through \eqref{eq:gs_stage2_J0_bounds},  it holds $\|J^{(0)}_{\tau_{t,i}}\|\le \widetilde{C}$ for any $0\le i \le K-1$. Next, assume that $\|J^{(n_0)}_{\tau_{t,i}}\|\le \widetilde{C}$ holds for all $0 \le i \le K-1$. We then prove $\|J^{(n_0+1)}_{\tau_{t,i}}\|\le \widetilde{C}$ holds for all $0 \le i \le K-1$ by induction again.  For the case $i=0$, we know
\(
J^{(n_0+1)}_{\tau_{t,0}}=I.
\)
Now  assume that \(
\|J_{\tau_{t,i}}^{(n_0+1)}\|\le \widetilde{C} \) holds for all $1 \le i \le k_0$.  According to \eqref{eq:gs_jacobian_update} and using the same approaches as before, we have 
\begin{align*}
\|J_{\tau_{t,k_0+1}}^{(n_0+1)}-I\|
&\le
\Biggl\|
\sum_{j<k_0+1} A_{ji}^{(t)}
\frac{\partial s_{\tau_{t,j}}^\star(x_{\tau_{t,j}}^{(n_0+1)})}
{\partial (x_{\tau_{t,j}}^{(n_0+1)}/\sqrt{1-\tau_{t,j}})}
J_{\tau_{t,j}}^{(n_0+1)}
+
\sum_{j\ge k_0+1} A_{ji}^{(t)}
\frac{\partial s_{\tau_{t,j}}^\star(x_{\tau_{t,j}}^{(n_0)})}
{\partial (x_{\tau_{t,j}}^{(n_0)}/\sqrt{1-\tau_{t,j}})}
J_{\tau_{t,j}}^{(n_0)}
\Biggr\|
\nonumber\\
&\quad+
\Biggl\|
\sum_{j<k_0+1} A_{ji}^{(t)}
\Biggl(
\frac{\partial s_{\tau_{t,j}}(x_{\tau_{t,j}}^{(n_0+1)})}
{\partial (x_{\tau_{t,j}}^{(n_0+1)}/\sqrt{1-\tau_{t,j}})}
-
\frac{\partial s_{\tau_{t,j}}^\star(x_{\tau_{t,j}}^{(n_0+1)})}
{\partial (x_{\tau_{t,j}}^{(n_0+1)}/\sqrt{1-\tau_{t,j}})}
\Biggr)
J_{\tau_{t,j}}^{(n_0+1)}
\Biggr\|
\nonumber\\
&\quad+
\Biggl\|
\sum_{j\ge k_0+1} A_{ji}^{(t)}
\Biggl(
\frac{\partial s_{\tau_{t,j}}(x_{\tau_{t,j}}^{(n_0)})}
{\partial (x_{\tau_{t,j}}^{(n_0)}/\sqrt{1-\tau_{t,j}})}
-
\frac{\partial s_{\tau_{t,j}}^\star(x_{\tau_{t,j}}^{(n_0)})}
{\partial (x_{\tau_{t,j}}^{(n_0)}/\sqrt{1-\tau_{t,j}})}
\Biggr)
J_{\tau_{t,j}}^{(n_0)}
\Biggr\| \\
&\lesssim \frac{d\log^2T\log K}{T}
+
\sqrt{\log K} \,\tau_{t,i}\frac{\log T}{T}
\Biggl(
\sum_{m=n_0}^{n_0+1}\sum_{j=0}^{K-1}
\left(\varepsilon_{\mathrm{Jacobi},t,j}^{(m)}(x_{\tau_{t,j}}^{(m)})\right)^2
\Biggr)^{1/2}.
\end{align*}
On the event $E_t$ and note that $T\ge C_2d \log^4 T$, one has $\|J_{\tau_{t,k_0+1}}^{(n_0+1)}-I\| \le c'$ for  a sufficiently
small absolute constant $c'>0$, which gives \( \|J_{\tau_{t,k_0+1}}^{(n_0+1)}\|\le \widetilde{C}\). By induction,  we have
\begin{equation*}
\|J_{\tau_{t,i}}^{(n)}\|\le \widetilde{C},
\qquad
0\le i\le K-1,\ \ 0\le n\le N,
\end{equation*}
which gives claim~\eqref{claim1}.

Next, we turn to prove the second part of Lemma \ref{lemma4}, namely,  to prove \eqref{lemma4_lemma2_condition}.  From  \eqref{eq:gs_stage2_final_closed}, and on the event $E_t$, we have
\[
\|J_{\tau_{t,K-1}}^{(N)}-J_{\tau_{t,K-1}}^\star\| \le \frac14
\]
Combining this with \eqref{eq:Jstar}, we obtain
\[
\|J_{\tau_{t,K-1}}^{(N)}-I\|
\le
\|J_{\tau_{t,K-1}}^{(N)}-J_{\tau_{t,K-1}}^\star\|
+
\|J_{\tau_{t,K-1}}^\star-I\|
\le
\frac34.
\]
Hence $J_{\tau_{t,K-1}}^{(N)}$ is invertible, and by the Neumann-series bound,
\[
\left\|\left(J_{\tau_{t,K-1}}^{(N)}\right)^{-1}\right\|
\le
\frac{1}{1-\|J_{\tau_{t,K-1}}^{(N)}-I\|}
\le 4.
\]
This proves \eqref{lemma4_lemma2_condition} and completes the proof.

\end{proof}

\subsection{Proof of Lemma~\ref{lemma5}} \label{lemma5_proof}
\begin{proof}
From the definition of the total variation distance and set $\mathcal E$ in \eqref{eq:def_E_step2_lemma},  we have
\begin{eqnarray*}
\operatorname{TV}(q_1,p_1)
&= &  \int_{\{x\in \R^d: q_1(x)>p_1(x)\}}(q_1(x)-p_1(x))\,dx \\
&=&  \int_{\mathcal E} (q_1(x)-p_1(x))\,dx + \int_{\{x\in \R^d: p_1(x)<q_1(x)\le \exp(-c_6 d\log T)\}} (q_1(x)-p_1(x))\,dx \notag\\
&\le&  \int_{\mathcal E} (q_1(x)-p_1(x))\,dx + \int_{\{x\in \R^d: q_1(x)\le \exp(-c_6 d\log T)\}} q_1(x)\,dx \notag\\
&=:&  \mathbb{E}_{Y_1\sim p_1}\!\left[ \left(\frac{q_1(Y_1)}{p_1(Y_1)}-1\right)\mathbf{1}\{Y_1\in \mathcal E\} \right] +r_T.
\end{eqnarray*}
For the term \(r_T\), let \(K_1> 0\) be any fixed  constant, one has
\begin{eqnarray}
r_T &\le & \int_{\{x\in\mathbb R^d:\ \|x\|_2\le T^{c_R+K_1}\sqrt d,\ 
q_1(x)\le \exp(-c_6 d\log T)\}} q_1(x)\,dx +
\int_{\{x\in\mathbb R^d:\ \|x\|_2> T^{c_R+K_1}\sqrt d\}}
q_1(x)\,dx  \notag\\
&\le &  \exp(-c_6 d\log T)\,
\operatorname{Vol}\!\left(
\left\{x\in\mathbb R^d:\|x\|_2\le T^{c_R+K_1}\sqrt d\right\}
\right) + \mathbb{P}_{X_1\sim q_1}\!\left(\|X_1\|_2>T^{c_R+K_1}\sqrt d\right), \label{eq:RT}
\end{eqnarray}
where $c_R$ is the constant defined in Assumption \ref{assup:1}, and  \(\operatorname{Vol}(\cdot)\) denotes the \(d\)-dimensional Lebesgue measure. Recall that 
\[
X_1=\sqrt{\bar{\alpha}_1}\,X_0+\sqrt{1-\bar{\alpha}_1}\,Z,
\qquad
Z\sim \mathcal{N}(0,I_d),
\]
and hence
\begin{align*}
\mathbb{E}\bigl[\|X_1\|_2^2\bigr]
\le
2\bar{\alpha}_1\mathbb{E}\bigl[\|X_0\|_2^2\bigr]
+
2(1-\bar{\alpha}_1)\mathbb{E}\bigl[\|Z\|_2^2\bigr] \le
2T^{2c_R}+2d,
\end{align*}
where we used Assumption  \ref{assup:1} in the last inequality. Therefore, by
Markov's inequality,
\begin{align}
\mathbb{P}_{X_1\sim q_1}\!\left(\|X_1\|_2>T^{c_R+K_1}\sqrt d\right)
\le
\frac{\mathbb{E}\bigl[\|X_1\|_2^2\bigr]}
     {d\cdot T^{2c_R+2K_1}} \le
\frac{2T^{2c_R}+2d}{d \cdot T^{2c_R+2K_1}}
\lesssim T^{-2K_1}.
\label{eq:tail_markov_step2_appendix}
\end{align}

By the standard formula for the volume of
the \(d\)-dimensional Euclidean ball,
\[
\operatorname{Vol}\left(\left\{x\in\mathbb R^d:\|x\|_2\le r\right\}\right)
=
\frac{\pi^{d/2}}{\Gamma(d/2+1)}r^d,
\]
see, e.g., \cite[Lecture~1]{ball1997elementary}. Moreover, by the
Stirling-type lower bound for the Gamma function, see \cite[p.~257, Eq.~(6.1.38)]{abramowitz1964handbook}, there exists an absolute constant \(c>0\)
such that
\[
\Gamma(z+1)\ge c\sqrt z\left(\frac ze\right)^z,\qquad z\ge 1.
\]
Taking \(z=d/2\), we obtain, for all \(d\ge 2\),
\[
\operatorname{Vol}\left(\left\{x\in\mathbb R^d:\|x\|_2\le r\right\}\right)
\le
\frac{\sqrt 2}{c\sqrt d}
\left(\frac{2\pi e}{d}\right)^{d/2}r^d
\le
\left(\frac{C r}{\sqrt d}\right)^d .
\]
 Hence, taking
\(r=T^{c_R+K_1}\sqrt d\), we get
\[
\operatorname{Vol}\!\left(\{x\in\mathbb R^d:\|x\|_2\le T^{c_R+K_1}\sqrt d\}\right)
\le
\bigl(C T^{c_R+K_1}\bigr)^d .
\]
Consequently,
\begin{eqnarray}
\exp(-c_6 d\log T)\,
\operatorname{Vol}\!\left(\{x\in\mathbb R^d:\|x\|_2\le T^{c_R+K_1}\sqrt d\}\right) & \le &
\exp\!\left(
-\bigl(c_6-c_R-K_1\bigr)d\log T+d\log C \right) \notag \\
&\le & \exp(-c_2\,d\log T), \label{eq:volume_term_step2_appendix}
\end{eqnarray}
where the last inequality follows by choosing  \(c_6\) sufficiently large so that \(c_6>c_R+K_1+1\), and $c_2$ is a constant.  Putting  \eqref{eq:tail_markov_step2_appendix} and \eqref{eq:volume_term_step2_appendix} into \eqref{eq:RT}, we obtain
\[
r_T
\le
CT^{-2K_1}+\exp(-c_2\,d\log T),
\]
where $C$ is a constant. This completes the proof.
\end{proof}

\section{Proofs of auxiliary lemmas}
\label{app:proof_auxiliary_estimates}
\subsection{Proof of Lemma~\ref{lemma8}}
\label{lemma8_proof}

\begin{proof}
Let
\[
\phi(z):=
\frac{\tau_{t,0}+\tau_{t,K-1}}{2}
+
\frac{\tau_{t,0}-\tau_{t,K-1}}{2}\,z,
\qquad z\in[-1,1].
\]
Then \(\phi\) maps \([-1,1]\) bijectively onto \([\tau_{t,K-1},\tau_{t,0}]\), and
\(\phi(z_j)=\tau_{t,j}\) for all \(0\le j\le K-1\). Let \(\ell_j\) denote the Lagrange basis polynomials associated with the standard
Chebyshev--Lobatto nodes \(\{z_j\}_{j=0}^{K-1}\) on \([-1,1]\), namely,
\begin{equation} \label{lj_define}
    \ell_j(z):=
\frac{\prod_{j':\, j'\neq j}(z-z_{j'})}
{\prod_{j':\, j'\neq j}(z_j-z_{j'})},
\qquad 0\le j\le K-1.
\end{equation}
By the definition of \(\psi_{t,j}\) as in \eqref{eq:basicfun}, a direct calculation gives
\begin{equation} \label{eq:relapsili}
\psi_{t,j}(\phi(z))=\ell_j(z),\qquad 0\le j\le K-1.
\end{equation}
Therefore,
\[
\sup_{\tau\in[\tau_{t,K-1},\tau_{t,0}]}\sum_{j=0}^{K-1}|\psi_{t,j}(\tau)|
=
\sup_{z\in[-1,1]}\sum_{j=0}^{K-1}|\ell_j(z)|.
\]
For the standard Chebyshev--Lobatto nodes on \([-1,1]\), the classical bound 
\[
\sup_{z\in[-1,1]}\sum_{j=0}^{K-1}|\ell_j(z)|
\le
\frac{2}{\pi}\log K+1
\]
holds; see \cite[Chapter~15, Theorem~15.2]{trefethen2019atap}. This completes the proof.
\end{proof}

\subsection{Proof of Lemma~\ref{lemma9}} \label{lemma9_proof}
\begin{proof}
We prove the lemma one by one.  First, \eqref{eq:23a}--\eqref{eq:23e} follow directly from Lemma 1 in \cite{li2025faster}.

\paragraph{Proof of \eqref{eq:23f}.}
For any fixed $2\le t\le T$, it holds $\tau_{t,K-1}=1-\bar{\alpha}_{t-1}\le \tau_{t,i}\le \tau_{t,0}=1-\bar{\alpha}_t$ for all $i=0,1,\ldots,  K-1$. Therefore, for any $0\le i_1,i_2,i_3,i_4\le K-1$,
\[
\left|
\frac{\tau_{t,i_1}-\tau_{t,i_2}}{\tau_{t,i_3}(1-\tau_{t,i_4})}
\right|
\le
\frac{\bar{\alpha}_{t-1}-\bar{\alpha}_t}{(1-\bar{\alpha}_{t-1})\bar{\alpha}_t}
=
\frac{1}{\alpha_t}\cdot
\frac{1-\alpha_t}{1-\bar{\alpha}_{t-1}} \le 8c_1\frac{\log T}{T},
\]
where the last inequality comes from \eqref{eq:23a} and \eqref{eq:23b}. This proves \eqref{eq:23f}.

\paragraph{Proof of \eqref{eq:23g}.}
 In the following, cause \(K\ge 2\) and $0\le j \le K-1$.
From the definition of  $\gamma_{t,j}(\cdot)$ in \eqref{eq:gammtk}, one has
\begin{equation} \label{eq:gammtj}
|\gamma_{t,j}(\tau_{t,i})|
\le \int_{\tau_{t,i}}^{\tau_{t,0}} |\psi_{t,j}(\tau)|\,d\tau
\le
(\tau_{t,0}-\tau_{t,i})
\sup_{\tau\in[\tau_{t,K-1},\,\tau_{t,0}]} |\psi_{t,j}(\tau)| =(\tau_{t,0}-\tau_{t,i}) \sup_{z\in[-1,1]} |\ell_j(z)|,
\end{equation}
where the last equation comes from \eqref{eq:relapsili}. It then suffices to upper bound $\sup_{z\in[-1,1]} |\ell_j(z)|$.
Let
\[
\omega_{K}(z):=\prod_{j=0}^{K-1}(z-z_j),\quad
z_j=\cos\left(\frac{j\pi}{K-1}\right),
\qquad 0\le j\le K-1.
\]
Then with \eqref{lj_define}, we obtain
\[
\ell_j(z)=\frac{\omega_{K}(z)}{\omega_{K}'(z_j)(z-z_j)}.
\]
According to  \cite[Chapter~1, \S1.2.2]{mason2003chebyshev}, one has
\[
\omega_{K}(z)=2^{1-(K-1)}(z^2-1)U_{K-2}(z),
\]
where \(U_{K-2}\) is the Chebyshev polynomial of the second kind, characterized by
$U_{K-2}(\cos\theta)=\frac{\sin((K-1)\theta)}{\sin\theta}$, $\theta\in[0,\pi]$.  A direct computation gives
\[
|\ell_0(z)|=\frac{(1+z)|U_{K-2}(z)|}{2(K-1)},
\qquad
|\ell_{K-1}(z)|=\frac{(1-z)|U_{K-2}(z)|}{2(K-1)},
\]
and, for \(1\le j\le K-2\),
\[
|\ell_j(z)|=\frac{(1-z^2)|U_{K-2}(z)|}{(K-1)|z-z_j|}.
\]

Now write
$z=\cos\theta, z_j=\cos\theta_j,
\theta\in[0,\pi], \theta_j=\frac{j\pi}{K-1}.$
Using
\[
U_{K-2}(\cos\theta)=\frac{\sin((K-1)\theta)}{\sin\theta},
\]
we estimate $\sup_{z\in[-1,1]} |\ell_j(z)|$  separately:

\emph{Endpoint case \(j=0\).}
We have
\[
|\ell_0(z)|
=
\frac{(1+\cos\theta)|\sin((K-1)\theta)|}{2(K-1)\sin\theta}
\le
\frac{(1+\cos\theta)\,(K-1)\sin\theta}{2(K-1)\sin\theta}
=
\frac{1+\cos\theta}{2}
\le 1.
\]

\medskip
\noindent
\emph{Endpoint case \(j=K-1\).}
Similarly,
\[
|\ell_{K-1}(z)|
=
\frac{(1-\cos\theta)|\sin((K-1)\theta)|}{2(K-1)\sin\theta}
\le
\frac{1-\cos\theta}{2}
\le 1.
\]

\medskip
\noindent
\emph{Interior case \(1\le j\le K-2\).}
In this case,
\[
|\ell_j(z)|
=
\frac{\sin\theta\,|\sin((K-1)\theta)|}{(K-1)|\cos\theta-\cos\theta_j|} \le \frac{\sin\theta\,|\sin(\theta-\theta_j)|}{|\cos\theta-\cos\theta_j|}  \le \frac{\sin\theta\left|\cos\frac{\theta-\theta_j}{2}\right|}
{\sin\frac{\theta+\theta_j}{2}}.
\]
where the first inequality follows from the fact that \((K-1)\theta_j=j\pi\), and hence
\[
|\sin((K-1)\theta)|
=
|\sin((K-1)(\theta-\theta_j))|
\le
(K-1)\,|\sin(\theta-\theta_j)|.
\]
On the other hand, note that 
\[
\sin\theta+\sin\theta_j
=
2\sin\frac{\theta+\theta_j}{2}\cos\frac{\theta-\theta_j}{2}
\ge \sin\theta,
\]
because \(\theta,\theta_j\in[0,\pi]\) implies
$\cos\frac{\theta-\theta_j}{2}\ge 0$, 
$\sin\theta_j\ge 0$. This immediately gives
\[
|\ell_j(z)|
\le
2\cos^2\frac{\theta-\theta_j}{2}
\le 2.
\]
Putting it into \eqref{eq:gammtj}, we have
\[
|\gamma_{t,j}(\tau_{t,i})|
\le
2(\tau_{t,0}-\tau_{t,i}),
\]
which proves \eqref{eq:23g}

\paragraph{Proof of \eqref{eq:23h}.}
For every $0\le i\le K-1$ and $2\le t\le T$, since
$\tau_{t,K-1}=1-\bar{\alpha}_{t-1}\le \tau_{t,i}\le \tau_{t,0}=1-\bar{\alpha}_t$, 
we have
$\bar{\alpha}_t=1-\tau_{t,0}\le 1-\tau_{t,i}\le 1-\tau_{t,K-1}=\bar{\alpha}_{t-1}$, 
and
$1-\bar{\alpha}_{t-1}=\tau_{t,K-1}\le \tau_{t,i}\le \tau_{t,0}=1-\bar{\alpha}_t$.
Therefore, it suffices to show
\[
\bar{\alpha}_t\asymp \bar{\alpha}_{t-1}
\qquad\text{and}\qquad
1-\bar{\alpha}_{t-1}\asymp 1-\bar{\alpha}_t,
\]
which follow directly from \eqref{eq:23a}and \eqref{eq:23c}. This proves \eqref{eq:23h}.

\paragraph{Proof of \eqref{eq:23i}.}
For $K\ge2$, by Lemma~\ref{lemma8}, we have
\[
\sum_{j=0}^{K-1}|\psi_{t,j} (\tau)|\le
\frac{2}{\pi}\log K + 1
\le
C_4\log K.
\]
Hence, for any $0\le i\le K-1$,
\begin{align*}
\sum_{j=0}^{K-1}|\gamma_{t,j}(\tau_{t,i})|
&\le
\int_{\tau_{t,i}}^{\tau_{t,0}}\sum_{j=0}^{K-1}|\psi_{t,j}(\tau)|\,d\tau
\le
C_4(\tau_{t,0}-\tau_{t,i})\log K,
\end{align*}
where $\gamma_{t,j}$ is defined in\eqref{eq:gammtk}. This proves\eqref{eq:23i}.

\paragraph{Proof of \eqref{eq:23j}.}
Using the definition of $A_{ji}^{(t)}=\frac{1}{2}\gamma_{t,j}\left(\tau_{t,i}\right)
\left(1-\tau_{t,j}\right)^{-3/2}$
together with \eqref{eq:23h} and \eqref{eq:23i}, we obtain
\begin{align*}
\sum_{j=0}^{K-1}|A_{ji}^{(t)}|
&=
\frac{1}{2}\sum_{j=0}^{K-1}\frac{|\gamma_{t,j}(\tau_{t,i})|}{(1-\tau_{t,j})^{3/2}}
\le
\frac{1}{2}\frac{1}{(1-\tau_{t,0})^{3/2}}
\sum_{j=0}^{K-1}|\gamma_{t,j}(\tau_{t,i})|
\le
C_4\frac{(\tau_{t,0}-\tau_{t,i})\log K}{(1-\tau_{t,0})^{3/2}}.
\end{align*}
This proves \eqref{eq:23j}.

\paragraph{Proof of \eqref{eq:23k}.}
According to \eqref{eq:23g}, one has $|\gamma_{t,j}(\tau_{t,i})| \le 2(\tau_{t,0}-\tau_{t,i})$. Therefore, 
\[
|A_{ji}^{(t)}|
=
\frac{1}{2}\frac{|\gamma_{t,j}(\tau_{t,i})|}{(1-\tau_{t,j})^{3/2}}
\le
C_4\frac{\tau_{t,0}-\tau_{t,i}}{(1-\tau_{t,0})^{3/2}},
\]
which implies
\begin{align*}
\sum_{j=0}^{K-1}|A_{ji}^{(t)}|^2
&\le
\Bigl(\max_{0\le j\le K-1}|A_{ji}^{(t)}|\Bigr)
\sum_{j=0}^{K-1}|A_{ji}^{(t)}|
\le
C_5\frac{(\tau_{t,0}-\tau_{t,i})^2\log K}{(1-\tau_{t,0})^{3}}.
\end{align*}
This proves \eqref{eq:23k}.
\end{proof}

\subsection{Proof of Lemma~\ref{lemma10}}
\label{lemma10_proof}
\begin{proof}
For any \(r>0\), Bayes' rule gives
\begin{eqnarray}
\mathbb{P}\!\left(
\left\|\sqrt{1-\tau}X_0-y\right\|_2\ge r
\,\middle|\, \bar{X}_\tau=y
\right) &=&  \frac{ \int_{\{\|\sqrt{1-\tau}x_0-y\|_2\ge r\}}
p_{X_0}(x_0)\,p_{\bar{X}_\tau\mid X_0}(y\mid x_0)\,dx_0
}{
p_{\bar{X}_\tau}(y)
} \notag \\
&\le & \frac{1}{p_{\bar{X}_\tau}(y)} \cdot 
\frac{1}{(2\pi\tau)^{d/2}}
\exp\!\left(-\frac{r^2}{2\tau}\right)  \notag \\
&\le & \exp\!\left(
\theta_\tau(y)d\log T+\frac{c_0}{2}d\log T-\frac{r^2}{2\tau}
\right), \label{eq:Bayes}
\end{eqnarray}
where the first inequality follows from  the fact that 
\[
p_{\bar{X}_\tau\mid X_0}(y\mid x_0)
=
\frac{1}{(2\pi\tau)^{d/2}}
\exp\!\left(
-\frac{\|y-\sqrt{1-\tau}\,x_0\|_2^2}{2\tau}
\right),
\]
and the second inequality comes from the definition of \(\theta_\tau(y)\) that  $ \frac{1}{p_{\bar{X}_\tau}(y)} \le
\exp\!\left(\theta_\tau(y)d\log T\right)$ and the fact \(\tau\ge T^{-c_0}\). Choosing 
\[
r=5c_5\sqrt{\theta_\tau(y)\,d\,\tau\log T}
\]
for some constant $c_5\ge2$, then
\[
\mathbb{P}\!\left(
\left\|\sqrt{1-\tau}X_0-y\right\|_2
\ge
5c_5\sqrt{\theta_\tau(y)\,d\,\tau\log T}
\,\middle|\, \bar{X}_\tau=y
\right)
\le
\exp\!\left(
\left[
1+\frac{c_0}{2\theta_\tau(y)}-\frac{25}{2}c_5^2
\right]
\theta_\tau(y)d\log T
\right).
\]
Since \(\theta_\tau(y)\ge c_6\ge c_0\), we have
\[
\frac{c_0}{2\theta_\tau(y)}\le \frac12,
\]
and therefore
\[
1+\frac{c_0}{2\theta_\tau(y)}-\frac{25}{2}c_5^2
\le
\frac32-\frac{25}{2}c_5^2
\le
-c_5^2
\qquad\text{for all }c_5\ge2.
\]
Thus
\begin{equation}\label{eq:mod_ct_special_tail_clean}
\mathbb{P}\!\left(
\left\|\sqrt{1-\tau}X_0-y\right\|_2
\ge
5c_5\sqrt{\theta_\tau(y)\,d\,\tau\log T}
\,\middle|\, \bar{X}_\tau=y
\right)
\le
\exp\!\left(-c_5^2\theta_\tau(y)d\log T\right).
\end{equation}
This gives \eqref{eq:mod_lemma2_ct_m1}. Next, note that  \(\theta_\tau(y)\ge c_6\ge c_0\). It then follows from \eqref{eq:Bayes} that 
\begin{equation*}
\mathbb{P}\!\left(
\left\|\sqrt{1-\tau}X_0-y\right\|_2\ge u
\,\middle|\, \bar{X}_\tau=y
\right)
\le
\exp\!\left(
\frac32\,\theta_\tau(y)d\log T-\frac{u^2}{2\tau}
\right),
\qquad \forall u>0.
\end{equation*}
For every integer \(k\ge1\), the tail-integral identity yields
\[
\mathbb{E}\!\left[
\left\|\sqrt{1-\tau}X_0-y\right\|_2^k
\,\middle|\, \bar{X}_\tau=y
\right]
=
k\int_0^\infty
u^{k-1}
\mathbb{P}\!\left(
\left\|\sqrt{1-\tau}X_0-y\right\|_2\ge u
\,\middle|\, \bar{X}_\tau=y
\right)\,du.
\]
In particular, for \(k=1\),  we have
\begin{eqnarray*}
&& \mathbb{E}\!\left[
\left\|\sqrt{1-\tau}X_0-y\right\|_2
\,\middle|\, \bar{X}_\tau=y
\right] \\
&\le&  10\sqrt{\theta_\tau(y)\,d\,\tau\log T} +
\int_{10\sqrt{\theta_\tau(y)\,d\,\tau\log T}}^\infty
\mathbb{P}\!\left(
\left\|\sqrt{1-\tau}X_0-y\right\|_2\ge u
\,\middle|\, \bar{X}_\tau=y
\right)\,du \\
&\le&  10\sqrt{\theta_\tau(y)\,d\,\tau\log T} +
\int_{10\sqrt{\theta_\tau(y)\,d\,\tau\log T}}^\infty  \exp\!\left(
\frac32\,\theta_\tau(y)d\log T-\frac{u^2}{2\tau}
\right) \,du \\
&\lesssim&  \sqrt{\theta_\tau(y)\,d\,\tau\log T},
\end{eqnarray*}
where the last inequality comes from 
 the Gaussian tail estimate
\[
\int_a^\infty e^{-u^2/(2\tau)}\,du
\le \frac{\tau}{a}e^{-a^2/(2\tau)},
\qquad \mbox{for all} \quad  a>0.
\]
This proves \eqref{eq:mod_lemma2_ct_m1}. For the cases  \(k=2,3,4\), the proof is similar, and so we omit it.
\end{proof}

\subsection{Proof of Lemma~\ref{lemma13}} \label{proof_lemma13}
\begin{proof}
The proof follows the argument used for Claims \textup{(64a)} and \textup{(64b)}
in Appendix~A of \cite{li2025faster}, while keeping track of the dependence on
the order $k$ explicitly. 
Define
\[
u_k:=
\frac{\partial^k}{\partial\tau^k}
\left(
\frac{s_\tau^\star(x_\tau^\star)}{(1-\tau)^{3/2}}
\right),
\qquad
\theta:=\theta_\tau(x_\tau^\star).
\]
Then for sufficiently small $\delta$,  the Taylor expansion  gives
\begin{equation*}   
   \frac{s_{\tau+\delta}^{\star}\left(x_{\tau+\delta}^{\star}\right)}{(1-\tau-\delta)^{3 / 2}}-\frac{s_\tau^{\star}\left(x_\tau^{\star}\right)}{(1-\tau)^{3 / 2}}= \sum_{k=1}^{\infty} \frac{\delta^k}{k!} u_k.
\end{equation*}
As shown in \cite[Equation (74)]{li2025faster}, and recall $u_\tau^\star:=x_\tau^\star/\sqrt{1-\tau}$ one has
\begin{equation}   \label{uk}
\frac{s_{\tau+\delta}^\star(x_{\tau+\delta}^\star)}{(1-\tau-\delta)^{3/2}} =
-\frac{1}{(\tau+\delta)(1-\tau-\delta)}\left(\frac{x_{\tau+\delta}^{\star}}{\sqrt{1-\tau-\delta}}-\frac{x_\tau^{\star}}{\sqrt{1-\tau}}+\frac{\displaystyle
\int p_{X_0|\bar{X}_\tau}(x_0|x_\tau^\star)\,
e^{\Delta}\,
(u_\tau^\star-x_0)\,dx_0}
{\displaystyle
\int p_{X_0|\bar{X}_\tau}(x_0|x_\tau^\star)\,
e^{\Delta}\,dx_0}\right),
\end{equation}
where
\[
\Delta:=\frac{(1-\tau)\left\|x_\tau^{\star} / \sqrt{1-\tau}-x_0\right\|_2^2}{\tau}-\frac{(1-\tau-\delta)\left\|x_{\tau+\delta}^{\star} / \sqrt{1-\tau-\delta}-x_0\right\|_2^2}{\tau+\delta}=:
\sum_{k\ge1}\frac{\delta^k}{k!}v_k.
\]
 We next prove that there exist universal
constants $M_1>1$ and $C_{\mathrm s},\bar{C}>0$ such that
\begin{equation} \label{uk1}
\|u_k\|_2 \le
C_{\mathrm s}(M_1K)^k k!\,
\sqrt{\frac{d\theta\log T}{\tau(1-\tau)^3}}
\left(
\frac{d\theta\log T}{\tau(1-\tau)}
\right)^k,
\qquad 0\le k\le K,
\end{equation}
and, for every $C\ge 2$,
\begin{equation} \label{vk1}
|v_k| \le
\frac{\bar{C}C^2}{M_1K}
(M_1K)^k k!\,
\left(
\frac{d\theta\log T}{\tau(1-\tau)}
\right)^k,
\qquad 1\le k\le K,
\end{equation}
provided that
\[
\|u_\tau^\star-x_0\|_2
\le
5C\sqrt{\frac{d\theta\tau\log T}{1-\tau}}.
\]
Indeed, for $k=0$, the results hold directly from Lemma~\ref{lemma10} (see \cite[Equation (75a)]{li2025faster} for more details). Suppose that \eqref{uk1} and \eqref{vk1}  hold for all $k\le k_0$. Then, by the same arguments to \cite{li2025faster}, we have
\begin{eqnarray*}
&& \left|v_{k_0+1}\right| \\
&\leq &  (k_0+1)!\Biggl[\tau^{-k_0}\left\|u_{\tau}^{\star}-x_0\right\|_2^2+\frac{1}{\left(k_0+1\right)!} \frac{1-\tau}{\tau}\left\|u_{\tau}^{\star}-x_0\right\|_2\left\|u_{k_0}\right\|_2+\|u_\tau^\star-x_0\|_2
\sum_{\ell=1}^{k_0} \frac{\tau^{-\ell-1}}{(k_0+1-\ell)!}\|u_{k_0-\ell}\|_2\\
&& +\frac{1-\tau}{4\tau}\sum_{\ell=1}^{k_0}
\frac{1}{\ell!(k_0+1-\ell)!}\|u_{\ell-1}\|_2\|u_{k_0-\ell}\|_2+\sum_{\ell=1}^{k_0-1}\sum_{j=1}^{k_0-\ell}
\frac{1}{j!(k_0+1-\ell-j)!}\tau^{-\ell-1}
\|u_{j-1}\|_2\|u_{k_0-\ell-j}\|_2\Biggl].
\end{eqnarray*}
Using the induction hypothesis \eqref{uk1} and \eqref{vk1}, we obtain
\begin{eqnarray*}
\left|v_{k_0+1}\right| &\leq &  (k_0+1)!\Biggl[ 25C^2
+\frac{5C_{\mathrm s}C}{k_0+1}(M_1K)^{k_0}
+5C_{\mathrm s}C(M_1K)^{k_0}
+C_{\mathrm s}^2(M_1K)^{k_0-1}
+2C_{\mathrm s}^2(M_1K)^{k_0}
\Biggr]
\left(\frac{d \theta \log T}{\tau(1-\tau)}\right)^{k_0+1} \\
&\le & \frac{\bar{C}C^2}{M_1K}
(M_1K)^{k_0+1} (k_0+1)! \left(\frac{d \theta \log T}{\tau(1-\tau)}\right)^{k_0+1}, 
\end{eqnarray*}
provided $\bar{C}>5(1+C_{s}+C^2_{s})$ and then $M_1 \ge 100 \bar{C}c_7^2$. Therefore, we have verified \eqref{vk1} for $k=k_0+1$. Next, we prove that \eqref{uk1} holds for $k=k_0+1$. Let $e^{\Delta}=: \sum_{k=0}^{\infty} \frac{\delta^k}{k!} w_k$ denote the Taylor expansion of $e^{\Delta}$. Then one can show that $w_0=1$ and for all $1 \leq k \leq k_0+1$,
\begin{align}
    \left|w_k\right|
    &=
    \left|\sum_{\ell=1}^k \frac{1}{\ell!}\sum_{j_1+\cdots+j_{\ell}=k} \frac{k!}{j_{1}!\cdots j_{\ell}!} v_{j_1} \cdots v_{j_{\ell}}\right| \le
    (M_1K)^k k!
    \left(\frac{d \theta \log T}{\tau(1-\tau)}\right)^k
    \sum_{\ell=1}^k
    \frac{1}{\ell!}
    \binom{k-1}{\ell-1}
    \left(\frac{\bar{C}C^2}{M_1K}\right)^\ell\nonumber \\
    &\le
    \frac{\bar{C}C^2}{M_1K}
    \left(1+\frac{\bar{C}C^2}{M_1K}\right)^{k-1}
    (M_1K)^k k!
    \left(\frac{d \theta \log T}{\tau(1-\tau)}\right)^k \le \frac{C}{K}
(MK)^k k!
\left(\frac{d \theta \log T}{\tau(1-\tau)}\right)^k.
    \label{wk}
\end{align}
Denote
\[
\int_{x_0} p_{X_0 \mid \bar{X}_\tau}
\left(x_0 \mid x_\tau^{\star}\right)
\exp(\Delta)\left(u_\tau^{\star}-x_0\right)\mathrm{d}x_0
:= \sum_{k=0}^{\infty}\frac{\delta^k}{k!}a_k, 
\quad
\int_{x_0} p_{X_0 \mid \bar{X}_\tau}
\left(x_0 \mid x_\tau^{\star}\right)
\exp(\Delta)\mathrm{d}x_0
:= \sum_{k=0}^{\infty}\frac{\delta^k}{k!}b_k .
\]
Then for $0\le k\le k_0+1$, similar to \cite{li2025faster}, one has
\[
\begin{aligned}
a_0
=&
\int_{x_0} p_{X_0 \mid \bar{X}_\tau}
\left(x_0 \mid x_\tau^{\star}\right)
\left(u_\tau^{\star}-x_0\right) w_0 \,\mathrm{d}x_0  =
\int_{x_0} p_{X_0 \mid \bar{X}_\tau}
\left(x_0 \mid x_\tau^{\star}\right)
\left(u_\tau^{\star}-x_0\right)\mathrm{d}x_0,\\
&\left\|a_k\right\|_2
\le
C' (M_1K)^k k!
\sqrt{\frac{d\theta\tau\log T}{1-\tau}}
\left(
\frac{d\theta\log T}{\tau(1-\tau)}
\right)^k,
\qquad 0\le k\le k_0+1,\\
&b_0=1,\qquad
\left|b_k\right|
\le
\frac{C''}{K}
(M_1K)^k k!
\left(
\frac{d\theta\log T}{\tau(1-\tau)}
\right)^k,
\qquad 1\le k\le k_0+1.
\end{aligned}
\]
Here $C' =O(1+c_7)$ is an absolute constant, and $0<C'' \le 1/4$.  Next define $d_k$ through the Taylor expansion
\[
 \frac{\sum_{m=0}^{\infty}\frac{\delta^m}{m!}a_m}{\sum_{\ell=0}^{\infty}\frac{\delta^\ell}{\ell!}b_\ell}=\frac{
\int_{x_0} p_{X_0 \mid \bar{X}_\tau}\left(x_0 \mid x_\tau^{\star}\right)
\exp(\Delta)\left(u_{\tau}^{\star}-x_0\right)\mathrm{d}x_0}
{
\int_{x_0} p_{X_0 \mid \bar{X}_\tau}\left(x_0 \mid x_\tau^{\star}\right)
\exp(\Delta)\mathrm{d}x_0}:=\sum_{k=0}^{\infty}\frac{\delta^k}{k!}d_k.
\]
Then
\begin{align}
    d_0
    =
    \int_{x_0} p_{X_0 \mid \bar{X}_\tau}\left(x_0 \mid x_\tau^{\star}\right)
    \left(\frac{x_\tau^{\star}}{\sqrt{1-\tau}}-x_0\right) \mathrm{d} x_0,\quad
    d_k
    =
    a_k-\sum_{\ell=0}^{k-1}\binom{k}{\ell}d_{\ell} b_{k-\ell},
    \qquad 1\le k\le k_0+1.\nonumber
\end{align}
Choosing $C''' \ge 2C'$, we obtain
\begin{align}
\left\|d_k\right\|_2
&\le
\left(
C'+\frac{k}{K}C''C'''
\right)
(M_1K)^k k!
\sqrt{\frac{d \theta \tau \log T}{1-\tau}}
\left(\frac{d \theta \log T}{\tau(1-\tau)}\right)^k\nonumber \\
&\le
C'''(M_1K)^k k!
\sqrt{\frac{d \theta \tau \log T}{1-\tau}}
\left(\frac{d \theta \log T}{\tau(1-\tau)}\right)^k,\quad0 \leq k \leq k_0+1.
\label{dk}
\end{align}
Similarly to (80) in \cite{li2025faster}, through \eqref{uk} we have
\[
\sum_{k=1}^{\infty} \frac{\delta^k}{k!} u_k=\sum_{k=1}^{\infty}\left(\frac{1}{\tau(1-\tau)} \frac{u_{k-1}-2 d_k}{2 k!}+\sum_{\ell=1}^k\left[(-\tau)^{-\ell-1}-(1-\tau)^{-\ell-1}\right] \frac{2 d_{k-\ell}-u_{k-1-\ell}}{2(k-\ell)!}\right) \delta^k.
\]
By hypothesis \eqref{uk1} and \eqref{dk} hold for $0\le k<k_0$, we now prove \eqref{uk1} for $k=k_0+1$
\begin{align}
    \left\|u_{k_0+1}\right\|_2
    &=
    \left\|\frac{1}{\tau(1-\tau)}\left(\frac{u_{k_0}}{2}-d_{k_0+1}\right)
    +\sum_{\ell=1}^{k_0+1}\left[(-\tau)^{-\ell-1}-(1-\tau)^{-\ell-1}\right]
    \frac{\left(k_0+1\right)!}{2\left(k_0+1-\ell\right)!}
    \left(2 d_{k_0+1-\ell}-u_{k_0-\ell}\right)\right\|_2\nonumber\\
    &\le
    C'''(M_1K)^{k_0+1}(k_0+1)!
    \sqrt{\frac{d \theta \log T}{\tau(1-\tau)^3}}
    \left(\frac{d \theta \log T}{\tau(1-\tau)}\right)^{k_0+1}+
    \frac{C_{\mathrm s}(M_1K)^{k_0}(k_0+1)!}{2}
    \sqrt{\frac{d \theta \log T}{\tau(1-\tau)^3}}
    \left(\frac{d \theta \log T}{\tau(1-\tau)}\right)^{k_0+1}\nonumber\\
    &\quad+
    \frac{(C'''+C_{\mathrm s})K}{M_1K}
    (M_1K)^{k_0+1}(k_0+1)!
    \sqrt{\frac{d \theta \log T}{\tau(1-\tau)^3}}
    \left(\frac{d \theta \log T}{\tau(1-\tau)}\right)^{k_0}\nonumber\\
    &\le
    C_{\mathrm s}(M_1K)^{k_0+1}(k_0+1)!
    \sqrt{\frac{d \theta \log T}{\tau(1-\tau)^3}}
    \left(\frac{d \theta \log T}{\tau(1-\tau)}\right)^{k_0+1}.
\end{align}
where the last inequality follows by choosing $C_{\mathrm s}\ge 2+C'''$ and $M_1\geq 100 C^2_{\mathrm{s}} c_7^2$. Thus we proved \eqref{uk1} holds for any $0\le k\le K$.

We next prove \eqref{eq:lemma132}. 
 By Tweedie's formula \cite{efron2011tweedie}, we have
\begin{align*}
\frac{1}{(1-\tau)^{3/2}}
\frac{\partial s_\tau^\star(x_\tau^\star)}{\partial u_\tau^\star}
=
-\frac{1}{\tau(1-\tau)} I_d
+
\frac{1}{\tau^2}
\operatorname{Cov}\left(
\left.
\frac{x_\tau^\star}{\sqrt{1-\tau}}-X_0
\,\right|\,
\bar{X}_\tau=x_\tau^\star
\right).
\end{align*}

Denote
$m_k:=
\frac{\partial^k}{\partial \tau^k}
\left[
\frac{1}{(1-\tau)^{3/2}}
\frac{\partial s_\tau^\star(x_\tau^\star)}{\partial u_\tau^\star}
\right].$
For sufficiently small $\delta$, we write
\begin{align*}
\sum_{k=1}^{\infty}\frac{\delta^k}{k!}m_k
&=
\frac{1}{(1-\tau-\delta)^{3/2}}
\frac{\partial s_{\tau+\delta}^\star(x_{\tau+\delta}^\star)}
{\partial u_{\tau+\delta}^\star}
-
\frac{1}{(1-\tau)^{3/2}}
\frac{\partial s_\tau^\star(x_\tau^\star)}
{\partial u_\tau^\star},
\end{align*}
where
\begin{align*}
&\frac{1}{(1-\tau-\delta)^{3/2}}
\frac{\partial s_{\tau+\delta}^\star(x_{\tau+\delta}^\star)}
{\partial u_{\tau+\delta}^\star}  =
-\frac{I_d}{(\tau+\delta)(1-\tau-\delta)}
+\frac{1}{(\tau+\delta)^2}
\Biggl[
\frac{
\int p_{X_0 \mid \bar{X}_\tau}(x_0 \mid x_\tau^\star)
e^{\Delta}
\left(u_\tau^\star-x_0\right)
\left(u_\tau^\star-x_0\right)^\top
\,\mathrm d x_0}
{
\int p_{X_0 \mid \bar{X}_\tau}(x_0 \mid x_\tau^\star)
e^{\Delta}\,\mathrm d x_0
} \\
&\qquad\qquad -
\left(
\frac{
\int p_{X_0 \mid \bar{X}_\tau}(x_0 \mid x_\tau^\star)
e^{\Delta}
\left(u_\tau^\star-x_0\right)
\,\mathrm d x_0}
{
\int p_{X_0 \mid \bar{X}_\tau}(x_0 \mid x_\tau^\star)
e^{\Delta}\,\mathrm d x_0
}
\right)
\left(
\frac{
\int p_{X_0 \mid \bar{X}_\tau}(x_0 \mid x_\tau^\star)
e^{\Delta}
\left(u_\tau^\star-x_0\right)
\,\mathrm d x_0}
{
\int p_{X_0 \mid \bar{X}_\tau}(x_0 \mid x_\tau^\star)
e^{\Delta}\,\mathrm d x_0
}
\right)^\top
\Biggr],
\end{align*}
and
$\Delta=\sum_{k\ge1}\frac{\delta^k}{k!}v_k.$
Applying the same argument as above, gives
\begin{equation}
\left\|m_k\right\|
=
\left\|
\frac{\partial^k}{\partial\tau^k}
\left[
\frac{1}{(1-\tau)^{3/2}}
\frac{\partial s_\tau^\star(x_\tau^\star)}{\partial u_\tau^\star}
\right]
\right\|
\le
\widetilde{C}(M_2K)^k k!
\left(
\frac{d\theta\log T}{\tau(1-\tau)}
\right)^{k+1},
\quad M_2\ge2,\quad0\le k\le K.
\label{mk}
\end{equation}

On the other hand, recall that $J_\tau^{\star}:=\frac{\partial\left(x_\tau^{\star} / \sqrt{1-\tau}\right)}{\partial\left(x_{\tau_{t, 0}}^{\star} / \sqrt{1-\tau_{t, 0}}\right)}$ and \eqref{eq:exact_jacobian_flow}, then the exact Jacobian satisfies
\[
\frac{d}{d\tau}J_\tau^\star
=
-\frac12
\left[
\frac{1}{(1-\tau)^{3/2}}
\frac{\partial s_\tau^\star(x_\tau^\star)}{\partial u_\tau^\star}
\right]
J_\tau^\star,
\qquad
J_{\tau_{t,0}}^\star=I.
\]
Because \eqref{eq:gs_stage2_Jstar_uniform} gives
$\|J_\tau^\star\|\le c_8$. Then we choose $\widehat{C}\ge c_8$ and assume that $\left\|\frac{\partial^k}{\partial \tau^k} J_\tau^{\star}\right\|\le \widehat{C} k!\left(M_2 K\right)^k\left(\frac{d \theta \log T}{\tau(1-\tau)}\right)^k$ for $0\le k\le k_0$. Differentiating the preceding identity
$k-1$ times yields
\begin{align} \label{diff_J}
    \frac{\partial^k}{\partial\tau^k}J_\tau^\star
=
-\frac12
\sum_{\ell=0}^{k-1}
\binom{k-1}{\ell}
m_\ell \frac{\partial^{k-1-\ell}}{\partial\tau^{k-1-\ell}}J_\tau^\star.
\end{align}
Using \eqref{mk} and assumption above, we prove \eqref{diff_J} for $k=k_0+1$
\begin{align*}
\left\|\frac{\partial^{k_0+1}}{\partial\tau^{k_0+1}}J_\tau^\star \right\|
&\le
\frac12
\sum_{\ell=0}^{k_0}
\binom{k_0}{\ell}
\|m_\ell\|\,
\left\|\frac{\partial^{k_0-\ell}}{\partial\tau^{k_0-\ell}}J_\tau^\star\right\| \\
&\le
\frac12 \widetilde{C}\widehat{C}
\sum_{\ell=0}^{k_0}
\binom{k_0}{\ell}
(M_2K)^\ell \ell!
\left(
\frac{d\theta\log T}{\tau(1-\tau)}
\right)^{\ell+1}
(M_2K)^{k_0-\ell}(k_0-\ell)!
\left(
\frac{d\theta\log T}{\tau(1-\tau)}
\right)^{k_0-\ell} \\
&\le
\frac{\widetilde{C}}{2M_2K}
\widehat{C} (k_0+1)!(M_2K)^{k_0+1}
\left(
\frac{d\theta\log T}{\tau(1-\tau)}
\right)^{k_0+1} \\
&\le \widehat{C} (k_0+1)!(M_2K)^{k_0+1}
\left(
\frac{d\theta\log T}{\tau(1-\tau)}
\right)^{k_0+1}.
\end{align*}
Finally, applying Leibniz' rule to
$\left[
\frac{1}{(1-\tau)^{3/2}}
\frac{\partial s_\tau^\star(x_\tau^\star)}{\partial u_\tau^\star}
\right]
J_\tau^\star$ and combining with \eqref{mk} and the bound of $\left\|\frac{\partial^k}{\partial \tau^k} J_\tau^{\star}\right\|$ above, for any $0\le k \le K$ we get
\begin{align*}
\left\|
\frac{\partial^k}{\partial\tau^k}
\left[
\frac{1}{(1-\tau)^{3/2}}
\frac{\partial s_\tau^\star(x_\tau^\star)}{\partial u_\tau^\star}
J_\tau^\star
\right]
\right\| \le
\sum_{\ell=0}^{k}
\binom{k}{\ell}
\|m_\ell\|\,
\|\frac{\partial^{k-\ell}}{\partial \tau^{k-\ell}} J_\tau^{\star}\| d\le
C_{\mathrm J}(M_2K)^{k+1} k!
\left(
\frac{d\theta\log T}{\tau(1-\tau)}
\right)^{k+1},
\end{align*}
where the last inequality follows from$C_{\mathrm{J}} \geq 2\widetilde{C} \widehat{C}$, $k+1\le K+1\le 2K$ and $M_2\ge2$.
This proves \eqref{eq:lemma132}.
\end{proof}

\subsection{Proof of Lemma~\ref{lemma14}}
\label{proof_lemma14}
\begin{proof}
For convenience, let
$\theta_t:=\theta_{\tau_{t,0}}(x_{\tau_{t,0}})$.
Recalling that \(x_\tau^\star\) is the solution of ODE \eqref{ODE} at \(\tau\) with
the initial condition \(x_{\tau_{t,0}}^\star=x_{\tau_{t,0}}\), we know from
Lemma~\ref{lemma12} that
\[
-\log p_{\bar X_\tau}(x_\tau^\star)
\le
2\theta_t d\log T,
\qquad
\tau_{t,K-1}\le \tau\le \tau_{t,0}.
\]
The same argument as in \cite[Eq. (91)]{li2025faster} gives, for all \(0\le i\le K-1\) and all \(\lambda\in[0,1]\),
\begin{equation} \label{eq:lemma14_initial_all_nodes}
-\log p_{\bar X_{\tau_{t,i}}}
\left(
\lambda x_{\tau_{t,i}}^{(0)}
+
(1-\lambda)x_{\tau_{t,i}}^\star
\right)
\le
2.1d\theta_t\log T.
\end{equation}
Combining \eqref{eq:lemma14_initial_all_nodes} with Lemma~\ref{lemma11}, we obtain
\begin{equation} \label{eq:lemma14_initial_score_jacobian}
\left\|
\frac{
\partial s_{\tau_{t,i}}^\star
\left(
\lambda x_{\tau_{t,i}}^{(0)}
+
(1-\lambda)x_{\tau_{t,i}}^\star
\right)
}{\partial x}
\right\|
\lesssim
\frac{d\theta_t\log T}{\tau_{t,i}},
\qquad
0\le i\le K-1,\quad \lambda\in[0,1].
\end{equation}
For any \(0\le n\le N-1\) and \(0\le i\le K-1\), we define
\begin{align} \label{u_n}
u_{t,i}^{(n+1)}
&:=
\sqrt{\frac{1-\tau_{t,0}}{1-\tau_{t,i}}}\,
x_{\tau_{t,i}}^{(n+1)}
-
x_{\tau_{t,0}} =
\sqrt{1-\tau_{t,0}}
\left[
\sum_{j=0}^{i-1}
A_{ji}^{(t)}
s_{\tau_{t,j}}\!\left(x_{\tau_{t,j}}^{(n+1)}\right)
+
\sum_{j=i}^{K-1}
A_{ji}^{(t)}
s_{\tau_{t,j}}\!\left(x_{\tau_{t,j}}^{(n)}\right)
\right],
\end{align}
where $A_{ji}^{(t)}=\frac{1}{2}\gamma_{t,j}\left(\tau_{t,i}\right) \left(1-\tau_{t,j}\right)^{-3/2}$.
Similarly, let
\begin{align} \label{u_star}
   u_{t,i}^\star
:=
\sqrt{\frac{1-\tau_{t,0}}{1-\tau_{t,i}}}\,
x_{\tau_{t,i}}^\star
-
x_{\tau_{t,0}}=
-\sqrt{1-\tau_{t,0}}
\int_{\tau_{t,0}}^{\tau_{t,i}}
\frac{1}{2(1-\tau)^{3/2}}
s_\tau^\star(x_\tau^\star)\,d\tau,
\end{align}
where the second inequality follows from \eqref{ODE_score}.
Next, we shall prove the following two estimates simultaneously by induction:  for each \(0\le n\le N-1\) and each \(0\le i\le K-1\), it holds
\begin{equation} \label{eq:lemma14_induction_goal_u}
\left\|u_{t,i}^{(n+1)}-u_{t,i}^\star\right\|_2
\le
C_{10}
\frac{\tau_{t,0}\log T}{T}
\sqrt{\log K\sum_{m,j}\left(\varepsilon_{\mathrm{score}, t, j}^{(m)}\left(x_{\tau_{t, j}}^{(m)}\right)\right)^2}\,
+
C_{10}
\sqrt{\frac{d\theta_t\tau_{t,0}\log^3T}{T^2}}
\left(
\frac{d\theta_t\log^2TK}{T}
\right)
\end{equation} 
 and for all \(\lambda\in[0,1]\),
\begin{equation} \label{eq:lemma14_induction_goal_p}
-\log p_{\bar X_{\tau_{t,i}}}
\left(
\lambda x_{\tau_{t,i}}^{(n+1)}
+
(1-\lambda)x_{\tau_{t,i}}^\star
\right)
\le
2.1d\theta_t\log T.
\end{equation}
Here,  $\varepsilon_{\mathrm{score}, t, j}^{(m)}(\cdot)$ is defined in \eqref{def_vareps} and 
$C_{10}>0$ is a sufficiently large constant.
 Moreover, with iteration \eqref{eq:gs_update} we have
\[
x_{\tau_{t,0}}^{(n)}
=
x_{\tau_{t,0}}^\star
=
x_{\tau_{t,0}},
\qquad
\text{for all} \quad n\ge 0
\]
Thus for $n=0,i=0$, we have 
\begin{align}
    u_{t,0}^{(1)}=u_{t,0}^\star=0,&\quad\left\|u_{t,0}^{(1)}-u_{t,0}^\star\right\|_2=0,
\label{eq:lemma14_i0_u}\\
-\log p_{\bar{X}_{\tau_{t, 0}}}\left(\lambda x_{\tau_{t, 0}}^{(1)}+(1-\lambda) x_{\tau_{t, 0}}^{\star}\right)&=-\log p_{\bar{X}_{\tau_{t, 0}}}\left( x_{\tau_{t, 0}}^{\star}\right)\le 2.1d\theta_t\log T, \label{eq:lemma14_i0_p}
\end{align}
where \eqref{eq:lemma14_i0_p} holds due to \eqref{eq:lemma14_initial_all_nodes}. This implies  \eqref{eq:lemma14_induction_goal_u} and \eqref{eq:lemma14_induction_goal_p} holds for $n=0,i=0$.
Similar to \eqref{eq:lemma14_initial_score_jacobian}, we also have
\begin{equation}
\left\|
\frac{
\partial s_{\tau_{t,0}}^\star
\left(
\lambda x_{\tau_{t,0}}^{(1)}
+
(1-\lambda)x_{\tau_{t,0}}^\star
\right)
}{\partial x}
\right\|
\lesssim
\frac{d\theta_t\log T}{\tau_{t,0}},
\quad \lambda\in[0,1].
\end{equation}
Assume that for all \(0\le i\le k_0\) and for all \(\lambda\in[0,1]\),  \eqref{eq:lemma14_induction_goal_u} and
 \eqref{eq:lemma14_induction_goal_p} hold, where $0\le k_0 \le K-2$.

\paragraph{Prove \eqref{eq:lemma14_induction_goal_u} for $n=0$ and $i=k_0+1$.}
By \eqref{u_n}, \eqref{u_star},  and the definition of $A_{ji}^{(t)}=\frac{1}{2}\gamma_{t,j}\left(\tau_{t,i}\right)
\left(1-\tau_{t,j}\right)^{-3/2}$,  we have
\[
u_{t,k_0+1}^{(1)}-u_{t,k_0+1}^\star
=
J_{1,k_0+1}^{(1)}
+
J_{2,k_0+1}^{(1)}
+
J_{3,k_0+1}^{(1)},
\]
where
\begin{align}
J_{1,k_0+1}^{(1)}
&=
\sqrt{1-\tau_{t,0}}
\Biggl[
\sum_{j=0}^{k_0}
A_{j,k_0+1}^{(t)}
\Bigl(
s_{\tau_{t,j}}\!\left(x_{\tau_{t,j}}^{(1)}\right)
-
s_{\tau_{t,j}}^\star\!\left(x_{\tau_{t,j}}^{(1)}\right)
\Bigr)
\nonumber\\
&\hspace{3.5cm}
+
\sum_{j=k_0+1}^{K-1}
A_{j,k_0+1}^{(t)}
\Bigl(
s_{\tau_{t,j}}\!\left(x_{\tau_{t,j}}^{(0)}\right)
-
s_{\tau_{t,j}}^\star\!\left(x_{\tau_{t,j}}^{(0)}\right)
\Bigr)
\Biggr],
\nonumber\\
J_{2,k_0+1}^{(1)}
&=
\sqrt{1-\tau_{t,0}}
\Biggl[
\sum_{j=0}^{k_0}
A_{j,k_0+1}^{(t)}
\Bigl(
s_{\tau_{t,j}}^\star\!\left(x_{\tau_{t,j}}^{(1)}\right)
-
s_{\tau_{t,j}}^\star\!\left(x_{\tau_{t,j}}^\star\right)
\Bigr)
\nonumber\\
&\hspace{3.5cm}
+
\sum_{j=k_0+1}^{K-1}
A_{j,k_0+1}^{(t)}
\Bigl(
s_{\tau_{t,j}}^\star\!\left(x_{\tau_{t,j}}^{(0)}\right)
-
s_{\tau_{t,j}}^\star\!\left(x_{\tau_{t,j}}^\star\right)
\Bigr)
\Biggr],
\nonumber\\
J_{3,k_0+1}^{(1)}
&=
\sqrt{1-\tau_{t,0}}
\Biggl[
\sum_{j=0}^{K-1}
A_{j,k_0+1}^{(t)}
s_{\tau_{t,j}}^\star\!\left(x_{\tau_{t,j}}^\star\right)
-
\int_{\tau_{t,k_0+1}}^{\tau_{t,0}}
\frac{1}{2(1-\tau)^{3/2}}
s_\tau^\star(x_\tau^\star)\,d\tau
\Biggr].
\nonumber
\end{align}
For the first term $J_{1, k_0+1}^{(1)}$,  the Cauchy--Schwarz yields
\begin{align} \label{eq:lemma14_J1}
\|J_{1,k_0+1}^{(1)}\|_2
&\lesssim
\sqrt{1-\tau_{t,0}}
\left(
\sum_{j=0}^{K-1}
|A_{j,k_0+1}^{(t)}|^2
\right)^{1/2}
\left(
\sum_{m=0}^{1}
\sum_{j=0}^{K-1}
\left(
\varepsilon_{\mathrm{score},t,j}^{(m)}
\left(x_{\tau_{t,j}}^{(m)}\right)
\right)^2
\right)^{1/2}
\nonumber\\
&\lesssim
\frac{\tau_{t,0}\log T}{T}
\sqrt{\log K\sum_{m=0}^{N}\sum_{j=0}^{K-1}\left(\varepsilon_{\mathrm{score}, t, j}^{(m)}\left(x_{\tau_{t, j}}^{(m)}\right)\right)^2}\,
,
\end{align}
where the last line follows from \eqref{eq:23k} and \eqref{eq:23f} by setting $i_1=0$, $i_2=k_0+1$, $i_3=0$ and $i_4=0$. For the second term $J_{2, k_0+1}^{(1)}$ term, since \eqref{eq:lemma14_induction_goal_p} holds for all \(0\le i\le k_0\), similar to \eqref{eq:lemma14_initial_score_jacobian},  we  have
\begin{align} \label{condition_s}
    \left\|\frac{\partial s_{\tau_{t, i}}^{\star}\left(\lambda x_{\tau_{t, i}}^{(1)}+(1-\lambda) x_{\tau_{t, i}}^{\star}\right)}{\partial x}\right\| \lesssim \frac{d \theta_t \log T}{\tau_{t, i}}, \quad \lambda \in[0,1] 
\end{align}
for all \(0\le i\le k_0\). Using \eqref{condition_s} and \eqref{eq:lemma14_initial_score_jacobian} we obtain
\begin{align}
\|J_{2,k_0+1}^{(1)}\|_2
&\lesssim
\sqrt{1-\tau_{t,0}}
\sum_{j=0}^{k_0}
|A_{j,k_0+1}^{(t)}|
\frac{d\theta_t\log T}{\tau_{t,j}}
\left\|x_{\tau_{t,j}}^{(1)}-x_{\tau_{t,j}}^\star\right\|_2
\nonumber\\
&\quad+
\sqrt{1-\tau_{t,0}}
\sum_{j=k_0+1}^{K-1}
|A_{j,k_0+1}^{(t)}|
\frac{d\theta_t\log T}{\tau_{t,j}}
\left\|x_{\tau_{t,j}}^{(0)}-x_{\tau_{t,j}}^\star\right\|_2
\nonumber\\
&\lesssim
\frac{d\theta_t\log^2T\,\log K}{T}
\Biggl[
\max_{0\le j\le k_0}
\left\|u_{t,j}^{(1)}-u_{t,j}^\star\right\|_2
+
\max_{k_0+1\le j\le K-1}
\left\|x_{\tau_{t,j}}^{(0)}-x_{\tau_{t,j}}^\star\right\|_2
\Biggr],
\label{eq:lemma14_J2}
\end{align}
where the last line follows from $x_{\tau_{t,j}}^{(m)}-x_{\tau_{t,j}}^\star
=
\sqrt{\frac{1-\tau_{t,j}}{1-\tau_{t,0}}}
\left(
u_{t,j}^{(m)}-u_{t,j}^\star
\right)$, $1 \le m \le N$ and \eqref{eq:23j} with the fact $\tau_{t,K-1}\le \tau_{t,j}$ ,$0\le j\le K-1$, \eqref{eq:23f} by setting $i_1=0$, $i_2=k_0+1$, $i_3=K-1$ and $i_4=0$.
The first term in \eqref{eq:lemma14_J2} is bounded by assumption \eqref{eq:lemma14_induction_goal_u} holds for all $0\le i \le k_0$. For the second term in \eqref{eq:lemma14_J2}, the same estimate
as in \cite[(90)]{li2025faster} gives
\[
\max_{k_0+1\le j\le K-1}
\left\|x_{\tau_{t,j}}^{(0)}-x_{\tau_{t,j}}^\star\right\|_2
\lesssim
\frac{\tau_{t,0}-\tau_{t,K-1}}{\sqrt{\tau_{t,0}}}
\sqrt{d\theta_t\log T}.
\]
Then we obtain
\begin{align} \label{eq:lemma14_J2_last}
\|J_{2,k_0+1}^{(1)}\|_2
&\lesssim
\frac{d\theta_t\log^2T\,\log K}{T}
\Biggl(
\begin{aligned}[t]
&
C_{10}\frac{\tau_{t,0}\log T}{T}
\sqrt{\log K\sum_{m=0}^{N}\sum_{j=0}^{K-1}\left(\varepsilon_{\mathrm{score}, t, j}^{(m)}\left(x_{\tau_{t, j}}^{(m)}\right)\right)^2}\,
\\
&+
C_{10}\sqrt{\frac{d\theta_t\tau_{t,0}\log^3T}{T^2}}
\left(
\frac{d\theta_t\log^2TK}{T}
\right)
+
\frac{\tau_{t,0}-\tau_{t,K-1}}{\sqrt{\tau_{t,0}}}
\sqrt{d\theta_t\log T}\Biggr)
\end{aligned}
\nonumber\\
&\lesssim
\frac{\tau_{t,0}\log T}{T}
\sqrt{\log K\sum_{m=0}^{N}\sum_{j=0}^{K-1}\left(\varepsilon_{\mathrm{score}, t, j}^{(m)}\left(x_{\tau_{t, j}}^{(m)}\right)\right)^2}\,
+
\sqrt{\frac{d\theta_t\tau_{t,0}\log^3T}{T^2}}
\left(
\frac{d\theta_t\log^2TK}{T}
\right).
\end{align}
Here the last inequality follows from $T\ge C_2d\log^4T$ and $\frac{\tau_{t, 0}-\tau_{t, K-1}}{\sqrt{\tau_{t, 0}}} \sqrt{d \theta_t \log T}\le\frac{(\tau_{t, 0}-\tau_{t, K-1})}{(1-\tau_{t,K-1})\tau_{t,0}} \sqrt{d \tau_{t,0}\theta_t \log T}$ with \eqref{eq:23f} in Lemma~\ref{lemma9}. Similar to Bound for $G_2$ \eqref{eq:exact_flow_integral_cheb} in Section~\ref{lemma3_proof}, $J_{3, k_0+1}^{(1)}$ is easily bounded:
\begin{equation} \label{eq:lemma14_J3_last}
\|J_{3,k_0+1}^{(1)}\|_2
\lesssim
\sqrt{\frac{d\theta_t\tau_{t,0}\log^3T}{T^2}}
\left(
\frac{M_0K d\theta_t\log^2T}{T}
\right)^K.
\end{equation}
Combining
\eqref{eq:lemma14_J1},
\eqref{eq:lemma14_J2_last}, and
\eqref{eq:lemma14_J3_last}, for sufficiently large \(T\) and large enough $C_{10}$ we get
\begin{equation} \label{eq:lemma14_u_last_node}
\left\|u_{t,k_0+1}^{(1)}-u_{t,k_0+1}^\star\right\|_2
\le
C_{10}
\frac{\tau_{t,0}\log T}{T}
\sqrt{\log K\sum_{m=0}^{N}\sum_{j=0}^{K-1}\left(\varepsilon_{\mathrm{score}, t, j}^{(m)}\left(x_{\tau_{t, j}}^{(m)}\right)\right)^2}\,
+
C_{10}
\sqrt{\frac{d\theta_t\tau_{t,0}\log^3T}{T^2}}
\left(
\frac{d\theta_t\log^2TK}{T}
\right).
\end{equation}
By induction,  we prove  \eqref{eq:lemma14_induction_goal_u}  for any $0\le i \le K-1$ at $n=0$.

\paragraph{Prove \eqref{eq:lemma14_induction_goal_p} for $n=0$ and $i=k_0+1$.}
We next derive the bound of $-\log p_{\bar{X}_{\tau_{t, k_0+1}}}\left(\lambda x_{\tau_{t, k_0+1}}^{(1)}+(1-\lambda) x_{\tau_{t, k_0+1}}^{\star}\right)$ from \eqref{eq:lemma14_u_last_node}.  We first claim that for $i=k_0+1$ at $n=0$, it holds by large enough $C_{10}$
\begin{align}
  \label{eq:lemma14_claim}
&\left|
\log
\frac{
p_{\sqrt{\frac{1-\tau_{t,0}}{1-\tau_{t,k_0+1}}}\bar X_{\tau_{t,k_0+1}}}
\left(
\sqrt{\frac{1-\tau_{t,0}}{1-\tau_{t,k_0+1}}}
\bigl(
\lambda x_{\tau_{t,k_0+1}}^\star
+
(1-\lambda)x_{\tau_{t,k_0+1}}^{(1)}
\bigr)
\right)
}{
p_{\bar X_{\tau_{t,0}}}(x_{\tau_{t,0}})
}
\right| \notag
\\  & \qquad\le
C_{10}
\dkh{
\frac{d\theta_t\log^2T}{T}
+
\frac{\sqrt{d\theta_t\log^3T\,\log K\sum_{m,j}\left(\varepsilon_{\mathrm{score}, t, j}^{(m)}\left(x_{\tau_{t, j}}^{(m)}\right)\right)^2}}{T}
},  
\end{align} 
By the affine change of variables,
\begin{align*}
p_{\sqrt{\frac{1-\tau_{t,0}}{1-\tau_{t,k_0+1}}}\bar X_{\tau_{t,k_0+1}}}&
\left(
\sqrt{\frac{1-\tau_{t,0}}{1-\tau_{t,k_0+1}}}(\lambda x_{\tau_{t,k_0+1}}^\star
+
(1-\lambda)x_{\tau_{t,k_0+1}}^{(1)})
\right)\\
=&
\left(
\frac{1-\tau_{t,0}}{1-\tau_{t,k_0+1}}
\right)^{-d/2}
p_{\bar X_{\tau_{t,k_0+1}}}(\lambda x_{\tau_{t,k_0+1}}^\star
+
(1-\lambda)x_{\tau_{t,k_0+1}}^{(1)}).
\end{align*}
Therefore,
\begin{align} \label{eq:p_last}
&-\log p_{\bar X_{\tau_{t,k_0+1}}}(\lambda x_{\tau_{t,k_0+1}}^\star
+
(1-\lambda)x_{\tau_{t,k_0+1}}^{(1)})\nonumber\\
&=
-\log
p_{\sqrt{\frac{1-\tau_{t,0}}{1-\tau_{t,k_0+1}}}\bar X_{\tau_{t,k_0+1}}}
\left(
\sqrt{\frac{1-\tau_{t,0}}{1-\tau_{t,k_0+1}}}(\lambda x_{\tau_{t,k_0+1}}^\star
+
(1-\lambda)x_{\tau_{t,k_0+1}}^{(1)})
\right)
-
\frac d2
\log
\left(
\frac{1-\tau_{t,0}}{1-\tau_{t,k_0+1}}
\right)
\nonumber\\
&\le
-\log p_{\bar X_{\tau_{t,0}}}(x_{\tau_{t,0}})
+
C_{10}
\left\{
\frac{d\theta_t\log^2T}{T}
+
\frac{\sqrt{d\theta_t\log^3T\,\log K\sum_{m,j}\left(\varepsilon_{\mathrm{score}, t, j}^{(m)}\left(x_{\tau_{t, j}}^{(m)}\right)\right)^2}}{T}
\right\}
+
c_1\frac{d\log T}{T}\nonumber\\
&\le
2.1d\theta_t\log T,
\end{align}
where the third line follows from $-\frac{d}{2} \log \left(\frac{1-\tau_{t, 0}}{1-\tau_{t, k_0+1}}\right)\le -\frac{d}{2} \log \left(\frac{1-\tau_{t, 0}}{1-\tau_{t, K-1}}\right)=-\frac{d}{2} \log \left(\frac{\bar{\alpha}_t}{\bar{\alpha}_{t-1}}\right)=-\frac{d}{2} \log \left(\alpha_t\right)=\frac{d}{2}\int_0^{\frac{c_1  \log T}{T}} \frac{1}{1-u} d u \leq \frac{d}{2} \int_0^{\frac{c_1  \log T}{T}} 2 d u\leq \frac{d}{2} \cdot \frac{2 c_1 \log T}{T}=\frac{c_1 d \log T}{T}$ using \eqref{eq:23a} in Lemma~\ref{lemma9}
and the last line holds by $-\log p_{\bar{X}_{\tau_{t, 0}}}\left(x_{\tau_{t, 0}}\right)\le d\theta_t\log T$ with \eqref{assu_1} and condition \eqref{eq:lemma14_condition_here} in Lemma~\ref{lemma14} for sufficiently large \(T\).  This proves that \eqref{eq:lemma14_induction_goal_p} holds for any $0 \le i \le K-1$ at $n=0$.

Repeat the same arguments above at $n=1,\cdots,N-1$, we can prove \eqref{eq:lemma14_induction_goal_u} and \eqref{eq:lemma14_induction_goal_p} hold for all $0\le n \le N-1$ and $0\le i \le K-1$.
Moreover,  we have
\begin{equation}
\log
\frac{
p_{\sqrt{\frac{1-\tau_{t,0}}{1-\tau_{t,i}}}\bar X_{\tau_{t,i}}}
\left(
\sqrt{\frac{1-\tau_{t,0}}{1-\tau_{t,i}}}
x_{\tau_{t,i}}^{(n+1)}
\right)
}{
p_{\bar X_{\tau_{t,0}}}(x_{\tau_{t,0}})
}
\le
\frac{4c_1d\log T}{T}
+
C_{10}
\left\{
\frac{d^2\theta_t^2\log^4TK}{T^2}
+
\frac{\sqrt{d\theta_t\log^3T\,\log K\sum_{m,j}\left(\varepsilon_{\mathrm{score}, t, j}^{(m)}\left(x_{\tau_{t, j}}^{(m)}\right)\right)^2}}{T}
\right\},
\label{eq:lemma14_claim_inter}
\end{equation}
holds for all $0\le n \le N-1$ and $0\le i \le K-1$. This gives the conclusion.  It remains to prove claim \eqref{eq:lemma14_claim}.

\paragraph{Proof of claim \eqref{eq:lemma14_claim}}
As in \cite[Eq. (95)--(96)]{li2025faster}, one has
\begin{equation}
\|u_{t,k_0+1}^\star\|_2
\le
C\sqrt{\frac{d\theta_t\tau_{t,0}\log^3T}{T^2}}.
\label{eq:lemma14_ustar}
\end{equation}
Therefore, by \eqref{eq:lemma14_u_last_node} and \eqref{eq:lemma14_ustar},
\begin{equation}
\|u_{t,k_0+1}^{(1)}\|_2
\le
C_{10}
\left\{
\sqrt{\frac{d\theta_t\tau_{t,0}\log^3T}{T^2}}
+
\frac{\tau_{t,0}\log T}{T}
\sqrt{\log K\sum_{m=0}^{N}\sum_{j=0}^{K-1}\left(\varepsilon_{\mathrm{score}, t, j}^{(m)}\left(x_{\tau_{t, j}}^{(m)}\right)\right)^2}\,
\right\}.
\label{eq:lemma14_u1}
\end{equation}

As in \cite[Eq. (93)]{li2025faster}, we have
\begin{align}
&
\frac{
p_{\sqrt{\frac{1-\tau_{t,0}}{1-\tau_{t,k_0+1}}}\bar X_{\tau_{t,k_0+1}}}
\left(
\sqrt{\frac{1-\tau_{t,0}}{1-\tau_{t,k_0+1}}}
x_{\tau_{t,k_0+1}}^{(1)}
\right)
}{
p_{\bar X_{\tau_{t,0}}}(x_{\tau_{t,0}})
}
\nonumber\\
&=
\left(
1+\frac d2
\frac{\tau_{t,0}-\tau_{t,k_0+1}}
{(1-\tau_{t,0})\tau_{t,k_0+1}}
+
O\left(
d^2
\left(
\frac{\tau_{t,0}-\tau_{t,k_0+1}}
{(1-\tau_{t,0})\tau_{t,k_0+1}}
\right)^2
\right)
\right)
\int_{x_0}
p_{X_0\mid \bar X_{\tau_{t,0}}}
(x_0\mid x_{\tau_{t,0}})
\nonumber\\
&\quad\cdot
\exp\left(
-\frac{
(\tau_{t,0}-\tau_{t,k_0+1})
\|x_{\tau_{t,0}}-\sqrt{1-\tau_{t,0}}x_0\|_2^2
}{
2(1-\tau_{t,0})\tau_{t,0}\tau_{t,k_0+1}
}
-
\frac{
\|u_{t,k_0+1}^{(1)}\|_2^2
+
2\left\langle
u_{t,k_0+1}^{(1)},
x_{\tau_{t,0}}-\sqrt{1-\tau_{t,0}}x_0
\right\rangle
}{
2\frac{(1-\tau_{t,0})\tau_{t,k_0+1}}{1-\tau_{t,k_0+1}}
}
\right)
\,dx_0.
\label{eq:lemma14_bayes_last_node}
\end{align}

Denote for \(\ell=1,2,\ldots\),
\[
E_\ell^{\mathrm{typical}}
:=
\left\{
x_0:
\left\|x_{\tau_{t,0}}-\sqrt{1-\tau_{t,0}}x_0\right\|_2
\le
5\ell\sqrt{d\theta_t\tau_{t,0}\log T}
\right\}.
\]
Then for any \(x_0\in E_\ell^{\mathrm{typical}}\), similar to \cite[Eq. (98)]{li2025faster}, by \eqref{eq:23f} and
\eqref{eq:23h},
\begin{equation}
\frac{
(\tau_{t,0}-\tau_{t,k_0+1})
\left\|x_{\tau_{t,0}}-\sqrt{1-\tau_{t,0}}x_0\right\|_2^2
}{
2(1-\tau_{t,0})\tau_{t,0}\tau_{t,k_0+1}
}
\lesssim
\ell^2
\frac{d\theta_t\log^2T}{T}.
\label{eq:lemma14_shell1_last_node}
\end{equation}
Also, using \eqref{eq:lemma14_u1}, \eqref{eq:23h}, and
$\left\|x_{\tau_{t,0}}-\sqrt{1-\tau_{t,0}}x_0\right\|_2
\le
5\ell\sqrt{d\theta_t\tau_{t,0}\log T}$,
similar to \cite[Eq. (99)]{li2025faster} we have
\begin{align}
&
\left|
\frac{
\|u_{t,k_0+1}^{(1)}\|_2^2
+
2\left\langle
u_{t,k_0+1}^{(1)},
x_{\tau_{t,0}}-\sqrt{1-\tau_{t,0}}x_0
\right\rangle
}{
2\frac{(1-\tau_{t,0})\tau_{t,k_0+1}}{1-\tau_{t,k_0+1}}
}
\right|
\nonumber\\
&\lesssim
\ell
\frac{d\theta_t\log^2T}{T}
+
\ell
\frac{\sqrt{d\theta_t\log^3T\,\log K\sum_{m=0}^{N}\sum_{j=0}^{K-1}\left(\varepsilon_{\mathrm{score}, t, j}^{(m)}\left(x_{\tau_{t, j}}^{(m)}\right)\right)^2}\,}{T}.
\label{eq:lemma14_shell2_last_node}
\end{align}
Combining \eqref{eq:lemma14_u_last_node}, \eqref{eq:lemma14_bayes_last_node},
\eqref{eq:lemma14_shell1_last_node}, and
\eqref{eq:lemma14_shell2_last_node}, and then repeating the same shell
decomposition and Jensen argument as in \cite[Eq. (100)--(106)]{li2025faster},
together with Lemma~\ref{lemma13} and \eqref{eq:23f}, yields
\begin{align}
  \label{eq:lemma14_density_last_node_first_sweep}
&\log
\frac{
p_{\sqrt{\frac{1-\tau_{t,0}}{1-\tau_{t,k_0+1}}}\bar X_{\tau_{t,k_0+1}}}
\left(
\sqrt{\frac{1-\tau_{t,0}}{1-\tau_{t,k_0+1}}}
x_{\tau_{t,k_0+1}}^{(1)}
\right)
}{
p_{\bar X_{\tau_{t,0}}}(x_{\tau_{t,0}})
} \notag \\
&\qquad \le
\frac{4c_1d\log T}{T}
+
C_{10}
\left\{
\frac{d^2\theta_t^2\log^4 T\,K}{T^2}
+
\frac{
\sqrt{
d\theta_t\log^3 T\,\log K
\sum_{m,j}
\left(
\varepsilon_{\mathrm{score},t,j}^{(m)}
\left(x_{\tau_{t,j}}^{(m)}\right)
\right)^2
}
}{T}
\right\}.
\end{align}

Next, fix any \(\lambda \in [0,1]\) and define
\[
\widetilde x_{\tau_{t,k_0+1}}^{(1)}(\lambda)
:=
\lambda x_{\tau_{t,k_0+1}}^\star
+
(1-\lambda)x_{\tau_{t,k_0+1}}^{(1)},
\]
and
\[
\widetilde u_{t,k_0+1}^{(1)}(\lambda)
:=
\sqrt{\frac{1-\tau_{t,0}}{1-\tau_{t,k_0+1}}}\,
\widetilde x_{\tau_{t,k_0+1}}^{(1)}(\lambda)
-
x_{\tau_{t,0}}.
\]
Then
\[
\widetilde u_{t,k_0+1}^{(1)}(\lambda)
=
\lambda u_{t,k_0+1}^\star
+
(1-\lambda)u_{t,k_0+1}^{(1)}.
\]
Consequently,
\[
\left\|\widetilde u_{t,k_0+1}^{(1)}(\lambda)\right\|_2
\le
\left\|u_{t,k_0+1}^\star\right\|_2
+
\left\|u_{t,k_0+1}^{(1)}-u_{t,k_0+1}^\star\right\|_2.
\]
It follows that the same bounds
\eqref{eq:lemma14_shell1_last_node} and
\eqref{eq:lemma14_shell2_last_node} remain valid with
\(u_{t,k_0+1}^{(1)}\) replaced by
\(\widetilde u_{t,k_0+1}^{(1)}(\lambda)\). Repeating the argument used to
derive \eqref{eq:lemma14_density_last_node_first_sweep}, we obtain
\begin{align}
&\left|
\log
\frac{
p_{\sqrt{\frac{1-\tau_{t,0}}{1-\tau_{t,k_0+1}}}\bar X_{\tau_{t,k_0+1}}}
\left(
\sqrt{\frac{1-\tau_{t,0}}{1-\tau_{t,k_0+1}}}
\bigl(
\lambda x_{\tau_{t,k_0+1}}^\star
+
(1-\lambda)x_{\tau_{t,k_0+1}}^{(1)}
\bigr)
\right)
}{
p_{\bar X_{\tau_{t,0}}}(x_{\tau_{t,0}})
}
\right| \notag \\
&\qquad \le
C_{10}
\left\{
\frac{d\theta_t\log^2 T}{T}
+
\frac{
\sqrt{
d\theta_t\log^3 T\,\log K
\sum_{m,j}
\left(
\varepsilon_{\mathrm{score},t,j}^{(m)}
\left(x_{\tau_{t,j}}^{(m)}\right)
\right)^2
}
}{T}
\right\}.
\label{eq:lemma14_segment_last_node_first_sweep}
\end{align}
This proves the claim \eqref{eq:lemma14_claim}.
\end{proof}

\bibliographystyle{unsrt}  

\end{document}